%% file: LZ_2015-preprint.tex
\documentclass[12pt,reqno]{amsart}
\usepackage{amssymb,color,mathrsfs}
\usepackage[english]{babel}
\usepackage{ifthen}
\newboolean{INSDEL}

\input cd
\newlength{\mylen}

\def\dela#1{\ifmmode
\text{\kern1pt\mbox{$#1$}\kern1pt\settowidth{\mylen}{$#1$}%
\addtolength{\mylen}{0.7pt}\capt{-\the\mylen}{0pt}%
{\kern-1pt\rule[-2pt]{0.5pt}{9pt}\rule[2pt]{\mylen}{0.5pt}%
\kern-0.5pt\rule[-2pt]{0.5pt}{9pt}\capt{-1.1pt}{4.6pt}{$_\wr$}}\kern0.8pt}%
\else
\kern1pt\mbox{#1\/}\kern1pt\settowidth{\mylen}{#1\/}%
\addtolength{\mylen}{0.7pt}\capt{-\the\mylen}{0pt}%
{\kern-1pt\rule[-2pt]{0.5pt}{9pt}\rule[2pt]{\mylen}{0.5pt}%
\kern-0.5pt\rule[-2pt]{0.5pt}{9pt}\capt{-1.1pt}{4.6pt}{$_\wr$}}\kern0.8pt
\fi}

\def\insa#1{\ifmmode
\text{\boxed{\vphantom{)}{\kern-1pt}#1{\kern-1pt}}}%
\else%
\boxed{\vphantom{)}{\kern-1pt}\text{#1\/}{\kern-1pt}}%
\fi}

\def\delbreaka#1{\break\dela{#1}}

\let\ins=\insb
\let\del=\delb
\let\delbreak=\delbreakb

\def\insdel{\fboxsep=1pt
\emergencystretch=42pt
\let\ins=\insa
\let\del=\dela
\let\delbreak=\delbreaka}

\def\capt#1#2#3{\rlap{\kern #1\smash{\lower #2\hbox{#3}}}}
\def\pict#1#2#3#4#5{\rlap{\kern #1\smash{\lower #2
\hbox{\includegraphics[bb= 0 0 #3 #4]{#5}}}}}
\def\pictPS#1#2#3{\rlap{\kern #1\smash{\lower
#2\hbox{\includegraphics{#3}}}}}

\insdel

\usepackage{epsfig}
\usepackage{yfonts}
\usepackage{enumerate, amssymb}
\usepackage[colorlinks]{hyperref}
\usepackage{hyperref}
\usepackage{comment}
\newcommand{\R}{\ensuremath{\mathbb{R}}}

\newcommand{\Z}{\ensuremath{\mathbb{Z}}}
\newcommand{\N}{\ensuremath{\mathbb{N}}}

\newcommand{\p}{\partial}
\newcommand{\id}{{\rm id}}

\def\bals#1\eals{\begin{align*}#1\end{align*}}
\def\bal#1\eal{\begin{align}#1\end{align}}

\def\Aut{\mathop{\rm Aut}}
\def\SAut{\mathop{\rm SAut}}

\def\Der{\mathop{\rm Der}}
\def\Norm{\mathop{\rm Norm}}

\def\deg{\mathop{\rm deg}}

\def\reg{{\mathop{\rm reg}}}

\def\rank{\mathop{\rm rank}}

\def\ML{\mathop{\rm ML}}

\def\lto{\longrightarrow}

\def\and{\quad\mbox{and}\quad}
\def\discr{\mathop{\rm discr}}

\def\Aff{\mathop{\rm Aff}}

\renewcommand{\epsilon}{\varepsilon}
\renewcommand{\phi}{\varphi}

\newcommand{\bnum}{\begin{itemize}}
\newcommand{\enum}{\end{itemize}}
\renewcommand{\emptyset}{\varnothing}
\newtheorem{thm}{Theorem}[section]
\newtheorem{cor}[thm]{Corollary}
\newtheorem{lem}[thm]{Lemma}
\newtheorem{prop}[thm]{Proposition}
\newtheorem{prob}[thm]{Problem}

\theoremstyle{definition}
\newtheorem{defi}[thm]{Definition}
\newtheorem{defis}[thm]{Definitions}
\newtheorem{conj}[thm]{Problem}
\newtheorem{conv}[thm]{Convention}
\newtheorem{nota}[thm]{Notation}
\newtheorem{rem}[thm]{Remark}
\newtheorem{rems}[thm]{Remarks}
\newtheorem{rem-def}[thm]{Remark-Definition}
\newtheorem{exa}[thm]{Example}
\newtheorem{exas}[thm]{Examples}
\newtheorem{exa-def}[thm]{Example-Definition}
\newtheorem{sit}[thm]{}

\newcommand{\brem}{\begin{rem}}
\newcommand{\brems}{\begin{rems}}
\newcommand{\erem}{\end{rem}}
\newcommand{\erems}{\end{rems}}
\newcommand{\bprob}{\begin{prob}}
\newcommand{\eprob}{\end{prob}}
\newcommand{\bprobs}{\begin{probs}}
\newcommand{\eprobs}{\end{probs}}
\newcommand{\bques}{\begin{ques}}
\newcommand{\eques}{\end{ques}}
\newcommand{\bexa}{\begin{exa}}
\newcommand{\bexas}{\begin{exas}}
\newcommand{\eexa}{\end{exa}}
\newcommand{\eexas}{\end{exas}}
\newcommand{\bdefi}{\begin{defi}}
\newcommand{\edefi}{\end{defi}}
\newcommand{\bdefis}{\begin{defis}}
\newcommand{\edefis}{\end{defis}}
\newcommand{\bcor}{\begin{cor}}
\newcommand{\ecor}{\end{cor}}
\newcommand{\blem}{\begin{lem}}
\newcommand{\elem}{\end{lem}}
\newcommand{\bconv}{\begin{conv}}
\newcommand{\econv}{\end{conv}}
\newcommand{\bconj}{\begin{conj}}
\newcommand{\econj}{\end{conj}}
\newcommand{\bprop}{\begin{prop}}
\newcommand{\eprop}{\end{prop}}
\newcommand{\bthm}{\begin{thm}}
\newcommand{\ethm}{\end{thm}}
\newcommand{\bnota}{\begin{nota}}
\newcommand{\enota}{\end{nota}}
\newcommand{\bsit}{\begin{sit}}
\newcommand{\esit}{\end{sit}}
\newcommand{\be}{\begin{equation}}
\newcommand{\ee}{\end{equation}}
\newcommand{\bproof}{\begin{proof}}
\newcommand{\eproof}{\end{proof}}
\def\ba{\begin{array}}
\def\ea{\end{array}}

\DeclareMathOperator{\bc}{bc} 
\DeclareMathOperator{\const}{const}

\DeclareMathOperator{\End}{End} 
\DeclareMathOperator{\hol}{hol} \DeclareMathOperator{\Hom}{Hom}
 
 \DeclareMathOperator{\Mor}{Mor}

\def\Lie{\mathop{\rm Lie}}
\def\LND{\mathop{\rm LND}}
\def\sing{\mathop{\rm sing}}

 \DeclareMathOperator{\Sym}{Sym}

\DeclareMathOperator{\Tors}{Tors} 

\DeclareMathOperator{\ord}{ord} \DeclareMathOperator{\blc}{blc}

\newcommand{\Def}{\stackrel{\mathrm{def}}{=\!\!=}}

\thanks{A part of this paper was done during the stay of the second author at the Max Planck 
Institute for Mathematics in Bonn, Germany and in the Technion, Haifa, Israel. The second author thanks these institutions for their  hospitality, 
excellent working conditions, and support.}

\begin{document}

\title[Automorphism groups of configuration spaces]{Automorphism groups of configuration spaces and discriminant varieties}

\author{Vladimir Lin and Mikhail Zaidenberg}
\address{Technion-Israel institute of Technology, Haifa 32000, Israel}
\email{vlin{\char064}tx.technion.ac.il}
\address{Universit\'e Grenoble I, Institut Fourier, UMR 5582
CNRS-UJF, BP 74, 38402 Saint Martin d'H\`eres c\'edex, France}
\email{Mikhail.Zaidenberg{\char064}ujf-grenoble.fr}

\begin{abstract} The configuration space $\mathcal{C}^n(X)$
of an algebraic curve $X$ is the algebraic variety consisting of
all $n$-point subsets $Q\subset X$. We describe the
automorphisms of $\mathcal{C}^n(\mathbb{C})$, deduce that the
(infinite dimensional) group $\Aut\mathcal{C}^n(\mathbb{C})$ is
solvable, and obtain an analog of the Mostow decomposition in this
group.  The Lie algebra and the Makar-Limanov invariant of
$\mathcal{C}^n(\mathbb{C})$ are also computed. We obtain similar
results for the  level hypersurfaces of the discriminant,
including its singular zero level.

This is an extended version of our paper \cite{Lin-Zaidenberg14}. 
We strengthened the results concerning the automorphism groups of cylinders 
over rigid bases, replacing the rigidity assumption by the weaker assumption 
of tightness. 
We also added 
alternative proofs of  two auxiliary results cited in \cite{Lin-Zaidenberg14} 
and due to  Zinde and to the first author.  This allowed us to provide 
the optimal dimension bounds in our theorems. \end{abstract}

\maketitle

\thanks{
{\renewcommand{\thefootnote}{} \footnotetext{ 2010
\textit{Mathematics Subject Classification:}
14R20,\,32M17.\mbox{\hspace{11pt}}\\{\it Key words}: affine
variety, automorphism group, configuration space, cylinder,  discriminant.}}

\vfuzz=2pt

\section{Introduction}\label{Sec: Introduction}
Let $X$ be an irreducible smooth algebraic curve over the field
$\mathbb{C}$. The $n$th {\em configuration space}
$\mathcal{C}^n(X)$ of $X$ is a smooth affine algebraic variety of dimension
$n$ consisting of all $n$-point subsets
$Q=\{q_1,\ldots,q_n\}\subset X$ with distinct $q_1,\ldots,q_n$.
We would like to study its biregular automorphisms and the
structure of the group $\Aut\mathcal{C}^n(X)$.

For a hyperbolic curve $X$ the group $\Aut\mathcal{C}^n(X)$ is finite
(possibly, trivial for a generic curve). We are interested in the case where $X$ is non-hyperbolic, i.e., 
one of the curves $\mathbb{C}$, ${\mathbb{P}}^1={\mathbb{P}}_{\mathbb{C}}^1$,
$\mathbb{C}^*=\mathbb{C}\setminus\{0\}$, and an elliptic curve. In the latter two cases the
 groups $\Aut\mathcal{C}^n(X)$ were described in \cite{Zin78} and \cite{FelTo07}, respectively. 
Here we investigate automorphisms of the configuration space 
$\mathcal{C}^n=\mathcal{C}^n(\mathbb{C})$ and of some related spaces.
Our central results are Theorems \ref{Thm: 1st main theorem in the Introduction} and 
\ref{Thm: 2d main theorem in the Introduction}; see also Theorems \ref{Thm: Mostow}, 
\ref{Thm: structure}, \ref{Thm: presentation}, and Corollary \ref{cor: neutral comp}. 
Let us first introduce some conventions and a portion of notation.  

All varieties in this paper are algebraic varieties defined over\footnote{However, 
all general results remain valid over any algebraically closed field of characteristic zero.} 
$\mathbb{C}$ 
and reduced; in general, irreducibility is not required. {\em Morphism} means a regular morphism of varieties. 
The same applies to the terms {\em automorphism} and {\em endomorphism}. 
The actions of algebraic groups are assumed to be regular. We use the standard notation 
$\mathcal{O}(\mathcal{Z})$, $\mathcal{O}_+(\mathcal{Z})$, and $\mathcal{O}^\times(\mathcal{Z})$
 for the algebra of all regular functions on a variety $\mathcal{Z}$, the additive group of this algebra,
 and its group of invertible elements, respectively.

For $z\in\mathbb{C}^n$, let $d_n(z)$ denote the discriminant of the monic polynomial
%&
\begin{equation}\label{eq: universal polynomial}
P_n(\lambda,z)=\lambda^n+z_1 \lambda^{n-1}+\ldots+z_n \,, \ \ z
=(z_1,...,z_n)\in\mathbb{C}^n=\mathbb{C}^n_{(z)}\,.
\end{equation}
%&
If $d_n(z)\ne{0}$ and $Q\subset\mathbb{C}$ is the set of all roots of $P_n(\cdot,z)$,
 then $Q\in\mathcal{C}^n$ and
%&
\begin{equation}\label{eq: discriminant of Q}
D_n(Q)\Def\prod\limits_{\{q',q''\}\subset{Q}}(q'-q'')^2
=d_n(z)\,.
\end{equation}
%&
Denoting by $\mathcal{P}^n$ the space of all polynomials (\ref{eq: universal polynomial}) 
with simple roots, we have the natural identification
%&
\begin{equation}\label{eq: Cn(C)=Cn-Sigma}
\mathcal{C}^n=\{Q\subset\mathbb{C}\mid\,\#Q=n\}
\cong\mathcal{P}^n=\mathbb{C}^n_{(z)}\setminus\Sigma^{n-1}\,, \ \
Q\leftrightarrow{z=(z_1,...,z_n)}\,,
\end{equation}
%&
where\,\footnote{The upper index will usually mean the dimension of the variety.} 
the {\em discriminant variety} $\Sigma^{n-1}$ is defined by
%&
\begin{equation}\label{eq: discriminant variety}
\Sigma^{n-1}\Def\{z\in\mathbb{C}^n\mid\,d_n(z)=0\}\,.
\end{equation}
%&
For $n>2$, we describe the automorphisms of the configuration
space $\mathcal{C}^n$, the discriminant variety
$\Sigma^{n-1}$, and the 
{\em special configuration space}
%&
\begin{equation}\label{eq: unit discriminant level}
\mathcal{SC}^{n-1}\Def\left\{Q\in\mathcal{C}^n\mid\,D_n(Q)=1\right\}
\cong\{z\in\mathbb{C}^n\mid\,d_n(z)=1\}\,.
\end{equation}
%&
This leads to structure theorems for the automorphism groups $\Aut\mathcal{C}^n$, 
$\Aut\mathcal{SC}^{n-1}$, and $\Aut\Sigma^{n-1}$.
\vskip3pt

The varieties $\mathcal{C}^n$ and $\Sigma^{n-1}$ can be viewed as
the complementary to each other parts of the symmetric power
$\Sym^n\mathbb{C}=\mathbb{C}^n/\mathbf{S}(n)$, where $\mathbf{S}(n)$
 is the symmetric group permuting the coordinates $q_1,...,q_n$ in
$\mathbb{C}^n=\mathbb{C}^n_{(q)}$. We have the natural morphisms
%&
\begin{equation}\label{eq: projection p}
p\colon\mathbb{C}_{(q)}^n\to\Sym^n\mathbb{C}\cong\mathbb{C}^n_{(z)}\,,\ \
\Delta^{n-1}\to\Sigma^{n-1}\,, \ \ \textup{and} \ \
\mathbb{C}_{(q)}^n\setminus\Delta^{n-1}\to\mathcal{C}^n\,,
\end{equation}
%&
where $\Delta^{n-1}\Def\bigcup\limits_{i\ne{j}}\{q=(q_1,...,q_n)\in\mathbb{C}_{(q)}^n\mid\,
q_i=q_j\}$} is the big diagonal.  The points $z\in\Sigma^{n-1}$ are in one-to-one 
correspondence with unordered $n$-term multisets (or corteges) 
$Q=\{q_1,\ldots,q_n\}$, $q_i\in\mathbb{C}$, 
with at least one repetition.
\vskip3pt

The {\em barycenter} $\bc(Q)$ of a point $Q\in\Sym^n\mathbb{C}=
\mathcal{C}^n\cup\Sigma^{n-1}\cong\mathbb{C}^n_{(z)}$
is defined as
%&
\begin{equation}\label{eq: barycenter}
\bc(Q)\Def\frac{1}{n}\sum\limits_{q\in{Q}}q=-z_1/n
\end{equation}
%&
(if $Q$ is a multiset, the summation takes into account multiplicities).
The {\em balanced
configuration space} $\mathcal{C}^{n-1}_{\blc}\subset\mathcal{C}^n$ is defined by
%&
\begin{equation}\label{eq: balanced spaces0}
\mathcal{C}^{n-1}_{\blc}=\{Q\in\mathcal{C}^n\mid\, \bc(Q)=0\}\,.
\end{equation}
%& 
One defines similarly  balanced hypersurfaces $\mathcal{SC}^{n-2}_{\blc}\subset\mathcal{SC}^{n-1}$ and $\Sigma^{n-2}_{\blc}\subset\Sigma^{n-1}$ (see also Notation \ref{Ntn: ordered configuration space}). 
\vskip3pt

Our main results related to automorphisms of $\mathcal{C}^n$ are
the following two theorems (for more general results see Theorems 
\ref{Thm: Mostow}, \ref{Thm: structure}, \ref{Thm: presentation}, 
and Corollary \ref{cor: neutral comp}).

\begin{thm}\label{Thm: 1st main theorem in the Introduction}
Assume that $n>2$. A map $F\colon\,\mathcal{C}^n\to\mathcal{C}^n$
is an automorphism if and only if it is of the form
%&
\begin{equation}\label{eq: formula for automorphism of Cn}
F(Q)=s\cdot\pi(Q)+A(\pi(Q))\bc(Q) \ \ \textup{for any} \
Q\in\mathcal{C}^n\,,
\end{equation}
%&
where $\pi(Q)=Q-\bc(Q)$, $s\in\mathbb{C}^*$, and
$A\colon\,\mathcal{C}^{n-1}_{\blc}\to\Aff\mathbb{C}$ is a  regular
map.
\end{thm}

\begin{thm}\label{Thm: 2d main theorem in the Introduction}
If $n>2$, then the following hold.
\vskip3pt

\textup{(a)} The group $\Aut\mathcal{C}^n$ is solvable. More precisely,
it is a semi-direct product
%&
$$
\Aut\mathcal{C}^n\cong\left(\mathcal{O}_+(\mathcal{C}^{n-1}_{\blc})
\rtimes(\mathbb{C}^*)^2\right)\rtimes\mathbb{Z}\,.
$$
%&
\vskip3pt

\textup{(b)} Any finite subgroup $\Gamma\subset\Aut\mathcal{C}^n$ is Abelian.
\vskip3pt

\textup{(c)} Any connected algebraic subgroup $G$ of $ \Aut\mathcal{C}^n$
is either Abelian or metabelian of rank $\le 2$.
 \footnote{The rank of an affine algebraic group is the dimension of its maximal tori. }
\vskip3pt

\textup{(d)} Any two maximal tori in $\Aut\mathcal{C}^n$ are conjugated.
\end{thm}

Similar facts are established for $\mathcal{SC}^{n-1}$ and $\Sigma^{n-1}$, 
see loc. cit.

Let us overview  some results of \cite{Lin72b, Kal76a, Zin78, Lin03,
Lin04b, FelTo07, Lin11} initiated the present paper and used in
the proofs. Given a smooth irreducible non-hyperbolic algebraic curve $X$, 
consider the diagonal action of the group $\Aut{X}$ on the configuration 
space $\mathcal{C}^n(X)$,
%&
\begin{equation}\label{eq: diagonal Aut X action}
\Aut{X}\ni{A}\colon\,{\mathcal{C}}^n(X)\to {\mathcal{C}}^n(X),
\quad {Q}=\{q_1,\dots,q_n\}\mapsto{A}Q\Def\{Aq_1,\dots, Aq_n\}\,.
\end{equation}
%&
To any morphism $T\colon{\mathcal{C}}^n(X)\to\Aut{X}$ we assign an endomorphism 
$F_T$ of $\mathcal{C}^n(X)$ defined by
%&
\begin{equation}\label{eq: tame endomorphism}
F_T(Q){\Def}T(Q)Q \ \ \textup{for all} \ \
Q\in{\mathcal{C}}^n(X)\,.
\end{equation}
%&
Such endomorphisms $F_T$ are called {\em tame}. A tame
endomorphism preserves each $(\Aut{X})$-orbit in
${\mathcal{C}}^n(X)$.  For automorphisms, the converse is also true: an automorphism of ${\mathcal{C}}^n(X)$ preserving  each $(\Aut{X})$-orbit is tame, see Proposition \ref{prop: weak-TMT} and the remark following this proposition.
In the general case, Tame Map Theorem below implies the following:  an endomorphism of
${\mathcal{C}}^n(X)$ whose image is not contained in a  single
$(\Aut{X})$-orbit is tame and hence preserves each
$(\Aut{X})$-orbit. If the image of $F$ is contained in a
single $(\Aut{X})$-orbit, then $F$ is called {\em orbit-like}.

The braid group of $X$, $B_n(X)=\pi_1(\mathcal{C}^n(X))$, is
non-Abelian for any $n\ge 3$. If $X=\mathbb{C}$ then
$B_n(X)=\mathbf{A}_{n-1}$ is the  Artin braid group on $n$ strands.
\footnote{In \cite{Lin-Zaidenberg14} 
we used the notation $B_n$ for the Artin braid group 
on $n$ strands. Here we prefer the notation $\mathbf{A}_{n-1}$ 
that indicates the place of  this group among the Artin-Brieskorn groups of series $A$. 
A similar notation will be applied to the Artin-Brieskorn groups 
of other series $B, D$ etc., while $\mathbf{WA}_{n-1},\,\mathbf{WB}_{n}$, 
etc.\ stands for the corresponding Coxeter group.
} An
endomorphism $F$ of $\mathcal{C}^n(X)$ is called {\em non-Abelian}
if the image of the induced endomorphism
$F_*\colon\,\pi_1(\mathcal{C}^n(X))\to\pi_1(\mathcal{C}^n(X))$ is
a non-Abelian group. Otherwise, $F$ is said to be {\em Abelian}.
Rather unexpectedly, this  evident  algebraic dichotomy gives rise
to the following analytic one.
\vskip3pt

\noindent {\bf Tame Map Theorem.}\label{tame-map-theorem} {\em Let
$X$ be a smooth irreducible non-hyperbolic algebraic curve. For
$n>4$ any non-Abelian endomorphism of $\mathcal{C}^n(X)$ is tame,
whereas any Abelian endomorphism of $\mathcal{C}^n(X)$ is
orbit-like.}
\vskip3pt

\begin{rems}\label{Rms: remarks to TMT}
(a) A proof of Tame Map Theorem for $X=\mathbb{C}$ is sketched in
\cite{Lin72b} and \cite{Lin79}; a complete proof for
$X=\mathbb{C}$ or $\mathbb{P}^1$ in the analytic
category can be found in \cite{Lin03}, \cite{Lin04b}, and
\cite{Lin11}. For $X={\mathbb{C}}^*$ the theorem is proved in
\cite{Zin78} \footnote{The complex Weyl chamber of type $B$
studied in \cite{Zin78} is isomorphic to
${\mathcal{C}}^n({\mathbb{C}}^*)$.}, and for elliptic curves in
\cite{FelTo07}. The proofs  apply mutatis mutandis in the algebraic setting.  
We use this theorem to describe automorphisms of the balanced spaces 
$\mathcal{C}^{n-1}_{\blc}$ and $\Sigma^{n-2}_{\blc}$
(see Theorems \ref{Thm: Kaliman-Lin-Zinde}(a),(c) and \ref{Thm: aut balanced Sigma});
 its analytic counterpart is involved in the proof of Theorem 
\ref{Thm: Non-Abelian endomorphisms of Cblc(n-1)}.
\vskip3pt

(b) A morphism $T\colon\mathcal{C}^n(X)\to\Aut{X}$ in the tame representation
$F=F_T$ is uniquely determined by a non-Abelian endomorphism $F$. Indeed, 
if $T_1(Q)Q=T_2(Q)Q$ for all $Q\in\mathcal{C}^n(X)$, then the automorphism 
$[T_1(Q)]^{-1}T_2(Q)$ is contained in the $\Aut(X)$-stabilizer of $Q$, which is trivial 
for general configurations $Q$. Therefore, $T_1=T_2$.
\vskip3pt

(c) According to Tame Map Theorem and Theorem \ref{Thm: 1st
main theorem in the Introduction}, the map $F$ in (\ref{eq:
formula for automorphism of Cn}), being an automorphism, must be
tame. This is indeed the case, with the morphism
%&
\begin{equation}\label{eq: TMT-T}
T\colon\,\mathcal{C}^n\to\Aff\mathbb{C},\,\quad \  T(Q)\zeta
=s\cdot(\zeta-\bc(Q))+A(\pi(Q))\bc(Q)\ \,,
\end{equation}
%&
where $\zeta\in\mathbb{C}$ and $Q\in\mathcal{C}^n$.
\vskip3pt

(d) Let $X=\mathbb{C}$. Then Tame Map Theorem holds also for $n=3$, but not for $n=4$. 
However, any {\em automorphism} of $\mathcal{C}^4(\mathbb{C})$
is tame. The automorphism groups of $\mathcal{C}^1(\mathbb{C})\cong\mathbb{C}$ and
$\mathcal{C}^2(\mathbb{C})\cong\mathbb{C}^*\times\mathbb{C}$ are well known, so we assume
 in the sequel that $n>2$.
\end{rems}

Using Tame Map Theorem, Zinde and Feler [loc. cit.] described all automorphisms of $\mathcal{C}^n(X)$ 
when $\dim_{\mathbb{C}}\Aut{X}=1$, or
more precisely, when $X$ is $\mathbb{C}^*$ and $n>4$, or $X$ is an elliptic curve and $n>2$. For $\mathbb{C}$ and $\mathbb{P}^1$, 
where the automorphism groups $\Aff\mathbb{C}$ and $\mathbf{PSL}(2,\mathbb{C})$ 
have dimension $2$ and $3$, respectively,
the problem becomes more difficult. The group $\Aut\mathcal{C}^n=\Aut\mathcal{C}^n(\mathbb{C})$
 is the subject of the present paper; the case $X=\mathbb{P}^1$ remains open.
\vskip3pt

The content of the paper is as follows. In Sections \ref{sc:
abstract scheme} and \ref{sc: special-Aut} we propose an abstract scheme to study the
automorphism groups of cylinders over rigid bases and, more generally, of tight  cylinders. 
An irreducible affine variety $\mathcal{X}$ will be called {\em rigid} if the images 
of non-constant morphisms 
$\mathbb{C}\to\reg\,\mathcal{X}$ do not cover any Zariski open dense subset in 
the smooth locus $\reg\,\mathcal{X}$, i.e., if
$\reg\,\mathcal{X}$ is {\em non-$\mathbb{C}$-uniruled}, see Definition \ref{def: rigidity}. 
Any cylinder $\mathcal{X}\times\mathbb{C}$ over a rigid base $\mathcal{X}$ is {\em tight}, 
meaning that its cylinder structure over $\mathcal{X}$ is unique, see Definition \ref{def: tightness} 
and Corollary \ref{Cor: automorphisms-of-rigid-cylinder-are-triangular}.

We show in Section \ref{ss: dir prod} that the bases $\mathcal{C}^{n-1}_{\blc}$,
 $\mathcal{SC}^{n-2}_{\blc}$, and $\Sigma^{n-2}_{\blc}$ of cylinders (\ref{eq: 3 cylinders}) are
rigid. So, the scheme of Section \ref{ss:
rigid cylinders} applied to the latter cylinders yields that their
automorphisms have a triangular form, see
(\ref{eq: triangular automorphism1})-(\ref{eq: tr-form1}).

For a tight cylinder $\mathcal{X}\times\mathbb{C}$ we describe in Subsection \ref{ss: spe-aut} the special automorphism group 
$\SAut(\mathcal{X}\times\mathbb{C})$ (see Definition \ref{def: alg-subgr}), and 
in Section \ref{sc: ind-grp} the
neutral component $\Aut_0\, (\mathcal{X}\times\mathbb{C})$ of the group
$\Aut (\mathcal{X}\times\mathbb{C})$ and its algebraic subgroups. In
Theorem \ref{Thm: Mostow} we establish an analog of Theorem
\ref{Thm: 2d main theorem in the Introduction} for such cylinders.
Besides, in Sections \ref{ss: actions and LND's on a rigid cylinder}-
\ref{ss: ML invariant of a cylinder over rigid base} we find the locally
nilpotent derivations of the algebra
$\mathcal{O}
(\mathcal{X}\times\mathbb{C})$ and its Makar-Limanov invariant subalgebra. In
Section \ref{ss: Lie-alg} we study the Lie algebra
$\Lie\,(\Aut_0\, (\mathcal{X}\times\mathbb{C}))$. These results are used
in the subsequent sections in the concrete setting of the
varieties $\mathcal{C}^n$, $\mathcal{SC}^{n-1}$, and $\Sigma^{n-1}$. 

Theorems \ref{Thm: 1st main theorem in the Introduction} and
\ref{Thm: 2d main theorem in the Introduction} are proven in Section
\ref{sc: Aut C-n Sigma-n}, see Theorems \ref{Thm: presentation}
and \ref{Thm: structure}, respectively. We provide analogs of our
main results for the automorphism groups of the special configuration space 
$\mathcal{SC}^{n-1}$, the discriminant variety $\Sigma^{n-1}$, and the pair 
$(\mathbb{C}^n,\Sigma^{n-1})$. All these groups are solvable; we also find 
presentations of their Lie algebras. In Section \ref{sc: Aut C-n} we show that all these groups are 
centerless and describe their commutator series,
semisimple and torsion elements. In Section \ref{sec:Zinde's Thm.} we give a description of the 
group $\Aut\mathcal{C}^{n}(\mathbb{C}^*)$ due to Zinde \cite{Zin78}. We provide a new proof 
based upon Zinde's analog of Artin's Theorem on the pure braid groups. 
In Section \ref{Sec: Aut-discriminant}  we apply Zinde's Theorem in order to complete the description of the automorphism group of $\Sigma^{n-1}$. In Section \ref{sec: KLZ-thm revisited} we
 provide an alternative proof of the structure theorem for the group $\Aut\mathcal{C}^{n-1}_{\blc}$, 
which does not refer to Tame Map Theorem.  In Section \ref{sec: Kaliman for n=4} we give a proof of Kaliman's Theorem on the group $\Aut\mathcal{SC}^{n-2}_{\rm blc}$ in the exceptional case $n=4$, where the original proof does not work. We reproduce an example from \cite{Lin79} which shows that the original Kaliman's Theorem on  endomorphisms of $\mathcal{SC}^{n-2}_{\rm blc}$ does not hold in this generality for $n=4$ (see Example \ref{ex: Lin's counterexample}).
Finally, in Section \ref{Sec: Holomorphic endomorphisms of Cblc(n-1)}, using the analytic counterpart of 
Tame Map Theorem, we obtain its analog for the
space $\mathcal{C}^{n-1}_{\blc}$ ($n>4$), describe the proper holomorphic self-maps of this space and the group 
of its biholomorphic automorphisms $\Aut_{\rm hol}\mathcal{C}^{n-1}_{\blc}$.

For the sake of uniformity, we work over the field $\mathbb C$. Indeed, 
this is essential when dealing with configuration spaces, 
since in this case we employ certain topological methods. 
By contrast, all the results 
concerning rigid varieties and tight cylinders, along with their proofs, remain 
valid over an arbitrary algebraically closed field of characteristic zero.

This preprint is an extended version of our paper \cite{Lin-Zaidenberg14}. 
Sections 2--4 of \cite{Lin-Zaidenberg14} have been modified so that the 
results stated in \cite{Lin-Zaidenberg14} under the  rigidity assumption are 
proven now under a weaker assumption of tightness. 
New Sections \ref{sec:Zinde's Thm.}--\ref{sec: Kaliman for n=4} 
are devoted to alternative proofs of some of the aforementioned results. 
This allowed to provide the optimal dimension bounds in our main theorems. 

The authors thank S.\ Kaliman for useful discussions concerning Section 
\ref{ss: actions and LND's on a rigid cylinder},  especially 
Example \ref{exa: Danielewski} and Remark \ref{rem: flexibility}. 
They are grateful also to N.~Ivanov and L.~Paris for 
providing useful information concerning different 
generalizations of the classical Artin's theorem on braid groups used in 
Sections \ref{sec:Zinde's Thm.} and \ref{sec: KLZ-thm revisited}.

\section{Automorphisms of cylinders over rigid bases}
\label{sc: abstract scheme}
\subsection{Triangular automorphisms}\label{ss: Triangular automorphisms}
(a) Let $\mathfrak{C}$ be a category of sets admitting direct
products $\mathcal{X}\times\mathcal{Y}$ of its objects with the 
standard projections to $\mathcal{X}$ and $\mathcal{Y}$ being morphisms.

Suppose that the automorphism group $\Aut\mathcal{Y}$ 
of a certain object $\mathcal{Y}$ is an object in $\mathfrak{C}$ 
satisfying the usual axioms, namely, that the following maps are morphisms:
\begin{equation}\label{eq: restriction for an object Y}
\aligned
%&\hskip-50pt\textup{\em }\\
&\textup{\em the action}\ \, (\Aut\mathcal{Y})\times\mathcal{Y}\to 
\mathcal{Y}\,, \ (\alpha, y)\mapsto \alpha(y)\,,\\
&\textup{\em the multiplication} \ \, \Aut\mathcal{Y}\times
\Aut\mathcal{Y}\to \Aut\mathcal{Y}\,, (\alpha,\beta)\mapsto 
\beta\circ\alpha\\
&\textup{\em the inversion}\ \, \Aut\mathcal{Y}\to\Aut\mathcal{Y}\,,\ 
\alpha\mapsto\alpha^{-1}.
\endaligned
\end{equation}
%&
\noindent Then for any $S\in\Aut\mathcal{X}$ and any morphism 
$A\colon\,\mathcal{X}\to\Aut\mathcal{Y}$ the map
%&
\begin{equation}\label{eq: triangular automorphism}
F=F_{S,A}\colon\,\mathcal{X}\times\mathcal{Y}
\to\mathcal{X}\times\mathcal{Y}\,, \ \ F(x,y)=(Sx,A(x)y)\ 
\textup{for all} \ (x,y)\in\mathcal{X}\times\mathcal{Y}\,,
\end{equation}
%&
is a morphism in the category $\mathfrak{C}$. 
Moreover, $F$ is an {\em automorphism} of 
$\mathcal{X}\times\mathcal{Y}$. Indeed, given $S,S'\in\Aut{X}$ 
and morphisms $A,A'\colon\,\mathcal{X}\to\Aut\mathcal{Y}$, 
for the corresponding $F,F'$ we have $F'F(x,y)=(S'Sx,A'(Sx)A(x)y)$, 
whereas the inverse $F^{-1}$ of $F$ corresponds to the couple $(S',A')$, 
where $S'=S^{-1}$ and the morphism 
$A'\colon\,\mathcal{X}\to\Aut\mathcal{Y}$ is defined by 
$A'(x)=(A(S^{-1}x))^{-1}$ for all $x\in\mathcal{X}$. 
We call such automorphisms $F$ of $\mathcal{X}\times\mathcal{Y}$ {\em
triangular} (with respect to the given product structure). 
All triangular automorphisms form a subgroup
${\Aut}_{\vartriangle}(\mathcal{X}\times\mathcal{Y})
\subset\Aut(\mathcal{X}\times\mathcal{Y})$.

(b) Suppose that an object $\mathcal{Y}\in\mathfrak{C}$
satisfies the conditions in (a). 
Then $\Mor(\mathcal{X},\Aut\mathcal{Y})$ with the pointwise 
multiplication of morphisms can be embedded in
${\Aut}_{\vartriangle}(\mathcal{X}\times\mathcal{Y})$
as a normal subgroup consisting of all $F$ of the form $F(x,y)=(x,A(x)y)$,
$A\in\Mor(\mathcal{X},\Aut\mathcal{Y})$. 
The corresponding quotient group is isomorphic to $\Aut\mathcal{X}$,  
and we have 
the semi-direct
product decomposition
%&
\begin{equation}\label{eq: s-d-decomp}
{\Aut}_{\vartriangle}(\mathcal{X}\times\mathcal{Y})
\cong\Mor(\mathcal{X},\Aut\mathcal{Y})\rtimes\Aut\mathcal{X}\,.
\end{equation}
%&
The second factor acts by conjugation on the first via $S.A=A\circ S^{-1}$, 
where $S\in\Aut\mathcal{X}$ and $A\in\Mor(\mathcal{X},\Aut\mathcal{Y})$.

\subsection{Automorphisms of tight cylinders}
\label{ss: rigid cylinders}
We are interested in the case where $\mathfrak{C}$ is 
the category of complex algebraic varieties and their morphisms, and 
$\mathcal{Y}=\mathbb{C}$. Thus, in the sequel we deal with {\em cylinders} 
$\mathcal{X}\times\mathbb{C}$. Since
$\Aut\mathcal{Y}=\Aff\mathbb{C}\in\mathfrak{C}$, the 
conditions 
(\ref{eq: restriction for an object Y}) 
are 
fulfilled. In fact, we deal mainly with affine varieties; moreover, 
in all our results
these varieties are assumed to be irreducible.

Let us introduce the following notions.

\bdefi\label{def: rigidity}
An irreducible variety $\mathcal{X}$ is called  $\mathbb{C}$-{\em uniruled} if for some variety $\mathcal{V}$ 
there is a dominant morphism $\mathcal{V}\times\mathbb{C}\to\mathcal{X}$ non-constant on a 
general ruling $\{v\}\times\mathbb{C}$, $v\in\mathcal{V}$ (\cite[Definition 5.2 and Proposition 5.1]{Je99}). 
We say that $\mathcal{X}$ is {\em rigid} if its smooth locus $\reg\,\mathcal{X}$ is
non-$\mathbb{C}$-uniruled. For such $\mathcal{X}$, the variety $\mathcal{X}\times\mathbb{C}$ is said to be the
{\em cylinder over a rigid base}.
\edefi

\bdefi\label{def: tightness} For an irreducible $\mathcal{X}$, we call the cylinder 
$\mathcal{X}\times\mathbb{C}$ {\em tight} if its cylinder
structure over $\mathcal{X}$ is unique, that is, if for any
automorphism $F\in\Aut(\mathcal{X}\times\mathbb{C})$ there is a
(unique) automorphism $S\in\Aut\mathcal{X}$ that fits in the
commutative diagram
%&
\begin{equation}\label{eq: diagram0}
\CD
{\mathcal{X}\times\mathbb{C}} @ > {F} >> {\mathcal{X}\times\mathbb{C}}\\
@V{\textup{pr}_1}VV @VV{\textup{pr}_1}V\\
{\mathcal{X}} @ > {S}>>{\mathcal{X}}
\endCD
\end{equation}
%&
Thus, $\mathcal{X}\times\mathbb{C}$ is tight if and only if every
$F\in\Aut(\mathcal{X}\times\mathbb{C})$ is triangular, so that
%&
\begin{equation}\label{eq: tr-form0}
\Aut(\mathcal{X}\times\mathbb{C})
={\Aut}_{\vartriangle}(\mathcal{X}\times\mathbb{C})\,.
\end{equation}
%&
For a cylinder $\mathcal{X}\times\mathbb{C}$ formula (\ref{eq: triangular automorphism}) 
takes the form
%&
\begin{equation}\label{eq: triangular automorphism1}
F(x,y)=(Sx,A(x)y)=(Sx, ay+b)\,\ \textup{for any} \
(x,y)\in\mathcal{X}\times\mathbb{C}
\end{equation}
%&
with $a\in\mathcal{O}^\times (\mathcal{X})$ and
$b\in\mathcal{O}_+(\mathcal{X})$. If $\mathcal{X}\times\mathbb{C}$ is tight, 
then, by (\ref{eq: tr-form0}) and (\ref{eq: triangular automorphism1}), we have
%&
\begin{equation}\label{eq: tr-form1}
\Aut(\mathcal{X}\times\mathbb{C})
\cong\Mor\,(\mathcal{X},\Aff\mathbb{C})\rtimes\Aut\mathcal{X}
\ \ \textup{and} \ \ \Mor\,(\mathcal{X},\Aff\mathbb{C})\cong
\mathcal{O}_+(\mathcal{X})\rtimes\mathcal{O}^\times
(\mathcal{X})\,,
\end{equation}
%&
where $\mathcal{O}^\times (\mathcal{X})$ acts on
$\mathcal{O}_+(\mathcal{X})$ by multiplication $b\mapsto{ab}$ for 
$a\in\mathcal{O}^\times(\mathcal{X})$ and $b\in\mathcal{O}_+(\mathcal{X})$.
 The group $\Aut(\mathcal{X}\times\mathbb{C})$ of a tight cylinder is solvable 
as long as $\Aut\mathcal{X}$ is.
\edefi

Thus the tightness is an important property. In 
Theorem \ref{cor: BMLC-improvement} we give several useful criteria of tightness. 

\bdefi\label{def: strong-cancellation} One says that a variety $\mathcal{X}$ possesses the 
{\em strong cancellation property} if
for any $m>0$, any variety $\mathcal{Y}$, and any isomorphism
$F\colon \mathcal{X}\times\mathbb{C}^m\stackrel{\simeq}{\longrightarrow}
\mathcal{Y}\times\mathbb{C}^m$ there is an isomorphism
$S\colon\mathcal{X}\stackrel{\simeq}{\longrightarrow}\mathcal{Y}$ that
fits in the commutative diagram
%&
\begin{equation}\label{eq: diagram1}
\CD
{\mathcal{X}\times\mathbb{C}^m} @ > {F} >> 
{\mathcal{Y}\times\mathbb{C}^m}\\
@V{\textup{pr}_1}VV @VV{\textup{pr}_1}V\\
{\mathcal{X}} @ > {S}>>{\mathcal{Y}}
\endCD
\end{equation}
%&
\edefi

\noindent {\bf Dry{\l}o's Theorem I}  (\cite[(I)]{Dr05bis}). 
{\em The strong cancellation holds for any rigid affine variety.}
\vskip3mm

For the reader's convenience, we provide a short argument for 
the next corollary.

\bcor\label{Cor: automorphisms-of-rigid-cylinder-are-triangular}
If $\mathcal{X}$ is rigid, then $\mathcal{X}\times\mathbb{C}$
is tight, i.e., $\Aut(\mathcal{X}\times\mathbb{C})
={\Aut}_{\vartriangle}(\mathcal{X}\times\mathbb{C})$.
\ecor 

\bproof
Let us show that any $F\in\Aut (\mathcal{X}\times\mathbb{C})$
sends the rulings $\{x\}\times\mathbb{C}$ into rulings.
Then the same holds for $F^{-1}$, and so $S\Def {\rm pr}_1\circ
F\vert_{\mathcal{X}\times\{0\}}\in \Aut\mathcal{X}$ fits in
diagram (\ref{eq: diagram0}).

Assuming the contrary, we consider the family
$\{F(\{x\}\times\mathbb{C})\}_{x\in\reg\,\mathcal{X}}$. Projecting
it to $\reg\,\mathcal{X}$ we get a contradiction with the rigidity
assumption.
\eproof

\bdefi\label{def: strong-1-cancellation} We say that the {\em strong 1-cancellation} holds for $\mathcal{X}$
if one has diagram (\ref{eq: diagram1})  with $m=1$.\edefi

Clearly, the strong 1-cancellation property implies the tightness. 
The converse is also true, see Proposition \ref{prop: non-tightness}. 

\brem\label{rem: BML05} The tightness of the cylinder $\mathcal{X}\times\mathbb{C}$ 
does not imply the rigidity of  $\mathcal{X}$, in general. Indeed, there are examples of
 non-rigid smooth, affine surfaces $\mathcal{X}$ with a tight cylinder $\mathcal{X}\times\mathbb{C}$, 
cf., e.g., \cite[Example 3]{BML05} and Theorem \ref{cor: BMLC-improvement}.\erem

\subsection{$\mathbb{C}_+$-actions and LND's on tight cylinders}
\label{ss: actions and LND's on a rigid cylinder}
In this subsection we propose several criteria of tightness. Let us recall the necessary notions. 

\bdefi\label{sit: lnd's}
A derivation $\p$ of a ring
$A$ is {\em locally nilpotent} if $\p^n a=0$ for any $a\in A$ and
for some $n\in\N$ depending on $a$. For any $f\in \ker\partial$
the derivation $f\partial\in\Der{A}$ is again locally nilpotent; it is called a {\em replica} of $\partial$ (see \cite{AFKKZ}). 
We let $\LND(A)$ denote the set of all  nonzero locally nilpotent derivations of $A$.\edefi

Let $\mathcal{X}$ be an affine variety. For any $\partial\in\LND(\mathcal{O}(\mathcal{X}))$ the correspondence
$\mathbb{C}_+\ni t\mapsto \exp(t\partial)\in\Aut\,\mathcal{X}$ defines  
a  unipotent one-parameter group of automorphisms of $\mathcal{X}$, or, in other words, 
an effective $\mathbb{C}_+$-action on $\mathcal{X}$. In fact, any $\mathbb{C}_+$-action 
on $\mathcal{X}$ arises in this way (see e.g., \cite{Fr}).

\bprop\label{prop: non-tightness}  \begin{itemize}\item[(a)] 
If the cylinder $\mathcal{X}\times\mathbb{C}$  is tight, then $\mathcal{X}$ does not admit any 
nontrivial $\mathbb{C}_+$-action, i.e., $\LND(\mathcal{O}(\mathcal{X}))=\emptyset$. 
\item[(b)] The cylinder $\mathcal{X}\times\mathbb{C}$ is tight if and only if 
 the strong 1-cancellation holds for $\mathcal{X}$.\end{itemize}\eprop

\bproof (a) 
Assume to the contrary that $\mathcal{X}\times\mathbb{C}$ is tight, while there is a 
nontrivial $\mathbb{C}_+$-action $t\mapsto \exp(t\partial)$ on $\mathcal{X}$, where 
$\partial\in\LND(\mathcal{O}(\mathcal{X})$. The induced $\mathbb{C}_+$-action on the cylinder 
$\mathcal{X}\times\mathbb{C}$ has the same infinitesimal generator  $\partial$, regarded this time 
as a locally nilpotent derivation of the algebra $\mathcal{O}(\mathcal{X})[u]$  
with $\partial u=0$. 
The derivation $\partial_1=u\partial$ is again locally nilpotent. The associated vector field on 
$\mathcal{X}\times\mathbb{C}$ vanishes on  the section $\mathcal{X}\times\{0\}$. 
Hence the induced $\mathbb{C}_+$-action $t\mapsto \exp(t\partial_1)\in\Aut\,(\mathcal{X}\times\mathbb{C})$ 
fixes each point $Q_0\in\mathcal{X}\times\{0\}$, and moves a general point $Q_1\in \mathcal{X}\times\mathbb{C}$ 
in the horizontal direction. Taking the points $Q_0$ and $Q_1$ on the same ruling $\{q\}\times\mathbb{C}$, 
where $q\in\mathcal{X}$ is general,  we see that their images under the automorphism $\alpha=exp(\partial_1)$ 
do not belong any longer to the same ruling. It follows that $\alpha\in\Aut(\mathcal{X}\times\mathbb{C})$ 
is not triangular. Therefore, the cylinder $\mathcal{X}\times\mathbb{C}$ cannot be tight, a contradiction. 

(b) Clearly, the  strong 1-cancellation property implies tightness. To prove the converse, suppose that the 
 strong 1-cancellation fails for $\mathcal{X}$, i.e., that 
$\mathcal{X}\times\mathbb{C}\simeq\mathcal{Y}\times\mathbb{C}$,
 where the rulings of these cylinders are different. Let us show that
 in this case $\mathcal{X}\times\mathbb{C}$ cannot be tight. 

Indeed, assume to the contrary that $\mathcal{X}\times\mathbb{C}$ is tight. Let $\tau'$ be the free 
$\mathbb{C}_+$-action on $\mathcal{X}\times\mathbb{C}\simeq\mathcal{Y}\times\mathbb{C}$ 
via the shifts along the rulings of the cylinder $\mathcal{Y}\times\mathbb{C}$. Since 
$\mathcal{X}\times\mathbb{C}$ is tight,
this action is triangular, and so, it induces a $\mathbb{C}_+$-action on $\mathcal{X}$. The latter 
action is non-trivial, since otherwise the both families of rulings would be the same. By (a), this implies
 that $\mathcal{X}\times\mathbb{C}$ is not tight, a contradiction.\eproof

\begin{exa}\label{exa: Danielewski}  Let $\mathcal{X}$ be the surface  given in $\mathbb{C}^3$ 
by equation $x^ny = p(x, z)$,
where $n\in\N$, $p\in\mathbb{C}[x,y]$,  and $p(0,z)\neq 0$. Then $\mathcal{X}$ admits
 a nontrivial $\mathbb{C}_+$-action $t\mapsto \exp(t\partial)$, where
 $\partial=(\partial p/\partial z)\partial/\partial y+x^n\partial/\partial z$.
Hence by Proposition \ref{prop: non-tightness}(a) the cylinder $\mathcal{X}\times\mathbb{C}$ is not tight.

This applies, in particular, to the Danielewski surfaces $\mathcal{X}_n=\{x^ny-z^2+1=0\}$ in $\mathbb{C}^3$, 
$n\in\N$. By Danielewski's Theorem \cite{Dan89},  the cylinders $\mathcal{X}_n\times\mathbb{C}$
 are all isomorphic, whereas, according to Fieseler \cite{Fi94}, $\mathcal{X}_n$ is not even
 homeomorphic to $\mathcal{X}_m$ for $n\neq m$. Thus these surfaces provide counterexamples
 to cancellation\footnote{See e.g., \cite{Je09} 
for further examples of non-cancellation.}.
In addition,  the cylinders over the Danielewski surfaces are not tight.

To make the latter example more explicit, take two points 
$Q_0=(q,0)$ and $Q_1=(q,1)$ on the same ruling $\{q\}\times\mathbb{C}$
 of the cylinder $\mathcal{X}_n\times\mathbb{C}$, where $q=(0,0,1)\in \mathcal{X}_n$. 
Letting $\partial=2z\p/\p y+x^n\partial/\partial z\in\LND(\mathbb{C}[x,y,z])$, 
we consider the locally nilpotent derivation $\partial_1=u\partial$ of the algebra 
$\mathbb{C}[x,y,z,u]$. Since $\partial(x^ny-z^2+1)=0$, the derivation $\partial_1$ 
descends to a locally nilpotent derivation of the quotient
 $\mathbb{C}[x,y,z,u]/(x^ny-z^2+1)\simeq\mathcal{O}(\mathcal{X}_n\times\mathbb{C})$. 
The triangular automorphism
$$\alpha=\exp(\partial_1)\in\Aut\,\mathbb{C}^4,\quad  (x,y,z,u)\mapsto (x,y+2zu+x^nu^2,z+x^nu,u)\,,$$
 preserves the hypersurface
$\{x^ny-z^2+1=0\}\simeq\mathcal{X}_n\times\mathbb{C}$ in $\mathbb{C}^4$. 
This action fixes $Q_0$ and sends
$Q_1$ to  $\alpha(Q_1)=(0,2,1,1)$, so that the points $\alpha(Q_0)$ and $\alpha(Q_1)$ 
do not belong any more to the same ruling of the cylinder.
 Therefore, $\alpha\in\Aut (\mathcal{X}_n\times\mathbb{C})\setminus\Aut_{\vartriangle} (\mathcal{X}_n\times\mathbb{C})$.
\end{exa}

\brem\label{rem: flexibility}
By Theorem 3.1 in \cite{AKZ}, the surface $\mathcal{X}_1$ is {\em flexible}, i.e.,
 the tangent vectors to the orbits of the $\mathbb{C}_+$-actions on $\mathcal{X}_1$ generate 
the tangent space at any point of $\mathcal{X}_1$. It follows that the cylinder $\mathcal{X}_1\times\mathbb{C}$ 
is also flexible. By Theorem 0.1 in \cite{AFKKZ}, the flexibility implies the $k$-transitivity of the automorphism group 
$\Aut (\mathcal{X}_1\times\mathbb{C})$ for any $k\ge 1$. In particular, for any $n\ge 1$, there are automorphisms
of the cylinder  $\mathcal{X}_n\times\mathbb{C}\cong \mathcal{X}_1\times\mathbb{C}$
 that do not preserve the cylinder structure, and so send it to another such structure 
over the same base $\mathcal{X}_n$. 
This shows again that none of the cylinders $\mathcal{X}_n\times\mathbb{C}$ is tight. \erem

 We call the {\em rulings} of a cylinder $\mathcal{X}\times\mathbb{C}$ the fibers of the first
 projection ${\rm pr}_1\colon\,\mathcal{X}\times\mathbb{C}\to \mathcal{X}$.
 These are the orbits of the `vertical' free $\mathbb{C}_+$-action $\tau$ on 
 $\mathcal{X}\times\mathbb{C}$ via translations along the second factor.
In addition to the criterion of tightness in Proposition \ref{prop: non-tightness}(b), 
we have the following one. Consider the locally nilpotent derivation $\p/\p y$ on the
algebra $\mathcal{O}(\mathcal{X}\times\mathbb{C})
\cong\mathcal{O}(\mathcal{X})[y]$ with the phase flow $(x,y)\mapsto (x, y+t x)$. Its replica $b(x)\p/\p y$, where
$b\in\mathcal{O}(\mathcal{X})
$, has the phase flow $
(x,y)\mapsto (x, y+t b(x))$.

\bprop\label{prop: preservation of a cylinder} 
The cylinder $\mathcal{X}\times\mathbb{C}$ is tight if and only if any $\mathbb{C}_+$-action 
on this cylinder preserves each ruling, i.e., is of the form 
$
(x,y)\mapsto (x, y+t b(x))$, where
$b\in\mathcal{O}(\mathcal{X})
$. 
\eprop

\bproof 
Suppose first that the cylinder $\mathcal{X}\times\mathbb{C}$ is not tight, and so, 
 admits a second, different cylinder structure with a different family of rulings. 
Then the induced $\mathbb{C}_+$-action, say, $\tau'$ on $\mathcal{X}\times\mathbb{C}$
 has different orbits, and so,  does not preserve the rulings of the original cylinder. 

Conversely, suppose that  the cylinder 
$\mathcal{X}\times\mathbb{C}$ 
is tight.
Consider a $\mathbb{C}_+$-action $\varphi$  on this cylinder. 
Since $\varphi$ is triangular, it induces a $\mathbb{C}_+$-action, say, 
$\psi$ on $\mathcal{X}$. Assuming to the contrary that $\varphi$ is not vertical, i.e., 
does not preserve the rulings $\{x\}\times\mathbb{C}$, $x\in\mathcal{X}$, 
 the action $\psi$ is nontrivial. 
By Proposition \ref{prop: non-tightness}(a), the cylinder $\mathcal{X}\times\mathbb{C}$
 cannot be tight, a contradiction. 
\eproof

Let us show further that an arbitrary locally 
nilpotent derivation on a tight cylinder is a replica of the derivation $\p/\p y$.

\bprop\label{Prp: nilpotent derivations in rigid cylinder}
If $\mathcal{X}\times\mathbb{C}$ is tight then any 
$\p\in\LND(\mathcal{O}(\mathcal{X}\times\mathbb{C}))$ is a replica of the derivation $\p/\p y$, i.e.,
%&
$$
\p = f\p/\p y,\quad\mbox{where}\quad
f\in\mathcal{O}^\tau(\mathcal{X}\times\mathbb{C}) = {\rm
pr}_1^*(\mathcal{O}(\mathcal{X}))=\ker \p/\p y\,.
$$
%&
Consequently, any $\mathbb{C}_+$-action on $\mathcal{X}\times\mathbb{C}$ is of the form
%&
$$
(x,y)\mapsto (x, y+t b(x)),\quad\mbox{where}\quad
t\in\mathbb{C}
\quad\mbox{and}\quad b\in\mathcal{O}(\mathcal{X})\,.
$$
%&
\eprop

\bproof Indeed, both $\p$ and $\p/\p y$ can be viewed as regular vector
fields on $\mathcal{X}\times\mathbb{C}$, where the latter field is
non-vanishing. By Proposition \ref{prop: preservation of a cylinder},
these vector fields are proportional. That is, there exists a function
$f\in\mathcal{O}^\tau(\mathcal{X}\times\mathbb{C})$ such that
$\p = f\p/\p y$, which proves the first assertion. Now the second follows.
\end{proof}

\subsection{The group $\SAut(\mathcal{X}\times\mathbb{C})$}
\label{ss: spe-aut} \bdefi\label{def: alg-subgr}
Let $\mathcal{Z}$ be an irreducible algebraic variety. A subgroup $G\subset\Aut\mathcal{Z}$
 is called {\em algebraic} if it admits a structure of an algebraic group such that the natural map 
$G\times\mathcal{Z}\to\mathcal{Z}$ is a morphism. The {\em special automorphism group} $\SAut\mathcal{Z}$ is the
subgroup of $\Aut\mathcal{Z}$ generated by all the algebraic subgroups of
$\Aut\mathcal{Z}$ isomorphic to $\mathbb{C}_+$ (see e.g. \cite{AFKKZ}). 
Clearly,  $\SAut\mathcal{Z}$ is a normal subgroup of $\Aut\mathcal{Z}$.
\edefi

Assume that $\mathcal{X}$ is tight. Due to
(\ref{eq: tr-form1}) we have the
decomposition
\begin{equation}\label{eq: triple decomposition}
\Aut(\mathcal{X}\times\mathbb{C})\cong
\left(\mathcal{O}_+ (\mathcal{X})\rtimes\mathcal{O}^\times
(\mathcal{X})\right)\rtimes\Aut{\mathcal{X}}\,.
\end{equation}
In the next corollary we show that the group $\SAut(\mathcal{X}\times\mathbb{C})$ corresponds
 to the factor $\mathcal{O}_+ (\mathcal{X})$ in (\ref{eq: triple decomposition}).

\bprop\label{Crl: SAut for rigid cylinder}
If the cylinder $\mathcal{X}\times\mathbb{C}$ is tight, then $G=\SAut (\mathcal{X}\times\mathbb{C})$
is an Abelian group with Lie algebra\footnote{See \S \ref{ss: Lie-alg}.}
$$
L={\rm Lie}\, G=\mathcal{O}_+^{\tau}(\mathcal{X}\times\mathbb{C})\p/\p
y=\mathcal{O}_+ (\mathcal{X})\p/\p y\,.
$$
Furthermore, the exponential map $\exp\colon\, L\to G$ yields an isomorphism of groups
$$
\mathcal{O}_+ (\mathcal{X})\stackrel\cong\longrightarrow
\SAut (\mathcal{X}\times\mathbb{C})\,.
$$
\eprop

\bproof
This follows easily from Proposition
\ref{Prp: nilpotent derivations in rigid cylinder}.
\eproof

From Proposition \ref{prop: preservation of a cylinder}  we deduce such a corollary.

\bcor\label{cor: tightness-SAut}  The cylinder $\mathcal{X}\times\mathbb{C}$ 
is tight if and only if the orbits of the group $\SAut (\mathcal{X}\times\mathbb{C})$ are the rulings of this cylinder. \ecor

\subsection{The Makar-Limanov invariant of a cylinder}
\label{ss: ML invariant of a cylinder over rigid base}  The subalgebra of
$\tau$-invariants $\mathcal{O}^\tau(\mathcal{X}\times\mathbb{C})
\subset\mathcal{O}(\mathcal{X}\times\mathbb{C})$ admits
yet another interpretation. 

\bdefi\label{lnds} Let $\mathcal{Z}$ be an affine
algebraic variety over $\mathbb{C}$. The ring of invariants
$\mathcal{O}(\mathcal{Z})^{\SAut \mathcal{Z}}$ is called the {\em Makar-Limanov
invariant} of $\mathcal{Z}$ and is denoted by $\ML(\mathcal{Z})$. This ring is
invariant under the induced action of the group $\Aut \mathcal{Z}$ on
$\mathcal{O}(\mathcal{Z})$.

Consider a locally nilpotent
derivation $\partial\in\LND(\mathcal{O}(\mathcal{Z}))$
and the corresponding unipotent one-parameter algebraic group $H=\exp(\mathbb{C}\partial)\subset\Aut \mathcal{Z}$.
We have $\mathcal{O}(\mathcal{Z})^H=\ker\partial$ and so
\be\label{eq: ML}
\ML(\mathcal{Z})=\bigcap_{\partial\in\LND(\mathcal{O}(\mathcal{Z}))}\ker\partial\,.
\ee
Since the $\SAut(\mathcal{X}\times\mathbb{C})$-invariant functions 
are exactly the functions constant on each ruling of the cylinder, the next corollary is immediate from
Corollary \ref{cor: tightness-SAut}.
\edefi

\bcor\label{cor: ML} The cylinder $\mathcal{X}\times\mathbb{C}$ over an affine variety 
$\mathcal{X}$ is tight if and only if
%&
\be\label{eq: ML=O}
\ML(\mathcal{X}\times\mathbb{C})
=\mathcal{O}^\tau(\mathcal{X}\times\mathbb{C})=\mathcal{O}(\mathcal{X})\,.
\ee
%&
\ecor

It is worthwhile mentioning the following related 
result. 
\medskip

\noindent {\bf Dry{\l}o's Theorem II} (\cite{Dr11}). 
Let $\mathcal{X}$ and $\mathcal{Y}$ be irreducible affine varieties over an algebraically
 closed field $k$. If $\mathcal{X}$ is rigid, then 
$\ML(\mathcal{X} \times\mathcal{Y})=\mathcal{O}(\mathcal{X}) \otimes_k \ML(\mathcal{Y})$.
\vskip 3mm

There is a stronger statement in the case where $\mathcal{Y}$ is the affine line. 
\vskip 3mm

\noindent {\bf Bandman, Makar-Limanov, and Crachiola's Theorem}
 (\cite[Lemma 2]{BML05}, \cite[Theorem 3.1]{CML08}). {\em Let $\mathcal{X}$
 be an affine variety over 
an arbitrary field $k$. If $\ML(\mathcal{X})=\mathcal{O}(\mathcal{X})$, then 
the equality $\ML(\mathcal{X}\times\mathbb{A}^1_k)
=\mathcal{O}(\mathcal{X})$ holds. }
\vskip 3mm

Summarizing several previous  results, we can deduce the following tightness 
criteria\footnote{As we already mentioned, these criteria remain valid 
over any algebraically closed field $k$ of zero characteristic.}.

\bthm\label{cor: BMLC-improvement} For  an affine variety $\mathcal{X}$, 
the following conditions are equivalent:
\begin{itemize}\item[(i)]  the cylinder $\mathcal{X}\times\mathbb{C}$ is tight, i.e., 
$\Aut(\mathcal{X}\times\mathbb{C})
={\Aut}_{\vartriangle}(\mathcal{X}\times\mathbb{C})$;
\item[(ii)] the strong 
1-cancellation holds for $\mathcal{X}$;
 \item[(iii)] any $\mathbb{G}_a$-action on 
$\mathcal{X}\times\mathbb{C}$ is of the form
\smallskip
$
(x,y)\mapsto (x, y+t b(x))$, where
$t\in\mathbb{C}$ and
$b\in\mathcal{O}(\mathcal{X})$;
\item[(iv)] any 
$\p\in\LND(\mathcal{O}(\mathcal{X}\times\mathbb{C}))$ is of the form
\smallskip
$
\p = b(x)\p/\p y$, where
$b\in\mathcal{O}(\mathcal{X})
$; 
\item[(v)] the orbits of the group $\SAut (\mathcal{X}\times\mathbb{C})$ are the rulings of this cylinder;
\item[(vi)] 
$\ML(\mathcal{X}\times\mathbb{C})
=\mathcal{O}(\mathcal{X})$;  
\item[(vii)] $\ML(\mathcal{X})=\mathcal{O}(\mathcal{X})$, i.e., $\LND(\mathcal{O}(\mathcal{X}))=\emptyset$. 
\end{itemize}\ethm

\bproof 
Equivalences (i)$\Leftrightarrow$(ii),  (i)$\Leftrightarrow$(iii),  
(i)$\Leftrightarrow$(v), and (i)$\Leftrightarrow$(vi) are established in Propositions   
\ref{prop: non-tightness} and \ref{prop: preservation of a cylinder} 
and Corollaries \ref{cor: tightness-SAut}  and \ref{cor: ML}, respectively. 
Equivalence (iii)$\Leftrightarrow$(iv) is easy, and  implication (i)$\Rightarrow$(iv) 
follows from Proposition \ref{Prp: nilpotent derivations in rigid cylinder}. 
Thus the conditions (i)-(vi) are all equivalent.  Implication 
(vii)$\Rightarrow$(vi) follows from the
Bandman-Makar-Limanov-Crachiola Theorem.  To show (vi)$\Rightarrow$(vii)  
suppose that 
$\ML(\mathcal{X})\neq\mathcal{O}(\mathcal{X})$.  
Then $\LND(\mathcal{O}(\mathcal{X}))\neq\emptyset$, 
and so, there is a nontrivial $\mathbb{C}_+$-action on $\mathcal{X}$. 
 It follows by Proposition \ref{prop: non-tightness}(a) that the cylinder 
$\mathcal{X}\times\mathbb{C}$ admits a non-triangular automorphism.  
Then it possesses a second cylinder structure (over the same base $\mathcal{X}$). 
Let $\partial$ be
the locally nilpotent derivation of $\mathcal{O}(\mathcal{X}\times \mathbb{C})$ 
generating the one-parameter group of shifts along the rulings of the new cylinder. 
The kernel  $\ker\partial$ is different from $\mathcal{O}^\tau(\mathcal{X}\times \mathbb{C})$. 
Hence the intersection $\ML(\mathcal{X}\times \mathbb{C})$ of the kernels of the 
locally nilpotent derivations on $\mathcal{O}(\mathcal{X}\times \mathbb{C})$
 is strictly smaller than $\mathcal{O}^\tau(\mathcal{X}\times \mathbb{C})$, that is, 
$\ML(\mathcal{X}\times \mathbb{C})
\neq\mathcal{O}^\tau(\mathcal{X}\times \mathbb{C})$ (see (\ref{eq: ML})). 
This provides  (vi)$\Rightarrow$(vii). 
\eproof

\subsection{Configuration spaces and discriminant levels as cylinders over rigid bases}\label{ss: dir prod}
In Proposition \ref{Prp: rigidity} we show that the bases 
$\mathcal{C}^{n-1}_{\blc}$, $\mathcal{SC}^{n-2}_{\blc}$, and $\Sigma^{n-2}_{\blc}$ 
of the cylinders (\ref{eq: 3 cylinders}) possess a property, which implies the rigidity. 
Therefore, all automorphisms of these cylinders are triangular 
(Corollary \ref{Crl: our automorphisms are triangular}). 
The latter applies as well to the hypersurfaces $D_n(Q)=c\ne{0}$. Indeed, since $D_n$ 
is homogeneous, any such hypersurface is isomorphic to 
$\mathcal{SC}^{n-1}$. 

\begin{nota}\label{Ntn: ordered configuration space} 
For any $X$ and any $n\in\mathbb{N}$, let $\mathcal{C}^n_{\ord}(X)$ denote
the {\em ordered configuration space} of $X$, i.e.,
$\mathcal{C}^n_{\ord}(X)=\{(q_1,...,q_n)\in{X}^n\mid\, q_i\ne{q_j}\ \
\textup{for all} \ i\ne{j}\}$. The symmetric group $\mathbf{S}(n)$ acts freely on 
$\mathcal{C}^n_{\ord}(X)$  permuting the coordinates $q_1,...,q_n$. By definition, 
$\mathcal{C}^n_{\ord}(X)/\mathbf{S}(n)=\mathcal{C}^n(X)$. We  let
%&
$$
\mathcal{C}^n_{\ord}\Def\mathcal{C}^n_{\ord}(\mathbb{C})\quad\textup{and}\quad
\mathcal{C}^{n-1}_{\ord,\blc}\Def\{q=\{q_1,...,q_n\}\in\mathcal{C}^n_{\ord}
\mid\, q_1+...+q_n=0\}\,.
$$
%&
Clearly $\mathcal{C}^{n}_{\ord}/\mathbf{S}(n)
=\mathcal{C}^{n}$ (see (\ref{eq: Cn(C)=Cn-Sigma})) 
and $\mathcal{C}^{n-1}_{\ord,\blc}/\mathbf{S}(n)
=\mathcal{C}^{n-1}_{\blc}$.
The variety $\mathcal{C}^n$ and the discriminant variety
%&
$$
\Sigma^{n-1}=\{z\in\mathbb{C}^n\mid\,d_n(z)=0\}\,
$$
%&
(see (\ref{eq: discriminant variety})) can be viewed as
complementary to each other parts of the symmetric power
$\Sym^n\mathbb{C}=\mathbb{C}^n_{(q)}/\mathbf{S}(n)$.
We have the quotient morphisms
%&
\begin{equation}\label{eq: projection p}
p\colon\mathbb{C}^n_{(q)}\to\Sym^n\mathbb{C}\cong\mathbb{C}^n_{(z)}\,,\ \
\Delta^{n-1}\to\Sigma^{n-1}\,, \ \ \textup{and} \ \
\mathbb{C}^n_{(q)}\setminus\Delta^{n-1}\to\mathcal{C}^n\,,
\end{equation}
%&
where $\Delta^{n-1}\Def\bigcup
\limits_{i\ne{j}}\{q=(q_1,...,q_n)\in\mathbb{C}^n\mid\,
q_i=q_j\}$
is the big diagonal.  The points $z\in\Sigma^{n-1}$ are in one-to-one 
correspondence with unordered $n$-term multisets (or corteges) 
$Q=\{q_1,\ldots,q_n\}$, $q_i\in\mathbb{C}$, 
with at least one repetition.
\vskip3pt

Let $\mathcal{Z}$ be one of the varieties $\mathcal{C}^n$, $\mathcal{SC}^{n-1}$, 
or $\Sigma^{n-1}$. The corresponding {\em balanced}
variety $\mathcal{Z}_{\blc}\subset\mathcal{Z}$ is defined by
%&
\begin{equation}\label{eq: balanced spaces}
\mathcal{Z}_{\blc}=\{Q\in\mathcal{Z}\mid\, \bc(Q)=0\}\,, \ \
\dim_{\mathbb{C}}\mathcal{Z}_{\blc}=\dim_{\mathbb{C}}\mathcal{Z}-1\,
\end{equation}
%&
(cf.\ (\ref{eq: balanced spaces0})). The free regular $\mathbb{C}_+$-action $\tau$ 
on $\mathbb{C}^n$ defined by
$\tau_\zeta{Q}=Q+\zeta=\{q_1+\zeta,...,q_n+\zeta\}$ for $\zeta\in\mathbb{C}$ and 
$Q\in\mathbb{C}^n$ preserves $\mathcal{Z}$. The orbit map of $\tau$ gives the morphism
%&
\begin{equation}\label{eq: pi}
\pi\colon\,\mathcal{Z}\to\mathcal{Z}_{\blc},\quad
{Q}\mapsto{Q^\circ\Def{Q-\bc(Q)}}\,,\quad \textup{and}\quad \pi'\colon
\,\mathcal{Z}\to\mathbb{C}\,,\quad {Q\mapsto\bc(Q)}\,
\end{equation}
%&
with all fibers reduced and isomorphic to 
$\mathbb{C}$. The corresponding {\em cylindrical} direct decomposition 
$\mathcal{Z}=\mathcal{Z}_{\blc}\times\mathbb{C}$ 
leads to decompositions of our varieties
%&
\begin{equation}\label{eq: 3 cylinders}
\mathcal{C}^n=\mathcal{C}^{n-1}_{\blc}\times\mathbb{C}\,,
\quad \mathcal{SC}^{n-1}=\mathcal{SC}^{n-2}_{\blc}\times\mathbb{C}\,,
\quad\textup{and}\quad \ \Sigma^{n-1}=\Sigma^{n-2}_{\blc}\times\mathbb{C}\,,
\end{equation}
%&
which play an important part in what follows. 
\vskip3pt
\end{nota}

Note that the regular part $\reg\,\Sigma^{n-1}$ 
of the discriminant variety $\Sigma^{n-1}$ 
consists of all the unordered $n$-multisets $Q=\{q_1,...,q_{n-2},u,u\}\subset\mathbb{C}$ 
with $q_i\ne{q_j}$ for $i\ne{j}$ and $q_i\ne{u}$ for all $i$. Since the
hyperplane $q_1+...+q_n=0$ is transversal to each of the hyperplanes $q_i=q_j$, 
the regular part $\reg\,\Sigma^{n-2}_{\blc}$ of\,
$\Sigma^{n-2}_{\blc}$ consists of all the multisets $Q=\{q_1,...,q_{n-2},u,u\}$ 
as above that satisfy the additional condition $\sum\limits_{i=1}^{n-2}q_i+2u=0$. 
In the proofs of Proposition \ref{Prp: rigidity} and Theorem \ref{Thm: aut balanced Sigma} 
 we need the following lemma.

\blem\label{Lm: reg Sigma(n-2)blc=C(n-2)(C*)}
For $n>2$ the regular part $\reg\,\Sigma^{n-2}_{\blc}$ of\,
$\Sigma^{n-2}_{\blc}$ is isomorphic to the configuration space
$\mathcal{C}^{n-2}(\mathbb{C}^{*})$. Consequently,
$\Aut(\reg\,\Sigma^{n-2}_{\blc})\cong
\Aut\mathcal{C}^{n-2}(\mathbb{C}^{*})$.
\elem

\begin{proof}
An isomorphism
$\reg\,\Sigma^{n-2}_{\blc}\cong\mathcal{C}^{n-2}(\mathbb{C}^{*})$ 
does exist since both these varieties are cross-sections of the standard $\mathbb{C}_+$-action 
$\tau$ on the cylinder $\reg\,\Sigma^{n-2}=(\reg\,\Sigma^{n-2}_{\blc})\times\mathbb{C}$
(see (\ref{eq: 3 cylinders})). To construct such an isomorphism explicitly, for any
%&
$$
Q=\{q_1,...,q_{n-2},u,u\}\in\reg\,\Sigma^{n-2}_{\blc}\,,\ \ \textup{where} \ u
=-\frac{1}{2}\sum\limits_{i=1}^{n-2}q_i\,,
$$
%&
we let $\widetilde{Q}=\{q_1-u,...,q_{n-2}-u\}$.
Then $\widetilde{Q}\in\mathcal{C}^{n-2}(\mathbb{C}^{*})$, and we have an epimorphism
%&
\begin{equation}\label{eq: phi-isomorphism}
\phi\colon\,\reg\,\Sigma^{n-2}_{\blc}\to\mathcal{C}^{n-2}(\mathbb{C}^{*})\,,
\quad \phi(Q)=\widetilde{Q}\,.
\end{equation}
%&
To show that $\phi$ is an isomorphism,
for any $Q'=\{q'_1,...,q'_{n-2}\}\in\mathcal{C}^{n-2}(\mathbb{C}^{*})$
take
%&
$$
v=-\frac{1}{n}\sum\limits_{i=1}^{n-2}q'_i\quad  \textup{and} \quad
Q''=\{q'_1+v,...,q'_{n-2}+v,v,v\}\,;
$$
%&
note that $v=u$ for $Q'=\widetilde{Q}$ as above. Then $Q''\in\reg\,\Sigma^{n-2}_{\blc}$ and the morphism
%&
\begin{equation}\label{eq: psi-isomorphism}
\psi\colon\mathcal{C}^{n-2}(\mathbb{C}^{*})\to\reg\,\Sigma^{n-2}_{\blc}\,, \quad \psi(Q')=Q''\,,
\end{equation}
%&
is the inverse of $\phi$.
\end{proof}

\begin{prop}\label{Prp: rigidity}
For $n>2$, let $\mathcal{X}$ be one of the varieties 
$\mathcal{C}^{n-1}_{\blc}$, $\mathcal{SC}^{n-2}_{\blc}$, or $\Sigma^{n-2}_{\blc}$. 
Then any morphism
$\mathbb{C}\to\reg\,\mathcal{X}$ is constant. In particular, these varieties are rigid, 
and the cylinders $\mathcal{C}^{n-1}_{\blc}\times\mathbb{C}$,\
$\mathcal{SC}^{n-2}_{\blc}\times\mathbb{C}$, and
$\Sigma^{n-2}_{\blc}\times\mathbb{C}$
in {\rm (\ref{eq: 3 cylinders})} are tight.
\end{prop}

\begin{proof} Let us show first that any morphism
$f\colon\,\mathbb{C}\to\mathcal{C}^{n-1}_{\blc}$ is constant. Consider the
unramified $\mathbf{S}(n)$-covering
$p\colon\mathcal{C}^{n-1}_{\blc,\ord}\to\mathcal{C}^{n-1}_{\blc}$. By the 
monodromy theorem $f$ can be lifted to a morphism
$g=(g_1,\ldots,g_n)\colon\,\mathbb{C}\to\mathcal{C}^{n-1}_{\blc,\ord}$.
For any $i\ne{j}$ the regular function $g_i-g_j$ on $\mathbb{C}$ does not
vanish, hence is constant. In particular, $g_i=g_1+c_i$, where
$c_i\in\mathbb{C}$, $i=1,\ldots,n$, and so,
$0=\sum\limits_{i=1}^ng_i=ng_1+c$, where $c=\sum\limits_{i=1}^n c_i$.
Thus, $g_1=\const$, so $g_i=\const$ for all $i=1,\ldots,n$. Hence $f=\const$, 
and the variety $\mathcal{C}^{n-1}_{\blc}$ is rigid.
\vskip4pt

Since $\mathcal{SC}^{n-2}_{\blc}\subset\mathcal{C}^{n-1}_{\blc}$,
any morphism $\mathbb{C}\to\mathcal{SC}^{n-2}_{\blc}$
is constant and $\mathcal{SC}^{n-2}_{\blc}$ is rigid.
\vskip4pt

It remains to show that any  morphism
$\mathbb{C}\to\reg\,\Sigma^{n-2}_{\blc}$ is constant.  For $n=3$ we have
$\reg\,\Sigma^{n-2}_{\blc}\cong\mathbb{C}^*$, hence the claim follows.

For $n\ge{4}$, by Lemma \ref{Lm: reg Sigma(n-2)blc=C(n-2)(C*)}, 
it suffices to show that any morphism $f\colon\,\mathbb{C}\to\mathcal{C}^{n-2}(\mathbb{C}^{*})$ is constant.
By monodromy theorem $f$ admits a lift $g\colon\,\mathbb{C}\to\mathcal{C}^{n-2}_{\ord}(\mathbb{C}^*)
\subset(\mathbb{C}^*)^{n-2}$ to the unramified $\mathbf{S}(n-2)$-covering
$\mathcal{C}^{n-2}_{\ord}(\mathbb{C}^*)\to\mathcal{C}^{n-2}(\mathbb{C}^{*})$.
This implies that both $g$ and $f$ are constant, since any morphism $\mathbb{C}\to\mathbb{C}^*$ is.
\end{proof}

Using Corollary \ref{Cor: automorphisms-of-rigid-cylinder-are-triangular} 
and Proposition \ref{Prp: rigidity}, we can deduce such a corollary.

\bcor\label{Crl: our automorphisms are triangular}
For $n>2$, all automorphisms of the cylinders
$$
\mathcal{C}^n\cong\mathcal{C}^{n-1}_{\blc}\times\mathbb{C},\quad
\mathcal{SC}^{n-1}\cong\mathcal{SC}^{n-2}_{\blc}\times\mathbb{C},
\quad\textup{and}\quad \Sigma^{n-1}\cong\Sigma^{n-2}_{\blc}\times\mathbb{C}
$$
are triangular, and {\rm (\ref{eq: tr-form0})}-{\rm
(\ref{eq: tr-form1})} hold for the corresponding automorphism
groups.
\ecor

\section{The automorphism groups of tight cylinders}\label{sc: ind-grp}
\label{sc: special-Aut}

\subsection{The  structure of the orbits}\label{ss: orbits}
Let $\tau$ stands as before for the standard $\mathbb{C}_+$-action  on
$\mathcal{X}\times\mathbb{C}$ by shifts along the second factor, and let
$U=\exp(\mathbb{C}\p/\p y)$ be the corresponding one-parameter unipotent
subgroup of $\SAut\,(\mathcal{X}\times\mathbb{C})$. Consider also the subgroup 
$B\Def{U}\cdot\Aut\mathcal{X}\cong U\rtimes\Aut\mathcal{X}$ of
$\Aut\,(\mathcal{X}\times\mathbb C)$, and let
$B_0\cong U\rtimes\Aut_0\,\mathcal{X}$ be its neutral component. 

More generally, given a character $\chi$ of $\Aut \mathcal{X}$
(of $\Aut_0\,\mathcal{X}$, respectively) we let 
%&
$$
B(\chi)=\big\{F\in \Aut\,(\mathcal{X}\times\mathbb{C})\,|\,F:(x,y)\mapsto (Sx, \chi(S)y+b),\, 
S\in \Aut \mathcal{X}, \,b\in\mathbb{C}\big\}\,,
$$
%&
 and let $B_0(\chi)$ be the neutral component of $B(\chi)$.
 Thus, $B=B(1)$ and $B_0=B_0(1)$. Clearly, $B(\chi)$ ($B_0(\chi)$, respectively) 
is algebraic as soon as $\Aut \mathcal{X}$  ($\Aut_0\,\mathcal{X}$, respectively) is. 

From Proposition \ref{Prp: nilpotent derivations in rigid cylinder} we deduce the following result.

\bcor\label{cor: orbits-rigid-cylinder} If the cylinder $\mathcal{X}\times\mathbb C$
over an affine variety $\mathcal{X}$ is tight,
then the orbits of the automorphism group
$\Aut (\mathcal{X}\times\mathbb C)$ {\rm (}of $\Aut_0\, (\mathcal{X}\times\mathbb C)$, 
respectively{\rm )} coincide
with the orbits of the group $B(\chi)$  {\rm (}$B_0(\chi)$, respectively{\rm )}, 
whatever is the character $\chi$ of $\Aut \mathcal{X}$ {\rm (}of $\Aut_0 \mathcal{X}$, respectively{\rm )}.
\ecor

\bproof We give a proof for the group $\Aut (\mathcal{X}\times\mathbb C)$; that for 
$\Aut_0\, (\mathcal{X}\times\mathbb C)$ is similar.  Recall that any automorphism $F$ 
of the tight cylinder $\mathcal{X}\times\mathbb C$ is triangular, and so, can be written as
%&
\begin{equation}\label{eq: triangular automorphism of XxC}
F(x,y)=(Sx,A(x)y)\,\ \textup{for any} \
(x,y)\in\mathcal{X}\times\mathbb C\,,
\end{equation}
%&
where $S\in\Aut\mathcal{X} $ and $A\in\Mor(\mathcal{X},\Aff\mathbb C)$.
It follows that the $B(\chi)$-orbit $B(\chi)Q$ of a
point $Q=(x,y)$ in $\mathcal{X}\times\mathbb C$ is
$B(\chi)Q=[(\Aut\mathcal{X})x]\times\mathbb C$. By virtue of Proposition
\ref{Prp: nilpotent derivations in rigid cylinder}
the $\SAut\, (\mathcal{X}\times\mathbb C)$-orbits in
$\mathcal{X}\times\mathbb C$ coincide with the $\tau$-orbits, i.e.,
with the rulings of the cylinder $\mathcal{X}\times\mathbb C$. Now the assertion 
follows from decomposition (\ref{eq: triple
decomposition}) and the isomorphism $\mathcal{O}_+(\mathcal{X})
\cong\SAut(\mathcal{X}\times\mathbb{C})$ of 
Proposition \ref{Crl: SAut for rigid cylinder}.
\eproof

\brem\label{rems: alg-grp=torus} 
Let $\mathcal{X}$ be an affine variety  with  
$\LND(\mathcal{X})=\emptyset$. 
If $\Aut_0\, \mathcal{X}$ is an algebraic group, 
then this is an algebraic torus, see \cite[Lemma 3]{Ii}. 
In this case $B_0(\chi)$ 
is a metabelian linear algebraic group isomorphic to 
a semi-direct product 
$\mathbb{C}_+\rtimes(\mathbb{C}^*)^r$, where $r\ge 0$ 
and $(\mathbb{C}^*)^r$ acts on $\mathbb{C}_+$ via 
multiplication by the character 
$\chi$ of the torus $(\mathbb{C}^*)^r$.
\erem

In the spirit of Tame Map Theorem, the following holds.

\bprop\label{prop: weak-TMT}
Given a tight cylinder $\mathcal{X}\times\mathbb{C}$ 
over an affine variety 
$\mathcal{X}$ and a character $\chi$ of $\Aut \mathcal{X}$, any automorphism $F$ of
$\mathcal{X}\times\mathbb{C}$ 
admits a unique factorization
%&
\be\label{eq: TMT-factorization-rigid-cylinder}
F\colon\mathcal{X}\times\mathbb{C}\stackrel{T\times\{\id\}}{\longrightarrow}
B(\chi)\times
(\mathcal{X}\times\mathbb{C})\stackrel{\alpha}{\longrightarrow}
\mathcal{X}\times\mathbb{C}\,,
\ee
%&
where $\alpha$ stands
for the $B(\chi)$-action on $\mathcal{X}\times\mathbb{C}$, and
$T\colon\mathcal{X}\times\mathbb{C}\to B(\chi)$ is a morphism
with a constant $(\Aut\mathcal{X})$-component. 
\eprop

\bproof   By Corollary \ref{cor: orbits-rigid-cylinder}
for any point $Q=(x,y)\in\mathcal{X}\times\mathbb{C}$ there
exists an element $T(Q)\in B(\chi)$, $T(Q)\colon (x',y')
\mapsto (S(Q)x', \chi(S(Q))y'+f(Q))$ for some $f(Q)\in\mathbb{C}$, such that
\be\label{eq: B-representation}
F(Q)=T(Q)Q=(S(Q)x, \chi(S(Q))y+f(Q))\in\mathcal{X}\times\mathbb{C}\,.
\ee
On the other hand, according to (\ref{eq: triangular automorphism of XxC}),
\be\label{eq: A-representation} F(Q)=(Sx, a(x)y+b(x))\in \mathcal{X}\times\mathbb{C}\,,\ee
where $S\in\Aut\mathcal{X}$, $a\in
\mathcal{O}^\times(\mathcal{X})$, and
$b\in\mathcal{O}_+(\mathcal{X})$ are uniquely determined by $F$.
Comparing  (\ref{eq: B-representation}) and (\ref{eq: A-representation})  yields $S(Q)x=Sx$ 
and $f(Q)=(a(x)-\chi(S))y+b(x)$ for any $Q\in\mathcal{X}\times\mathbb{C}$.
Vice versa, the latter equalities define unique
$f\in\mathcal{O}(\mathcal{X}\times\mathbb{C})$ and
$S\in\Aut\mathcal{X}$ such that $T\colon\mathcal{X}\times\mathbb{C}\to B(\chi)$, 
$Q\mapsto (S, z\mapsto\chi(S)z+f(Q))$,
fits in (\ref{eq: TMT-factorization-rigid-cylinder}) i.e., $F=\alpha\circ (T\times\id)$, as required.  \eproof

Formula (\ref{eq: TMT-T}) corresponds to the particular case where
$\mathcal{X}\times\mathbb{C}=\mathcal{C}^{n-1}_{\blc}\times\mathbb{C}\cong \mathcal{C}^n$. 
In this case $\Aut\mathcal{X}=\Aut\mathcal{C}^{n-1}_{\blc}= \mathbb{C}^*$ 
(see Theorem \ref{Thm: Kaliman-Lin-Zinde}(a)), and the  character 
$\chi:\mathbb{C}^*\to \mathbb{C}^*$ is the identity. 

\subsection{$\Aut (\mathcal{X}\times\mathbb{C})$ as ind-group}
\label{ss: ind-grp}
Recall the following notions (see \cite{Ku}, \cite{Sh}).

\bdefi\label{def: ind-groups}
An {\em ind-group} is a group $G$ equipped with an increasing filtration
$G=\bigcup_{i\in\N} G_i$, where the components $G_i$ are algebraic varieties 
(and not necessarily algebraic groups) such that the natural inclusion 
$G_i\hookrightarrow {G}_{i+1}$, the multiplication map $G_i\times G_j\to G_{m(i,j)}$, 
$(g_i,g_j)\mapsto g_ig_j$, and the inversion $G_i\to G_{k(i)}$, $g_i\mapsto g_i^{-1}$, 
are morphisms for any $i,j\in\N$ with a suitable choice of $m(i,j),\, k(i)\in\N$.
\edefi

\bexas \label{ex: ind-structure on Add}
(a) ({\em Ind-structure on $\mathcal{O}_+(\mathcal{X})$}).
Given an affine variety $\mathcal{X}$ we fix a closed embedding
$\mathcal{X}\hookrightarrow \mathbb{C}^N$. For
$f\in\mathcal{O}_+(\mathcal{X})$ we define its degree $\deg f$ as
the minimal degree of a polynomial extension of $f$ to
$\mathbb{C}^N$. Letting
\begin{equation}\label{eq: Gi}
G_i=\{f\in\mathcal{O}_+(\mathcal{X})\,\vert\,\deg f\le i\}
\end{equation}
we obtain a filtration of the group $\mathcal{O}_+(\mathcal{X})$
by an increasing sequence of connected Abelian algebraic subgroups $G_i$ 
($i\in\N$), hence an ind-structure on $\mathcal{O}_+(\mathcal{X})$.
\medskip

(b) ({\em Ind-structure on $\Aut\mathcal{X}$}).
\label{ex: ind-structure on Aut}
Given a closed embedding
$\mathcal{X}\hookrightarrow \mathbb{C}^N$, any automorphism $F\in
\Aut\mathcal{X}$ can be written as $F=(f_1,\ldots,f_N)$, where
$f_j\in\mathcal{O}_+(\mathcal{X})$. Letting
%&
$$
\deg F= \max_{j=1,\ldots,N}\,\{\deg f_j\}\quad\mbox{and}\quad G_i=
\{F\in  \Aut\mathcal{X}\,\vert\,\deg F\le i\}
$$
%&
we obtain an ind-group structure $\Aut\mathcal{X}=\bigcup_{i\in\N} G_i$
compatible with the action of $\Aut\mathcal{X}$ on $\mathcal{X}$.
The latter means that the maps $G_i\times \mathcal{X}\to
\mathcal{X}$, $(F,x)\mapsto F(x)$, are morphisms of algebraic
varieties. It is well known that any two such ind-structures on
$\Aut\mathcal{X}$ are equivalent.
\medskip

(c) ({\em Ind-structure on} $\Aut(\mathcal{X}\times\mathbb{C})$).
\label{ex:  ind-str  on Aut-cylin}
For a tight cylinder $\mathcal{X}\times\mathbb{C}$ an ind-structure on the group 
$\Aut(\mathcal{X}\times\mathbb{C})$ can be defined via the ind-structures on the 
factors $\mathcal{O}_+(\mathcal{X})$, $\mathcal{O}^\times(\mathcal{X})$, and 
$\Aut \mathcal{X}$ in decomposition (\ref{eq:
triple decomposition}).
\eexas

\subsection{$\mathcal{O}^\times(\mathcal{X})$  as ind-group}
\label{ss: ind-grp Mult} Any extension of an algebraic group by a countable group is an ind-group. 
In particular, the group $\mathcal{O}^\times(\mathcal{X})$ is an ind-group due 
to the following well-known fact (see \cite[Lemme 1]{Samuel}; see also \cite[Ch.\ 3, Lemma 1.2.1]{Mi} 
or \cite[Lemma 1.1]{Ford}).

\begin{lem}[{\em Samuel's Lemma}\rm]\label{lem: Mult-group}
For any irreducible algebraic variety $\mathcal{X}$ defined over an algebraically closed field $k$ we have
$$
\mathcal{O}^\times(\mathcal{X})\cong k^*\times\mathbb{Z}^m \ \
\textup{for some} \ \ m\ge 0\,.
$$
If $k=\mathbb{C}$, then $m\le \rank{H^1(\mathcal{X},\mathbb{Z})}$.
\end{lem}

We provide an argument for $k=\mathbb{C}$, which follows  the sheaf-theoretic proofs of the 
topological Bruschlinsky \cite{Bru} and Eilenberg \cite{Eil38}, \cite{Eil39} theorems.
\vskip3pt

\begin{proof} The sheaves $\mathcal{Z}_\mathcal{X}$, $\mathcal{C}_\mathcal{X}$, 
and $\mathcal{C}^*_\mathcal{X}$ of germs of continuous functions with values in $\mathbb{Z}$, 
$\mathbb{C}$, and $\mathbb{C}^*$, respectively, form the exact sequence
$
0\to\mathcal{Z}_\mathcal{X}\overset{{\boldsymbol\cdot}\,2\pi{i}} 
\longrightarrow\mathcal{C}_\mathcal{X}\overset{\exp}
\longrightarrow\mathcal{C}^*_\mathcal{X}
\to{1}\,.
$
As $\mathcal{X}$ is connected and paracompact,  and the sheaf $\mathcal{C}_\mathcal{X}$ is fine, 
${H^1}(\mathcal{X},\mathcal{C}_\mathcal{X})=0$ and the corresponding exact cohomology
 sequence takes the form
$$
0\to\mathbb{Z}\overset{{\boldsymbol\cdot}\,2\pi{i}} \longrightarrow{C}(\mathcal{X})
\overset{\exp}\longrightarrow{C}^*(\mathcal{X})
\overset{\rho}\longrightarrow{H^1}(\mathcal{X},\mathbb{Z})\to{0}\,.
$$
Restricting the homomorphism $\rho$ to 
$\mathcal{O}^\times(\mathcal{X})\subset\mathcal{C}^*(\mathcal{X})$ 
and taking into account that the conditions $\phi\in\mathcal{O}(\mathcal{X})$ 
and $e^\phi\in\mathcal{O}^\times(\mathcal{X})$ imply $\phi=\const$, 
we obtain the exact sequence
$$
\mathbb{C}\overset{\exp}\longrightarrow\mathcal{O}^\times(\mathcal{X})
\overset{\rho}\longrightarrow{H^1}(\mathcal{X},\mathbb{Z})\,.
$$
Since $H^1(\mathcal{X},\mathbb{Z})$ is a free Abelian group of
finite rank, the image of $\rho$ is isomorphic to $\mathbb{Z}^m$
with some $m\le\rank{H^1(\mathcal{X},\mathbb{Z})}$. This implies
that the Abelian group $\mathcal{O}^\times(\mathcal{X})$ admits
the desired direct decomposition.
\end{proof}

\bexas [{\em The groups of units on the balanced spaces}\rm]
\label{ex: Mult}
(a) The discriminant $D_n$ is the restriction to 
$\mathcal{C}^n=\mathbb{C}^n_{(z)}\setminus\{z\mid\, d_n(z)=0\}$ 
of the irreducible discriminant polynomial $d_n$. Since 
$\mathcal{C}^n=\mathcal{C}^{n-1}_{\blc}\times\mathbb{C}$,
the group $H^1(\mathcal{C}^{n-1}_{\blc},\mathbb{Z})
=H^1(\mathcal{C}^n,\mathbb{Z})\cong\mathbb{Z}$ is generated by the
cohomology class of $D_n$. Any element of
$\mathcal{O}^\times(\mathcal{C}^{n-1}_{\blc})$ is of the form
$sD_n^k$ with $s\in\mathbb{C}^*$ and $k\in\mathbb{Z}$. Hence
$\mathcal{O}^\times(\mathcal{C}^{n-1}_{\blc})
\cong\mathbb{C}^*\times\mathbb{Z}$.
\vskip3pt

(b) The projection $D_n\colon\, \mathcal{C}^{n-1}_{\blc}\to\mathbb{C}^*$,
$Q\mapsto{D_n(Q)}$, is a locally trivial fiber bundle with 
fibers isomorphic to $\mathcal{SC}^{n-2}_{\blc}$. 
Since $\pi_2(\mathbb{C}^*)=0$, the final segment 
of the corresponding long exact sequence of homotopy groups looks as follows:
$$
1\to\pi_1(\mathcal{SC}^{n-2}_{\blc})\to\pi_1(\mathcal{C}^{n-1}_{\blc})
\to\pi_1(\mathbb{C}^*)\to{1}\,.
$$
Now, $\pi_1(\mathcal{C}^{n-1}_{\blc})$ is the Artin braid group 
$\mathbf{A}_{n-1}$, 
and we can rewrite this sequence as
$$
1\to\pi_1(\mathcal{SC}^{n-2}_{\blc})\to{\mathbf{A}_{n-1}}
\to\mathbb{Z}\to{1}\,,
$$
so that the commutator subgroup $\mathbf{A}'_{n-1}$ 
is contained in $\pi_1(\mathcal{SC}^{n-2}_{\blc})$.
Since $\mathbf{A}_{n-1}/\mathbf{A}'_{n-1}\cong\mathbb{Z}$ 
and the torsion of any nontrivial quotient group of 
$\mathbb{Z}$ is nontrivial, it follows that 
$\pi_1(\mathcal{SC}^{n-2}_{\blc})\cong{\mathbf{A}'_{n-1}}$. 
By \cite[Lemma
2.2]{G-L69}, $\mathbf{A}''_{n-1}=\mathbf{A}'_{n-1}$ 
for any $n>4$, and so 
$\Hom(\mathbf{A}'_{n-1},\mathbb{Z})=0$.
Finally, $H^1(\mathcal{SC}^{n-2}_{\blc},\mathbb{Z})
\cong\Hom(\pi_1(\mathcal{SC}^{n-2}_{\blc}),\mathbb{Z})
=\Hom(\mathbf{A}'_{n-1},\mathbb{Z})=0$ and
$\mathcal{O}^\times(\mathcal{SC}^{n-2}_{\blc})\cong\mathbb{C}^*$.
\vskip3pt

(c) The discriminant $d_n$ and its restriction $d_n|_{z_1=0}$ 
to the hyperplane $z_1=0$ are quasi-homogeneous. So, the zero level sets 
$\Sigma^{n-1}=\{d_n=0\}$ and 
$\Sigma^{n-2}_{\blc}=\Sigma^{n-1}\cap\{z_1=0\}$ 
are contractible. Hence $H^1(\Sigma^{n-2}_{\blc},\mathbb{Z})=0$. 
By Lemma \ref{lem: Mult-group},
$\mathcal{O}^\times(\Sigma^{n-2}_{\blc})\cong\mathbb{C}^*$.
\eexas

\subsection{The neutral component 
$\Aut_0\,(\mathcal{X}\times\mathbb{C})$}
\label{ss: neutral component}
\bdefi\label{def: neutral comp}
The {\em neutral component} $G_0$ of an ind-group 
$G=\bigcup_{i\in\N} G_i$ is defined as the  union of those 
connected components of the 
$G_i$ that contain the unity $e_G$ of $G$. In other words, $G_0$ 
is the union of all connected algebraic subvarieties of $G$ passing
through $e_G$. Recall that a subset $V\subset G$ is an algebraic
subvariety if it is a subvariety of some $G_i$. Clearly, $G_0$
is a normal ind-subgroup of $G$.

For an irreducible affine variety $\mathcal{X}$, 
${\Aut}_{0}\,\mathcal{X}$ 
is as well the neutral component
of $\Aut\mathcal{X}$ in the sense of \cite{Ra}.
\edefi

From Corollary \ref{Cor: automorphisms-of-rigid-cylinder-are-triangular},
Lemma \ref{lem: Mult-group}, and decomposition (\ref{eq: triple decomposition}) we derive the following.

\bthm\label{Thm: neutral comp}
For a tight cylinder $\mathcal{X}\times\mathbb{C}$ we have
\begin{equation}\label{eq: triple decomposition0}
{\Aut}_0\, (\mathcal{X}\times\mathbb{C})\cong
\mathcal{O}_+(\mathcal{X})\rtimes(\mathbb{C}^*\times{\Aut}_0\, \mathcal{X})\,.
\end{equation}
\ethm

\bproof For a semi-direct product of two ind-groups $H$ and $H'$
we have $(H\rtimes H')_0=H_0\rtimes H'_0$. Thus, from (\ref{eq:
triple decomposition}) we get a decomposition
%&
$$
{\Aut}_0\, (\mathcal{X}\times\mathbb{C})\cong (\mathcal{O}_+(\mathcal
X)\rtimes \mathbb{C}^*)\rtimes{\Aut}_0\,\mathcal{X}\,.
$$
%&
 It suffices to show that the factors $\mathbb{C}^*$ and ${\Aut}_0\, (\mathcal{X})$
in this decomposition  commute, i.e., that $FF'=F'F$ for any
two automorphisms $F,F'\in {\Aut}_0\, (\mathcal{X}\times\mathbb{C})$ of
the form $F:(x,y)\mapsto (x,ty)$ and $F':(x,y)\mapsto (Sx,y)$,
where $S\in\Aut\mathcal{X}$ and $t\in\mathbb{C}^*$, see (\ref{eq:
triangular automorphism1}). However, the latter equality is
evident. \eproof

\brem [{\em The unipotent radical of} $\Aut_0\, (\mathcal{X}\times\mathbb{C})$]\label{rem: alg-gen}
Due to Proposition \ref{prop: non-tightness}(a), the base $\mathcal{X}$ of 
a tight cylinder $\mathcal{X}\times\mathbb{C}$ (in particular, any rigid variety $\mathcal{X}$) 
does not admit any non-trivial action of a unipotent linear algebraic group. 
Thus, any such subgroup of $\Aut_0\, (\mathcal{X}\times\mathbb{C})$ is contained in the subgroup
$\SAut(\mathcal{X}\times\mathbb{C})\cong\mathcal{O}_+(\mathcal{X})$, see (\ref{eq: triple
decomposition0}), and so, is Abelian. Due to Proposition 
\ref{Crl: SAut for rigid cylinder}, 
the normal Abelian subgroup $\SAut(\mathcal{X}\times\mathbb{C})$ 
can be regarded as the unipotent radical of 
$\Aut_0\, (\mathcal{X}\times\mathbb{C})$. 
Note that $\SAut(\mathcal{X}\times\mathbb{C})$ is 
a union of an increasing sequence of connected algebraic subgroups, 
see Example \ref{ex: ind-structure
on Add}\,(a). We need the following more precise statement.
\erem

\blem\label{lem: normal}
Let $\mathcal{X}\times\mathbb{C}$ be a tight cylinder.
Then the special automorphism group 
$\SAut(\mathcal{X}\times\mathbb{C})\subset\Aut_0\, (\mathcal{X}\times\mathbb{C})$
 is a countable increasing union of connected unipotent algebraic subgroups 
$U_i\subset\SAut(\mathcal{X}\times\mathbb{C})$, which are normal in 
$\Aut_0\, (\mathcal{X}\times\mathbb{C})$.
\elem

\bproof The action of ${\Aut}_0\,\mathcal{X}$ on the normal subgroup 
$\mathcal{O}_+(\mathcal{X})\vartriangleleft
\Aut_0\, (\mathcal{X}\times\mathbb{C})$ in (\ref{eq:
triple decomposition0}) is given by $b\mapsto b\circ{S}$
for $b\in\mathcal{O}_+(\mathcal{X})$ and $S\in{\Aut}_0\,\mathcal{X}$,
cf. the proof of Theorem \ref{Thm: neutral comp}. The
$\mathbb{C}^*$-subgroup in (\ref{eq: triple decomposition0}) acts
on $\mathcal{O}_+(\mathcal X)$ via homotheties $b\mapsto t^{-1}b$,
where $b\in\mathcal{O}_+(\mathcal X)$ and $t\in\mathbb{C}^*$. 
Therefore, the linear representation of the product 
$\mathbb{C}^*\times {\Aut}_0\, \mathcal{X}$ on $\mathcal{O}_+(\mathcal{X})$
 is locally finite. In particular, the finite dimensional subspace 
$G_i=\{f\in\mathcal{O}_+(\mathcal{X})\,\vert\,\deg f\le i\}$ as in (\ref{eq: Gi}) is contained in another finite dimensional subspace, say $U_i$,
which is stable under the action of 
$\mathbb{C}^*\times {\Aut}_0\,\mathcal{X}$, hence is
normal when regarded as a
subgroup of $\Aut_0\, (\mathcal{X}\times\mathbb{C})$. Since the
sequence $(G_i)_{i\in\N}$ is increasing, we can choose the
sequence $(U_i)_{i\in\N}$ being also increasing.
\eproof

\bcor\label{cor: exhaustion} Let $\mathcal{X}\times\mathbb{C}$ be a tight cylinder
over an affine variety $\mathcal{X}$ Suppose that ${\Aut}_0\, \mathcal{X}$ is an algebraic  
group.\footnote{The latter assumption holds if $\bar k(\reg\,\mathcal{X})\ge 0$, where 
$\bar k$ stands for the logarithmic Kodaira dimension. If this is the  case, then 
${\Aut}_0\, \mathcal{ X}$ is an algebraic torus (\cite[Proposition 5]{Ii}).}
Then $\Aut_0\, (\mathcal{X}\times\mathbb{C})=\bigcup_{i\in\N}
B_i$, where $(B_i)_{i\in\N}$ is an increasing sequence of
connected algebraic subgroups.\ecor

\bproof It is enough to let $B_i=U_i\rtimes (\mathbb{C}^*\times{\Aut}_0\,
\mathcal{X})$.\eproof

\subsection{Algebraic subgroups of $\Aut_0\, (\mathcal{ X}\times\mathbb{C})$}
\label{ss: Algebraic subgroups of Aut0}  In this subsection we
keep the assumptions of Corollary \ref{cor: exhaustion}. By this
corollary the group $\Aut_0\, (\mathcal{X}\times\mathbb{C})$ is a union
of connected affine algebraic subgroups. The notions of
semisimple and unipotent elements, and as well of the Jordan
decomposition, are well defined in $\Aut_0\, (\mathcal{X}\times\mathbb{C})$ due to their 
invariance. Moreover, by virtue of
Remark \ref{rem: alg-gen} for any connected affine algebraic
subgroup $G$ of $\Aut_0\, (\mathcal{X}\times\mathbb{C})$, the unipotent
radical of $G$ equals $G\cap \SAut  (\mathcal{X}\times\mathbb{C})$. So
$\SAut (\mathcal{X}\times\mathbb{C})$ is the set of all unipotent
elements of $\Aut_0\, (\mathcal{X}\times\mathbb{C})$. The next result
shows that, under the assumptions of Corollary \ref{cor: exhaustion},  
decomposition (\ref{eq: triple decomposition0}) can be viewed as an analog
of the Mostow decomposition  for algebraic groups. Recall that Mostow's version of the Levi-Malcev 
Theorem  \cite{Mo55} (see also  \cite{Hoch61} or
\cite[ Ch.\ II, \S 1, Theorem 3]{PR}) states that any connected
algebraic group over a field of characteristic zero admits a
decomposition into a semi-direct product of its unipotent radical
and a maximal reductive subgroup. Any two such maximal reductive
subgroups are conjugated via an element of the unipotent radical.

\bthm\label{Thm: Mostow} Let $\mathcal{X}$ be an affine
variety of dimension $>1$ such that $\Aut_0\,\mathcal{X}$ is an algebraic
group. Then the following hold.
\vskip3pt

\textup{(a)} The group ${\Aut}_0\,\mathcal{X}$ is isomorphic to an algebraic torus 
$(\mathbb{C}^*)^r$.
\vskip3pt

\textup{(b)} The group ${\Aut}_0\, (\mathcal{X}\times\mathbb{C})\cong
\mathcal{O}_+(\mathcal{X})\rtimes (\mathbb{C}^*)^{r+1}$ is metabelian.
\vskip3pt

\textup{(c)} Any connected algebraic subgroup $G$ of $ \Aut (\mathcal{X}\times\mathbb{C})$
 is either Abelian or metabelian of rank $\le r+1$.
\vskip3pt

\textup{(d)} Any algebraic torus in \
${\Aut}_0\, (\mathcal{X}\times\mathbb{C})$ is contained in a maximal torus.  
Any maximal torus is of rank  $r+1$, and two such tori are  conjugated via  an element of $\SAut (\mathcal{X}\times\mathbb{C})$.
\vskip3pt

\textup{(e)} Any semisimple element of ${\Aut}_0\,(\mathcal{X} \times\mathbb{C})$ 
is contained 
in a maximal torus. Any finite subgroup of ${\Aut}_0\, (\mathcal{X}\times\mathbb{C})$ 
is Abelian  and contained in a maximal torus.
\ethm

\bproof By our assumptions  ${\Aut}_0\, \mathcal{X}$ is a
connected linear algebraic group without any unipotent subgroup. 
Indeed, assuming that there is such a subgroup $U$, and taking it to be one-dimensional, 
we have $U=\exp(\mathbb{C}\partial)$, where $\partial\in\LND(\mathcal{O}(\mathcal{X}))$. Then
$\mathcal{U}=\exp((\ker\partial)\partial)$ is an infinite dimensional subgroup of 
${\Aut}_0\, \mathcal{X}$, a contradiction. 

Hence by Theorem \ref{cor: BMLC-improvement} 
the  cylinder $\mathcal{X}\times\mathbb{C}$ is  tight, and so, 
${\Aut}_0\, \mathcal{X}$ is an algebraic torus  by \cite[Lemma 3]{Ii}.
This proves (a).

By virtue of (\ref{eq: triple decomposition0}) and (a) we have
\be\label{eq: double decomposition}
{\Aut}_0\,(\mathcal{X}\times\mathbb{C})\cong
\mathcal{O}_+(\mathcal{X})\rtimes(\mathbb{C}^*)^{r+1}\,.
\ee
This proves (b).

By Corollary \ref{cor: exhaustion} the group 
$\Aut_0\,(\mathcal{X}\times\mathbb{C})$ is covered by 
an increasing sequence of connected algebraic subgroups 
$(B_i)_{i\in\mathbf{N}}$, where $B_i=U_i\rtimes (\mathbb{C}^*)^{r+1}$ 
is metabelian.. Any algebraic subgroup 
$G\subset\Aut_0\,(\mathcal{X}\times\mathbb{C})$ 
is contained in one of them, say $G\subset{B_i}$. This proves (c).

Now (d) follows by the classical Mostow Theorem applied to an
appropriate subgroup $B_i$, which contains the tori under
consideration.

 The same argument proves (e). Indeed, both assertions of (e)
hold for connected solvable affine algebraic groups due to
\cite[Ch.\ VII, Proposition 19.4(a)]{Hum}. \eproof

\brem\label{rem: non-algebraic} The assumption that $\Aut_0\,(\mathcal{X})$ is
an algebraic group
is important. For instance, the group $\Aut_0\,(S_n)$ of  the Danielewski surface $S_n$ is non-algebraic,
and (d) does not hold for the cylinder $S_n\times\mathbb{C}$, 
see Example \ref{exa: Danielewski}. Indeed, the group $\Aut\,(S_n\times\mathbb{C})$ contains 
a sequence of pairwise non-conjugate algebraic two-tori (\cite[Thm.\ 2]{Dan89}).
\erem

\subsection{Semisimple and torsion elements}\label{ss: ss-and-torsion elms}
We let $\mathcal{T}$ denote the maximal torus in
$\Aut_0\,(\mathcal{X}\times\mathbb{C})$ which corresponds to the
factor $\mathbb{C}^*\times\Aut_0\,\mathcal{X}\cong
(\mathbb{C}^*)^{r+1}$ under the isomorphisms as in (\ref{eq:
triple decomposition0}) and (\ref{eq: double decomposition}).
 From
Theorem \ref{Thm: Mostow} ((d) and (e)) we deduce the following
corollary.

\bcor\label{cor: torsion-rigid-cylinder}
Under the assumptions of Theorem {\rm \ref{Thm: Mostow}} 
any semisimple $($in particular, any torsion$)$ 
element of the group
$\Aut_0\,(\mathcal{X}\times\mathbb{C})$ is conjugate to an element
of the maximal torus $\mathcal{T}$ via an element of the unipotent
radical $\,\SAut(\mathcal{X}\times\mathbb{C})$. The same
conclusion holds for any finite subgroup of
$\Aut_0\,(\mathcal{X}\times\mathbb{C})$.
\ecor

Using Corollary \ref{cor: torsion-rigid-cylinder} we arrive at the following description of all
semisimple and  torsion elements in the automorphism groups of
tight  cylinders.

\bprop\label{cor: form of ss elements}
Under the assumptions of Theorem {\rm \ref{Thm: Mostow}} an element
$F\in\Aut_0\,(\mathcal{X}\times\mathbb{C})$ is semisimple if and
only if it can be written as
\begin{equation}\label{eq: form  of ss elements}
F\colon(x,y)\mapsto\left(Sx,ty+tb(x)-b(Sx)\right),\quad\mbox{where}\quad 
(x,y)\in\mathcal{X}\times\mathbb{C}\,,
\end{equation}
for some triplet $(S,t,b)$ with $S\in\Aut_0\,\mathcal{X}$,
$t\in\mathbb{C}^*$, and $b\in\mathcal{O}(\mathcal{X})$. Such an
element $F$ is torsion with $F^m=\id$ if and only if $S^m=\id$ and
$t^m=1$. 
\eprop

\subsection{The Lie algebra of $\Aut_0\, (\mathcal{X}\times\mathbb{C})$}
\label{ss: Lie-alg} The Lie algebra of an ind-group is defined in
\cite{Sh}, see also \cite{Ku}. For an ind-group $G$ of type $G=
\underrightarrow{\lim}_i G_i$, where $(G_i)_{i\in\N}$ is an
increasing sequence of connected algebraic subgroups of $G$, the
Lie algebra $\Lie\,(G)$ coincides with the inductive limit $
\underrightarrow{\lim}_i \Lie\,(G_i)\,.$ From Corollary \ref{cor:
exhaustion} and decomposition (\ref{eq: triple decomposition0}) we
deduce the following presentation.

\bthm\label{Thm: Lie-alg} Under the assumptions of Theorem {\rm
\ref{Thm: Mostow}} we have
\begin{equation}\label{eq: Lie-alg}
\Lie\left({\Aut}_0\, (\mathcal{X}\times\mathbb{C})\right)=I\rtimes L\,.
\end{equation}
Here
\begin{equation}\label{eq: I}
I\Def\left\{b(x)\p/\p y\,\vert\,b\in \mathcal{O}_+(\mathcal{X})\right\}=\mathcal{O}_+(\mathcal{X})\p/\p y
\end{equation}
is the Abelian ideal consisting of all locally nilpotent
derivations of the algebra $\mathcal{O}(\mathcal{X}\times\mathbb{C})$,
and
\begin{equation}\label{eq: L}
L\cong\Lie \left(\mathbb{C}^*\times{\Aut}_0\, \mathcal{X}\right)
\end{equation}
is  the  Cartan Lie subalgebra  of $\Lie\,({\Aut}_0\, (\mathcal{X}\times\mathbb{C}))$ corresponding to the second factor in {\rm (\ref{eq: triple
decomposition0})}, i.e., a maximal Abelian
subalgebra
consisting of semisimple elements. Furthermore, we have
\begin{equation}\label{eq: Lie-presentation}
\Lie\left({\Aut}_0\, (\mathcal{X}\times\mathbb{C})\right)=\left\langle
b(x)\p/\p y,\, y\p/\p y,\,\p\,\vert\,b\in 
\mathcal{O}_+(\mathcal{X}),\,\p\in \Lie\,({\Aut}_0\,\mathcal{X}) \right\rangle\,
\end{equation}
with relations
\begin{equation}\label{eq: Lie-relations}
[\p, y\p/\p y]=0,\quad [\p, b\p/\p y]= (\p b)\p/\p
y,\quad\mbox{and}\quad [b\p/\p y, y\p/\p y]=b\p/\p y\,
\end{equation}
for any $b\in\mathcal{O}_+(\mathcal{X})$ and any $\p\in
\Lie\,({\Aut}_0\,\mathcal{X})$. \ethm

\bproof Decomposition (\ref{eq: Lie-alg}) is a direct consequence
of (\ref{eq: triple decomposition0}), and (\ref{eq:
Lie-presentation}) follows from (\ref{eq: triple decomposition0})
and (\ref{eq: Lie-alg}). The first relation in (\ref{eq:
Lie-relations}) follows from the fact that the factors $\mathbb{C}^*$ and
$\Aut_0\,\mathcal{X}$ in (\ref{eq: triple decomposition0})
commute. To show the  other two relations it suffices to verify
these on the functions of the form
$f(x)y^k\in\mathcal{O}(\mathcal{X}\times\mathbb{C})=
\mathcal{O}(\mathcal{X})[y]$, where $k\ge 0$. The latter
computation is easy, and so we omit it.\eproof

\section{Automorphisms of configuration spaces and discriminant levels}
\label{sc: Aut C-n Sigma-n}

\subsection{Automorphisms of balanced spaces}
\label{ss: Automorphisms of balanced spaces}
In view of Corollary \ref{Crl: our automorphisms are triangular},
to compute the automorphism groups of the varieties $\mathcal{C}^{n}$,
$\mathcal{SC}^{n-1}$, and $\Sigma^{n-1}$ we need to know
the automorphism groups of the corresponding balanced spaces 
$\mathcal{C}^{n-1}_{\blc}$, $\mathcal{SC}^{n-2}_{\blc}$, and 
$\Sigma^{n-2}_{\blc}$. The latter groups have been already described in
 the literature. We formulate 
the corresponding results and provide necessary references. 
Then we give a short argument for (a) based on Tame Map Theorem. 
The proof of (c) will be done in Section \ref{Sec: Aut-discriminant}.

\bthm\label{Thm: Kaliman-Lin-Zinde} For any natural $n>2$ the following holds. 

\textup{(a)} $\Aut\mathcal{C}^{n-1}_{\blc}\cong\mathbb{C}^*$.
Any automorphism $S\in\Aut\mathcal{C}^{n-1}_{\blc}$ is of the form\, 
$Q^\circ\mapsto{sQ^\circ}$, where
$Q^\circ\in\mathcal{C}^{n-1}_{\blc}$ and $s\in\mathbb{C}^*$. While 
$\Aut\mathcal{C}^{1}_{\blc}=\Aut\mathbb{C}^*\cong\mathbb{C}^*\rtimes (\Z/2\Z)$.
\vskip3pt

\textup{(b)}
$\Aut\mathcal{SC}^{n-2}_{\blc}\cong\mathbb{Z}/n(n-1)\mathbb{Z}$. 
Any automorphism $S\in\Aut\mathcal{SC}^{n-2}_{\blc}$ is of the form 
$Q^\circ\mapsto{sQ^\circ}$, where $Q^\circ\in\mathcal{SC}^{n-2}_{\blc}$, 
$s\in\mathbb{C}^*$, and $s^{n(n-1)}=1$.
\vskip3pt

\textup{(c)} $\Aut\Sigma^{n-2}_{\blc}\cong\mathbb{C}^*$.
Any automorphism $S\in\Aut\Sigma^{n-2}_{\blc}$ is of
the form $Q^\circ\mapsto{s}Q^\circ$, where $s\in\mathbb{C}^*$ and
every point $Q^\circ\in\Sigma^{n-2}_{\blc}$ is considered as
an unordered multiset $Q^\circ=\{q_1,\ldots,q_n\}\subset\mathbb{C}$
with at least one repetition.
\ethm

For $n>4$ statement (a) is a simple consequence of Tame Map Theorem, 
see \cite{Lin72b} and \cite[Sec. 8.2.1]{Lin79}; we reproduce
a short argument. In Theorem \ref{Thm: properties of C*-tame
maps}(c) we provide a more general result in the analytic setting. 
See also Section \ref{sec: KLZ-thm revisited} for an alternative proof 
avoiding the reference to Tame Map Theorem 
and including the cases $n=3,4$. 

A proof of (b) is sketched in \cite{Kal76a}. Actually, the theorem
of Kaliman (\cite[Theorem]{Kal76a}) says  that {\em for $n\neq 4$ every non-constant
holomorphic endomorphism of $\mathcal{SC}^{n-2}_{\blc}$ is a biregular
automorphism of the above form}. A complete proof of this result can be found in
\cite[Theorem 12.13]{Lin04b}. This proof exploits the
following property of the Artin braid group $\mathbf{A}_{n-1}$ (see
\cite[Theorem 7.7]{Lin96} or \cite[Theorem 8.9]{Lin04a}): {\em
For $n>4$, the intersection 
$\mathbf{A}'_{n-1}\cap{\mathbf{PA}_{n-1}}$ of the commutator
subgroup $\mathbf{A}'_{n-1}$ of $\mathbf{A}_{n-1}$ 
with the pure braid group 
$\mathbf{PA}_{n-1}$ is 
invariant under any 
endomorphism of
$\mathbf{A}'_{n-1}$.} This is no longer true for $n=3, 4$, see Example \ref{ex: Lin's counterexample-2}. 
In case $n=3$, $\mathcal{SC}^{1}_{\blc}$
 is a smooth affine elliptic curve with $j=0$, and, once again, 
its automorphism group 
is as in (b). 
%We do not know whether 
%the aforementioned property of the commutator subgroup  
%$\mathbf{A}'_{n-1}$ holds for $n=4$. Hence 
In case $n=4$ we extend 
Kaliman's Theorem for automorphisms using a different 
approach, see Theorem \ref{thm: Kaliman for n=4}. 
Note that in the case of endomorphisms, the original 
Kaliman's Theorem does not hold if $n=4$; 
see Example \ref{ex: Lin's counterexample}.

Our proof of (c) is based on a part of Tame Map Theorem due to Zinde 
(\cite[Theorems 7 and 8]{Zin78}), which describes the
automorphisms of the configuration space $\mathcal{C}^{n}(\mathbb{C}^*)$. 
Since by Lemma \ref{Lm: reg Sigma(n-2)blc=C(n-2)(C*)}
$\reg\,\Sigma^{n-2}_{\blc}\cong\mathcal{C}^{n-2}(\mathbb{C}^*)$,
from results in [loc. cit.] it follows that for $n>4$
%&
$$
\Aut(\reg\,\Sigma^{n-2}_{\blc})\cong\Aut\mathcal{C}^{n-2}(\mathbb{C}^*)
\cong(\Aut\mathbb{C}^*)\times\mathbb{Z}
\cong(\mathbb{C}^*\times\mathbb{Z})\rtimes(\mathbb{Z}/2\mathbb{Z})\,.
$$
%&
In Theorem \ref{Thm: aut balanced Sigma} we show that
only the elements of the connected component $\mathbb{C}^*$ of
$\Aut(\reg\,\Sigma^{n-2}_{\blc})$ can be extended to automorphisms
of the whole variety $\Sigma^{n-2}_{\blc}$. This implies both
assertions in (c) for $n>4$. In Section \ref{sec:Zinde's Thm.} 
we provide an alternative proof of 
Zinde's Theorem in our particular setting, which does not address 
Tame Map Theorem. We extend the description of the group 
$\Aut\mathcal{C}^{n}(\mathbb{C}^*)$ given in Zinde's Theorem 
to the cases where $n=4$ and $n=3$. 
The theorem does not hold any longer for $n=2$; this case is treated 
in Theorem \ref{thm: Zinde Thm for n=2}. Using the description of the group
$\Aut\mathcal{C}^{2}(\mathbb{C}^*)$ from 
Theorem \ref{thm: Zinde Thm for n=2} we complete the proof of (c) in cases
$n=4$ and $n=3$.
\medskip

\noindent {\em Proof of \textup{(a)} for $n>4$.} The extension $F$ of $S$ to $\mathcal{C}^n$ defined by 
$F(Q)=S(Q-\bc(Q))$ for all $Q\in\mathcal{C}^n$
is a non-Abelian endomorphism of $\mathcal{C}^n$ such that
$F(\mathcal{C}^n)\subset\mathcal{C}^{n-1}_{\blc}$.
By Tame Map Theorem and Remark \ref{Rms: remarks to TMT}\,(b), there is a unique morphism 
$T\colon\,\mathcal{C}^{n}\to\Aff\mathbb{C}$ such that $F=F_T$. Since $F$ preserves  
$\mathcal{C}^{n-1}_{\blc}$, we have $T(Q)(0)=0$ for any $Q\in \mathcal{C}^{n-1}_{\blc}$ 
and hence also for any $Q\in\mathcal{C}^{n}$. So, $T(Q)\zeta=a(Q)\zeta$ for all $\zeta\in\mathbb{C}$ 
and $Q\in\mathcal{C}^{n}$, where $a\in\mathcal{O}^\times(\mathcal{C}^{n})$. According to Example 
\ref{ex: Mult}\,(a), $a=sD_n^k$ for some $s\in\mathbb{C}^*$ and $k\in\mathbb{Z}$, so that 
$S(Q)=sD_n^k(Q)\cdot Q$ on $\mathcal{C}^{n-1}_{\blc}$. %
Similarly, for the inverse automorphism $S^{-1}$ we obtain that $S^{-1}(Q)=tD_n^l\cdot Q$ on 
$\mathcal{C}^{n-1}_{\blc}$ with some $t\in\mathbb{C}^*$ and $l\in\mathbb{Z}$. 
Since $D_n$ is a homogeneous function on $\mathcal{C}^n$ (namely, $D_n(sQ)=s^{n(n-1)}Q$ 
for all $Q\in\mathcal{C}^n$ and $s\in\mathbb{C}^*$), from the identity $S\circ{S}^{-1}=\id$ 
we deduce that $k=l=0$ and $t=s^{-1}$, as required. \hfill $\square$
\medskip

By Theorem \ref{Thm: Kaliman-Lin-Zinde} in all three cases
the automorphism groups of the  corresponding  balanced spaces are
algebraic groups. Hence Theorem \ref{Thm: Mostow} applies and
leads to the following corollary.

\bcor\label{cor: neutral comp} For any $n>2$
the conclusions \textup{(b)-(e)} of Theorem \textup{\ref{Thm: Mostow}} 
hold  with $r=1$ for the groups $\Aut_0\,\mathcal{C}^{n}$
and $\Aut_0\,\Sigma^{n-1}$, and  with $r=0$ for the group
$\Aut_0\,\mathcal{SC}^{n-1}$, when these varieties are viewed as the cylinders in 
{\rm (\ref{eq: 3 cylinders})}.
\ecor

\begin{rem} \label{Rm: orbits of Aut and Aff}
Recall that 
$\Sym^n\mathbb{C}$ regarded as the space of all unordered 
multisets $Q=\{q_1,...,q_n\}\subset\mathbb{C}$ is a disjoint union of 
$\mathcal{C}^n$ and $\Sigma^{n-1}$. 
The tautological  $(\Aff\mathbb{C})$-action on $\mathbb{C}$ 
induces the diagonal $(\Aff\mathbb{C})$-action on 
$\Sym^n\mathbb{C}$; both $\mathcal{C}^n$ and 
$\Sigma^{n-1}$ are invariant under the latter action. It follows from 
Tame Map Theorem and Remark \ref{Rms: remarks to TMT}\,(d) that for $n>2$ the 
$(\Aut\mathcal{C}^{n})$-orbits coincide with the orbits of the diagonal $(\Aff\mathbb{C})$-action 
on $\mathcal{C}^{n}$ (see \cite[Section 2.2]{Lin11}). As follows from Corollaries \ref{cor: orbits-rigid-cylinder}, 
\ref{cor: neutral comp} and Theorems \ref{Thm: Mostow}, \ref{Thm: Kaliman-Lin-Zinde}, for 
$n>2$ the ($\Aut\Sigma^{n-1}$)-orbits coincide with the orbits of the above diagonal 
$(\Aff\mathbb{C})$-action on $\Sigma^{n-1}$. For $n>2$ the ($\Aut\mathcal{SC}^{n-1}$)-orbits 
coincide with the orbits of the subgroup $\mathbb{C}\rtimes (\mathbb{Z}/n(n-1)\mathbb{Z})\subset 
\Aff\mathbb{C}$ acting on $\mathcal{SC}^{n-1}$.
\end{rem}

\subsection{The groups $\Aut\mathcal{C}^{n}$,
$\Aut\mathcal{SC}^{n-1}$, and $\Aut\Sigma^{n-1}$}
\label{ss: main results}
For our favorite varieties $\mathcal{C}^{n}$,
$\mathcal{SC}^{n-1}$, and $\Sigma^{n-1}$ we dispose at present all
necessary ingredients in decomposition (\ref{eq: triple
decomposition}). Gathering this information we obtain the
following description.

\bthm\label{Thm: structure}
If $n>2$, then 
%&
\be\label{eq: decomp-autC}
\Aut\mathcal{C}^n\cong\left(\mathcal{O}_+(\mathcal{C}^{n-1}_{\blc})\rtimes
(\mathbb{C}^*)^2\right)\rtimes \mathbb{Z}\,,
\ee
%&
while for $n=2$,
%&
\be\label{eq: decomp-autC-2}
\Aut\mathcal{C}^2\cong\left(\mathcal{O}_+(\mathbb{C}^{*})\rtimes
(\mathbb{C}^*)^2\right)\rtimes (\mathbb{Z}\rtimes (\mathbb{Z}/2\mathbb{Z}))\,,
\ee
%&
where $\mathcal{O}_+(\mathbb{C}^{*})=\mathbb{C}[t,t^{-1}]$. Furthermore, for $n>2$
\be\label{eq: decomp-autSC}
\Aut\mathcal{SC}^{n-1}\cong\mathcal{O}_+(\mathcal{SC}^{n-2}_{\blc})\rtimes 
\left(\mathbb{C}^*\times (\mathbb{Z}/n(n-1)\mathbb{Z})\right)\,
\ee
%&
and
%&
\be\label{eq: decomp-autSigma}
\Aut\Sigma^{n-1}\cong \mathcal{O}_+(\Sigma^{n-1}_{\blc})\rtimes (\mathbb{C}^*)^2\,.
\ee
%&
All these groups are solvable, and 
$\Aut\mathcal{SC}^{n-1}$ and $\Aut\Sigma^{n-1}$
are metabelian. In addition, any finite subgroup of one of the
groups in {\rm (\ref{eq: decomp-autC})},{\rm  (\ref{eq: decomp-autSC})}, and {\rm (\ref{eq: decomp-autSigma})} is Abelian. 
\ethm

\bproof Likewise this is done
in the proof of Theorem \ref{Thm:
neutral comp}, one can show that the factor $\mathbb{C}^*$ of the group of
units on the corresponding balanced space commutes with the last
factor in (\ref{eq: triple decomposition}). Taking this into
account, the isomorphisms in (\ref{eq: decomp-autC})-(\ref{eq:
decomp-autSigma}) are obtained after substitution of the factors
in (\ref{eq: triple decomposition}) using Examples \ref{ex: Mult}
and Theorem \ref{Thm: Kaliman-Lin-Zinde}.

For the connected group $\Aut \Sigma^{n-1}$ in $(\ref{eq:
decomp-autSigma})$ the last assertion holds due to Theorem
\ref{Thm: Mostow}. The same argument applies in the case of
$\Aut\mathcal{C}^n$. Indeed, the decomposition in (\ref{eq:
decomp-autC}) provides a surjection $\eta:
\Aut\mathcal{C}^n\to\mathbb{Z}$, and any finite subgroup of
$\Aut\mathcal{C}^n$ is contained in the kernel
$\ker\eta=\Aut_0\,\mathcal{C}^n$.

The isomorphism in (\ref{eq: decomp-autSC}) yields a surjection
$\Aut\mathcal{SC}^{n-1}\to \mathbb{C}^*\times (\mathbb{Z}/n(n-1)\mathbb{Z})$ with a torsion free kernel
$\SAut\mathcal{SC}^{n-1}\cong\mathcal{O}_+(\mathcal{SC}^{n-2}_{\blc})$.
 Since any finite subgroup of $\Aut\mathcal{SC}^{n-1}$
meets this kernel just in the neutral element, it injects into the
Abelian group $\mathbb{C}^*\times (\mathbb{Z}/n(n-1)\mathbb{Z})$ and so is Abelian.
\eproof

\subsection{Automorphisms of $\mathcal{C}^{n}$, $\mathcal{SC}^{n-1}$, and
$\Sigma^{n-1}$}\label{ss: presentation}
These varieties can be viewed as subvarieties of the $n$th symmetric power 
$\Sym^n\mathbb{C}=\mathbb{C}^{n}_{(q)}/
\mathbf{S}(n)\cong\mathbb{C}^{n}_{(z)}$. 
The elements of the first two are $n$-point configurations in $\mathbb{C}$, 
while 
the discriminant variety 
$\Sigma^{n-1}=(\Sym^n\mathbb{C})\setminus\mathcal{C}^{n}$ 
consists of all
unordered multisets 
$Q=\{q_1,\ldots,q_n\}\subset\mathbb{C}$ with at least one repetition 
(see Section \ref{Sec: Introduction}). As before, we let $D_n$ be the discriminant viewed 
as a regular function on $\mathcal{C}^{n}$.

\bthm\label{Thm: presentation} Let $n>2$, 
let $\mathcal{Z}$ be one of the varieties $\mathcal{C}^n$, 
$\mathcal{SC}^{n-1}$, and $\Sigma^{n-1}$, and $\mathcal{Z_{\blc}}$ be the corresponding balanced space 
\textup{({\em see} (\ref{eq: balanced spaces}))}. A map $F\colon\mathcal{Z}\to\mathcal{Z}$ 
is an automorphism if and only if
%&
\begin{equation}\label{eq: presentation}
F(Q)=sQ^\circ+a(Q^\circ)\bc(Q)+b(Q^\circ)\ \
\textup{\em for all} \ \, Q\in\mathcal{Z}\,,
\end{equation}
%&
where $Q^\circ=Q-\bc(Q)\in\mathcal{Z}_{\blc}$, $s\in\mathbb{C}^*$, 
$b\in\mathcal{O}_+(\mathcal{Z}_{\blc})$, and
\begin{itemize}
\item[(a)] $a=tD_n^k$ with $t\in\mathbb{C}^*$ and $k\in\mathbb{Z}$,
if \,$\mathcal{Z}=\mathcal{C}^{n}$;
\item[(b)] $a\in\mathbb{C}^*$ and $s^{n(n-1)}=1$, if
\,$\mathcal{Z}=\mathcal{SC}^{n-1}$;
\item[(c)] $a\in\mathbb{C}^*$, if \,$\mathcal{Z}=\Sigma^{n-1}$.
\end{itemize}
\ethm

\bproof
Let $F$ be an automorphism of the cylinder $\mathcal{Z}=\mathcal{Z}_{\blc}\times
\mathbb{C}$ (cf. (\ref{eq: 3 cylinders})). 
According to Corollary \ref{Crl: our automorphisms are triangular}, 
Theorem \ref{Thm: Kaliman-Lin-Zinde}, and Example \ref{ex: Mult}, 
$F$ is triangular of the form
%&
$$
F(Q)=(sQ^\circ,a(Q^\circ)\bc(Q)+b(Q^\circ))
=sQ^\circ+a(Q^\circ)\bc(Q)+b(Q^\circ)\,,
$$
%&
where in each of the cases (a), (b), (c) the triplet ($s, a, b$) is as before. 
Conversely, such an $F$ 
is a (triangular) automorphism of 
 $\mathcal{Z}_{\blc}\times\mathbb{C}=\mathcal{Z}$ 
corresponding to the automorphism 
$S\colon\,Q^\circ\mapsto{s}Q^\circ$ of $\mathcal{Z}_{\blc}$ 
and the morphism 
$A\colon\,\mathcal{Z}_{\blc}\to\Aff\mathbb{C}$, \ 
$A(Q^\circ)\colon\,\zeta\mapsto{a\zeta+b}$ for all 
$Q^\circ\in\mathcal{Z}_{\blc}$ and $\zeta\in\mathbb{C}$.
\eproof

\begin{rems}\label{rm: TMThm} (a) Consider the algebraic torus 
$\mathcal{T}$ of rank $2$ consisting of all transformations
%&
\begin{equation}\label{eq: 2-torus}
\nu(s,t)\colon Q\mapsto s\cdot (Q-\bc(Q))+t\bc(Q)\quad\textup{for
any}\quad (s,t)\in(\mathbb{C}^*)^2 \ \ \textup{and} \
Q\in\Sym^n\mathbb{C}\,.
\end{equation}
%&
Both subvarieties $\mathcal{C}^n$ and $ \Sigma^{n-1}$ in $\Sym^n\mathbb{C}$ 
are invariant under this action. In fact, $\mathcal{T}$ is a maximal torus 
in each of the automorphism groups $\Aut\mathcal{C}^n$ and $\Aut \Sigma^{n-1}$. 
The subgroup of $\mathcal{T}$ given by $s^{n(n-1)}=1$ and isomorphic to 
$(\mathbb{Z}/n(n-1)\mathbb{Z})\times \mathbb{C}^*$ acts on
$\mathcal{SC}^{n-1}$.
\vskip3pt

(b) Let $\mathcal{Z}$ be again one of the varieties $\mathcal{C}^{n}$, 
$\mathcal{SC}^{n-1}$, and $\Sigma^{n-1}$, where  $n>2$, 
and let $\mathcal{Z_{\blc}}$ 
be the corresponding balanced space. 
Using Proposition \ref{Prp: nilpotent derivations in rigid cylinder} 
one can deduce the following:
{\em Any $\mathbb{C}_+$-action on $\mathcal{Z}$ is of the form}
%&
\begin{equation}\label{eq: Cplus-actions}
Q\mapsto Q+\lambda b(Q-\bc(Q))\,, \ \, \textup{\em where} \ \,  
Q\in\mathcal{Z}\,, \ \, \lambda\in\mathbb{C}_+\,, \ \,
b\in\mathcal{O}(\mathcal{Z}_{\blc})\,.
\end{equation}
%&
{\em The case $b=1$ corresponds to the $\tau$-action on $\mathcal{Z}$.}
\vskip3pt

(c) It follows from (\ref{eq: presentation}) 
 that for any $F$ in one of the above groups, one has $F=F_T$ 
with $T$ as in (\ref{eq: TMT-T}).
\end{rems}

\subsection{The group $\Aut\,(\mathbb{C}^n,
\Sigma^{n-1})$}\label{ss:aut rel}
The space $\Sym^n\mathbb{C}\cong\mathcal{C}^n\cup\Sigma^{n-1}$
of all unordered $n$-multisets $Q=\{q_1,...,q_n\}\subset\mathbb{C}$
can be identified with the space $\mathbb{C}^n_{(z)}\cong\mathbb{C}^n$
of all polynomials (\ref{eq: universal polynomial}). The corresponding balanced space 
$\mathcal{C}^{n-1}_{\blc}\cup\Sigma^{n-2}_{\blc}\cong\mathbb{C}^{n-1}$ 
consists of all polynomials $\lambda^n+z_2\lambda^{n-2}+...+z_n$.
An automorphism $F$ of $\mathcal{C}^n$ as in
(\ref{eq: presentation}) extends to an endomorphism of the ambient
affine space $\mathbb{C}^n$ if and only if the rational functions
$a(Q-\bc(Q))$ and $b(Q-\bc(Q))$ on $\mathbb{C}^n$ in (\ref{eq:
presentation}) are regular, i.e.\ $a,b\in\mathcal{O}
(\mathbb{C}^n_{\blc})\cong\mathbb{C}[z_2,\ldots,z_n]$.
Such an endomorphism $F$ admits an inverse, say $F'$, on
$\mathbb{C}^n$ if and only if the corresponding functions $a'$ and
$b'$ are also regular i.e. $a',b'\in\mathcal{O}
(\mathbb{C}^n_{\blc})$. In particular $a=\const\in\mathbb{C}^*$. 
This leads to the following description.

\bthm\label{rel-aut} For any $n>2$ we have
$$
\Aut(\mathbb{C}^n,\Sigma^{n-1})\cong
\mathbb{C}[z_2,\ldots,z_n]\rtimes (\mathbb{C}^*)^2\,,
$$
where the $2$-torus $(\mathbb{C}^*)^2$ and the group 
$\mathbb{C}[z_2,\ldots,z_n]\cong\mathcal{O}_+(\mathbb{C}^{n-1})$ 
act on $\mathbb{C}^n\cong\Sym^n\mathbb{C}$
via {\rm (\ref{eq: 2-torus})} and
{\rm (\ref{eq: Cplus-actions})} with $\lambda=1$, respectively.  
\hfill $\square$
\ethm

\subsection{The Lie algebras $\Lie\,(\Aut_0\,\mathcal{C}^{n})$,
$\Lie\,(\Aut_0\,\mathcal{SC}^{n-1})$, and
$\Lie\,(\Aut_0\,\Sigma^{n-1})$}\label{ss: main Lie results}

\subsubsection{The Lie algebra $\Lie\,(\Aut_0\,\mathcal{C}^n)$}
\label{sss: Lie algebra Cn}
Let $\p_\tau\in\LND(\mathcal{O}(\mathcal{C}^n))$ be the infinitesimal
generator of the $\mathbb{C}_+$-action $\tau$ on
$\mathcal{C}^n\subset\mathbb{C}^n_{(z)}$. By (\ref{eq: Lie-alg}) and
Remark \ref{rm: TMThm}\,(b),
 for $n>2$ there is the Levi-Malcev-Mostow decomposition
%&
$$
\Lie\,({\Aut}_{0}\,\mathcal{C}^n)=I\oplus \Lie \mathcal{T}
$$
%&
with Abelian summands, where
$I=\mathcal{O}(\mathcal{C}^{n-1}_{\blc})\p_\tau$ is as in
(\ref{eq: I}) and the $2$-torus $\mathcal{T}\subset
{\Aut}_{0}\,\mathcal{C}^n$ as in Remark \ref{rm: TMThm}\,(b)
consists of all transformations $\nu(s,t)$ as in (\ref{eq:
2-torus}) with $(s,t)\in (\mathbb{C}^*)^2$ and $Q\in\mathcal{C}^n$. 
Thus $\mathcal{T}$ is the direct product of two its 1-subtori  
with infinitesimal generators, say $\p_s$ and $\p_t$, respectively. 
These derivations are locally finite and locally bounded on 
$\mathcal{O}(\mathcal{C}^n)$, 
and their sum $\p_s+\p_t$ is the infinitesimal generator of the
$\mathbb{C}^*$-action $Q\mapsto\lambda Q$ ($\lambda\in\mathbb{C}^*$) on
$\mathcal{C}^n$.
With this notation we have the following
description.

\bprop\label{pr: brack-Cn} For $n>2$ the Lie algebra 
%&
\begin{equation}\label{eq: 3-generators}
\Lie\,({\Aut}_0\,\mathcal{C}^n)=\langle{I}, \p_s,\p_t\rangle\,,
\quad\mbox{where}\quad I=\mathcal{O}(\mathcal{C}^{n-1}_{\blc})\p_\tau\,,
\end{equation}
is uniquely determined by the commutator relations
%&
\be\label{eq: relations-Cn}
[\p_s,\p_t]=0,\quad
[\p_s,b\p_\tau]=
(\p_sb)\p_\tau,\quad\mbox{and}\quad
[b\p_\tau,\p_t]=b\p_\tau\,,
\ee
%&
where $b$ runs over $\mathcal{O}(\mathcal{C}^{n-1}_{\blc})$.
Furthermore, in the coordinates  $z_1,\ldots,z_n$
 in $\mathbb{C}^n=\mathbb{C}^n_{(z)}$ the derivations $\p_\tau$,
$\p_t$, and $\p_s$ are given by
\be\label{eq: three derivations} \p_\tau=\sum_{i=1}^n
(n-i+1)z_{i-1}\frac{\partial}{\partial z_i},\quad \p_t
=(-z_1/n)\p_\tau\,,\quad\mbox{and}\quad \p_s=\sum_{k=1}^n
kz_k\frac{\p}{\p z_k}-\p_t\,,\ee
%%&
where $z_0\Def 1$.\eprop

\bproof From (\ref{eq: Lie-presentation}) and (\ref{eq:
Lie-relations}) in Theorem \ref{Thm: Lie-alg} we obtain (\ref{eq:
3-generators})  and (\ref{eq: relations-Cn}), respectively. The
diagonal $\mathbb{C}_+$-action $(q_1,\ldots,q_n)\mapsto
(q_1+\lambda,\ldots,q_n+\lambda),\quad\lambda\in\mathbb{C}_+\,,$ on the
affine space $\mathbb{C}^n_{(q)}$ has for infinitesimal generator
the derivation
%&
$$
\partial^{(n)}=\sum_{i=1}^n \frac{\partial}{\partial q_i}\in
{\rm LND}(\mathbb{C}[q_1,\ldots,q_n])\,.
$$
%&
This $\mathbb{C}_+$-action on $\mathbb{C}^n_{(q)}$ descends to the
$\mathbb{C}_+$-action $\tau$ on the base of the Vieta covering
%&
$$
p\colon\,\mathbb{C}^n_{(q)}\to\mathbb{C}^n_{(q)}/\mathbf{S}(n)=
\mathbb{C}^n_{(z)},\quad (q_1,\ldots,q_n)\mapsto
(z_1,\ldots,z_n)\,,
$$
%&
where $z_i=(-1)^i\sigma_i(q_1,\ldots,q_n)$ and $\sigma_i$ is
the elementary symmetric polynomial of degree $i$. We have
$\partial^{(n)}(\sigma_i)=(n-i+1)\sigma_{i-1}$. Hence
in the coordinates $z_1,\ldots,z_n$ on $\mathbb{C}^n_{(z)}$
the infinitesimal generator $\p_{\tau}$ of $\tau$ is given by the first equality in (\ref{eq: three derivations}).

The derivations $\p_s,\p_t$, and $\p_{\tau}$ preserve the subring
$\mathbb{C}[z_1,\ldots,z_n]\subset\mathcal{O}(\mathcal{C}^n)$
and admit natural extensions from $\mathbb{C}[z_1,\ldots,z_n]$ to $\mathbb{C}[q_1,\ldots,q_n]$ 
denoted by the same symbols, where
%&
$$
\p_{\tau} \colon q_i\mapsto 1,\quad \p_t\colon q_i\mapsto
\frac{1}{n}\sum_{k=1}^n q_k,\quad\mbox{and}\quad\p_s\colon
q_i\mapsto q_i-\frac{1}{n}\sum_{k=1}^n q_k,\quad i=1,\ldots,n\,.
$$
It follows that $\p_t=(-z_1/n)\p_t$, which yields the second
equality in (\ref{eq: three derivations}).
%&
Applying these derivations to the coordinate functions
$z_i=(-1)^i\sigma_i(q_1,\ldots,q_n)$ the last equality in
(\ref{eq: three derivations}) follows as well. \eproof

\subsubsection{The Lie algebra $\Lie\,(\Aut_0\,\Sigma^{n-1})$}
\label{sss: Lie Sigma} We have seen in the proof of Proposition
\ref{pr: brack-Cn} that the $\mathbb{C}_+$-action $\tau$ and the action of
the 2-torus $\mathcal{T}$ on $\mathcal{C}^n$ extend regularly to
the ambient affine space $\mathbb{C}^n_{(z)}$, along with the derivations
$\p_{\tau}$, $\p_t$, and $\p_s$ given by (\ref{eq: three derivations}).  The
discriminant $d_n$ on $\mathbb{C}^n_{(z)}$ is invariant under $\tau$.
Hence $\p_\tau d_n=0$, and so, the complete vector field $\p_{\tau}$
is tangent along the level hypersurfaces of $d_n$, in particular,
along $\mathcal{SC}^{n-1}=\{d_n=1\}$ and $\Sigma^{n-1}=\{d_n=0\}$.
The induced locally nilpotent derivations of the structure rings
$\mathcal{O}(\mathcal{SC}^{n-1})$ and $\mathcal{O}(\Sigma^{n-1})$
will be still denoted by $\p_{\tau}$.

The action of the 2-torus $\mathcal{T}$ on $\mathbb{C}^n_{(z)}$ stabilizes
$\Sigma^{n-1}$. Hence $\p_t$ and $\p_s$ generate commuting
semisimple derivations of $\mathcal{O}(\Sigma^{n-1})$ denoted by
the same symbols. Using these observations and notation we can
deduce  from Theorem \ref{Thm: Lie-alg}  and Corollary \ref{cor:
neutral comp}  the following description  (cf.\ \cite{Lya78}).

\bprop\label{pr: brack-Sigma1}
For $n>2$ the Lie algebra 
%&
$$
\Lie\,({\Aut}_{0}\,\Sigma^{n-1})=\langle I, \p_s,\p_t\rangle\,,\quad\mbox{where}\quad I=
\mathcal{O}_+(\Sigma^{n-2}_{\blc})\p_{\tau}\,,
$$
is uniquely determined   by relations {\rm (\ref{eq: relations-Cn})}, where $b$
runs over $\mathcal{O}_+(\Sigma^{n-2}_{\blc})$.
\eprop

\bproof
The proof goes along the same lines as that of Proposition
\ref{pr: brack-Cn}, and so we leave it to the reader.
\eproof

\subsubsection{The Lie algebra $\Lie\,(\Aut_0\,\mathcal{SC}^{n-1})$}
\label{sss: Lie SC}

Since $\p_\tau d_n=0$, for the derivation $\p_t =(-z_1/n)\p_\tau$
(see (\ref{eq: three derivations})) we have $\p_td_n=0$. Hence the vector field
$\p_t$ is tangent  as well to each of the level hypersurfaces of
$d_n$. In particular, $\p_t$ induces a semisimple derivation of
$\mathcal{O}(\mathcal{SC}^{n-1})$ (denoted again by $\p_t$) and
generates a $\mathbb{C}^*$-action $T$ on $\mathcal{SC}^{n-1}$.  So we
arrive at  the following description.

\bprop\label{pr: brack-Sigma}
For $n>2$ the Lie algebra 
%&
$$
\Lie\,({\Aut}_{0}\,\mathcal{SC}^{n-1})=\langle I,\p_t\rangle\,,\quad\mbox{where}\quad I=
\mathcal{O}_+(\mathcal{SC}^{n-2}_{\blc})\p_{\tau}\,,
$$
is uniquely determined  by the relations $[b\p_\tau,\p_t]=b\p_\tau$,
where $b$ runs over
$\mathcal{O}_+(\mathcal{SC}^{n-2}_{\blc})$.\eprop

\bproof
This follows from Theorem \ref{Thm: Lie-alg} and Corollary
\ref{cor: neutral comp} in the same way as before. We leave the
details to the reader.
\eproof

\section{More on the group $\Aut\,(\mathcal{X}\times\mathbb{C})$}
\label{sc: Aut C-n}
\subsection{The center of $\Aut\,(\mathcal{X}\times\mathbb{C})$}
\label{ss: center}

The following lemma provides a formula for the commutator of two
triangular automorphisms of a product  $\mathcal{X}\times \mathbb{C}$. We
let
%&
$$F=F(S,A) \colon (x,y)\mapsto(Sx,A(x)y)\,,$$
%&
where $(x,y)\in\mathcal{X}\times \mathbb{C}$, $S\in\Aut \mathcal{X}$, and
$A:\mathcal{X}\to \Aff \mathbb{C}$ (cf.\ (\ref{eq: triangular
automorphism of XxC})).

\blem\label{l: commutator}
Suppose that the group $\Aut \mathcal{X}$ is
Abelian. Then for any $F=F(S,A)$ and $F'=F(S',A')$ in
$\Aut_{\vartriangle}\,(\mathcal{X}\times \mathbb{C})$ and any
$(x,y)\in\mathcal{X}\times \mathbb{C}$ we have
%&
\begin{equation}\label{eq: commutator}
\aligned {[F',F]}(x,y)=(x,
(A'(x))^{-1}(A(S'x))^{-1}A'(Sx)A(x))y)\,.
\endaligned
\end{equation}
%&
Consequently, $F$ and $F'$ commute if and only if
%&
\be\label{eq: commuting} A(S'x)A'(x)=A'(S x)A(x)\quad \mbox{for
any}\quad x\in\mathcal{X}\,.\ee
%&
\elem

\bproof The proof is straightforward.\eproof

\smallskip

Applying this lemma to general cylinders we deduce the following
facts.

\begin{prop}\label{pr: center}
Let $\mathcal{X}$ be an affine variety. If the group $\Aut
\mathcal{X}$ is Abelian then the center of the group
$\Aut_{\vartriangle}\,(\mathcal{X}\times\mathbb{C})$ is trivial.
The same conclusion holds for the groups $\Aut\mathcal{C}^{n}$, 
$\Aut\mathcal{SC}^{n-1}$, and $\Aut\Sigma^{n-1}$, where  $n>2$.
\end{prop}

\bproof
Consider two elements $F=F(S,A)$ and $F'=F(S',A')$ in
$\Aut_{\vartriangle}\,(\mathcal{X}\times\mathbb{C})$, where
$S,S'\in\Aut\mathcal{X}$ and  $A\colon y\mapsto ay+b$, $A'\colon
y\mapsto a'y+b'$ with $a, a'\in\mathcal{O}^\times (\mathcal{X})$
and $b,b'\in\mathcal{O}_+ (\mathcal{X})$. If $F$ and $F'$ commute
then (\ref{eq: commuting}) is equivalent to the system
%&
\be\label{eq: commuting1}
(a\circ S')\cdot a'=a\cdot (a'\circ S)
\quad\mbox{and}\quad (a'\circ S)\cdot b+b'\circ S=(a\circ S')\cdot
b'+b\circ S'\,.
\ee
%&
Assume that $F$ is a central element i.e. (\ref{eq: commuting1})
holds for any $F'$. Letting in the second relation $S'=\id$,
$a'=2$, and $b'=0$ yields $b=0$. Now this relation reduces to
$$
b'\circ S=(a\circ S')\cdot{b'}\,.
$$
Letting $b'=1$ yields $a=1$ and so $A=\id$ and $b'\circ S=b'$ for
any $b'\in\mathcal{O}_+ (\mathcal{X})$. If $S\neq\id$ this leads
to a contradiction, provided that $b'$ is non-constant on an
$S$-orbit in $\mathcal{X}$. Hence $S=\id$, and so, $F=\id$, as
claimed.

The last assertion follows now from Corollary
\ref{Cor: automorphisms-of-rigid-cylinder-are-triangular},
since the bases of the cylinders $\mathcal{C}^{n},\,\mathcal{SC}^{n-1}$, and 
$\Sigma^{n-1}$ are rigid varieties with Abelian automorphism groups, see (\ref{eq: 3 cylinders}), 
Proposition \ref{Prp: rigidity}, and Theorem \ref{Thm: Kaliman-Lin-Zinde}.
\eproof

\subsection{Commutator series}\label{ss: commutator}
Let us introduce the following notation.
\smallskip

\noindent {\bf Notation.} Let $\mathcal{X}$ be an affine variety with a tight cylinder 
$\mathcal{X}\times\mathbb{C}$, and let $\mathcal{D}
\subseteq\Aut(\mathcal{X}\times\mathbb{C})$ 
be the subgroup consisting of all automorphisms of the form $F=F(\id, A)$,
where $A\colon{y\mapsto}ty+b$ with $t\in\mathbb{C}^*$ and
$b\in\mathcal{O}_+(\mathcal{X})$. It is easily seen that
$$
\SAut\,(\mathcal{X}\times\mathbb{C})\vartriangleleft
\mathcal{D}\vartriangleleft {\Aut}_0\,(\mathcal{X}\times\mathbb{C})\,.
$$
Furthermore, $\mathcal{D}\cong
\mathcal{O}_+(\mathcal{X})\rtimes\mathbb{C}^*$ under the isomorphism in
(\ref{eq: triple decomposition0}), with quotient group
${\Aut}_0\,(\mathcal{X}\times\mathbb{C})/\mathcal{D}\cong{\Aut}_0\,
\mathcal{X}$. In particular, for
$\mathcal{X}\times\mathbb{C}\cong\mathcal{SC}^{n-1}$, $n>2$, we have
$\mathcal{D}=\Aut_0\,(\mathcal{X}\times\mathbb{C})$, see
Corollary \ref{Cor: automorphisms-of-rigid-cylinder-are-triangular}
and Theorem \ref{Thm: structure}.

It is known (\cite[Ch.\ VII, Theorem 19.3(a)]{Hum}) that for any
connected solvable affine algebraic group $G$ the commutator
subgroup $[G,G]$ is contained in the unipotent radical $G_{\rm u}$
of $G$. In our setting a similar result holds.

\bthm\label{Thm: com-ser}
Let $\mathcal{X}\not\cong\mathbb{C}$ be an affine variety such that $\Aut_0\,\mathcal{X}$ is an 
algebraic group. Then 
%&
\be\label{eq: commutators equalities0}
[{\Aut}_0\,(\mathcal{X}\times\mathbb{C}),\,
{\Aut}_0\,(\mathcal{X}\times\mathbb{C})] 
=[\mathcal{D},\mathcal{D}]=\SAut\,(\mathcal{X}\times\mathbb{C})\,.
\ee
%&
Consequently, the commutator series of the group
${\Aut}_0\,(\mathcal{X}\times\mathbb{C})$ is
%&
\be\label{eq: commutator series-0}
1\vartriangleleft \SAut(\mathcal{X}\times\mathbb{C})\vartriangleleft
{\Aut}_0\, (\mathcal{X}\times\mathbb{C}) \,.\ee
%&
The same conclusions hold for the groups 
$\Aut_0\,\mathcal{C}^{n}$, 
$\Aut_0\,\mathcal{SC}^{n-1}$, and
$\Aut\Sigma^{n-1}=\Aut_0\,\Sigma^{n-1}$, where $n>2$.
\ethm

\bproof Since the  group $\Aut_0\,\mathcal{X}$ is  
algebraic and $\mathcal{X}\not\cong\mathbb{C}$, 
this group does not contain any 
unipotent one-parameter subgroup. Hence by \cite[Lemma 3]{Ii} 
 $\Aut_0\,\mathcal{X}$ is an algebraic torus, hence is Abelian. 
By (\ref{eq: triple decomposition0}),
$\SAut\,(\mathcal{X}\times\mathbb{C})\vartriangleleft
{\Aut}_0\,(\mathcal{X}\times\mathbb{C})$ is a normal subgroup with the
Abelian quotient
%&
$$
{\Aut}_0\,(\mathcal{X}\times\mathbb{C})/\SAut\,(\mathcal{X}
\times\mathbb{C})\cong
\mathbb{C}^*\times {\Aut}_0\,\mathcal{X}\,.
$$
%&
Hence $[{\Aut}_0\,(\mathcal{X}\times\mathbb{C}),\,
{\Aut}_0\,(\mathcal{X}\times\mathbb{C})]\subset
\SAut\,(\mathcal{X}\times\mathbb{C})$.
To show (\ref{eq: commutators equalities0}) it suffices to
establish the inclusion $\SAut\,(\mathcal{X}\times\mathbb{C})\subset
[\mathcal{D}, \mathcal{D}]$. However, by virtue of (\ref{eq:
commutator}) any $F=F(\id,A)\in\SAut\,(\mathcal{X}\times\mathbb{C})$,
where $A:y\mapsto y+b$ with $b\in\mathcal{O}_+(\mathcal{X})$, can
be written as commutator $F=[F',F'']$, where $F'=F(\id, A')$ and
$F''=F(\id, A'')$ with
$A':y\mapsto -y-b/2$ and $A'':y\mapsto{y+b/2}$ (in fact, $A=[A',A'']$).

Now (\ref{eq: commutator series-0}) follows from (\ref{eq:
commutators equalities0}), since the group
$\SAut\,(\mathcal{X}\times\mathbb{C})\cong\mathcal{O}_+(\mathcal{X})$ is
Abelian.

By Theorem \ref{Thm: Kaliman-Lin-Zinde}, for $n>2$ the groups 
${\Aut}_0\,\mathcal{C}^n$, 
$\Aut_0\,\mathcal{SC}^{n-1}$, 
and $\Aut\Sigma^{n-1}=\Aut_0\Sigma^{n-1}$ satisfy our assumptions. 
So,  the conclusions hold also for these groups.
\eproof

For the group $\Aut\mathcal{SC}^{n-1}$  the following hold.

\bthm\label{Thm: commutator subgroup-SCn-1}
For $n>2$  we have
$[\Aut\mathcal{SC}^{n-1},\Aut\mathcal{SC}^{n-1}]=\SAut\mathcal{SC}^{n-1}$. 
Hence the commutator series of $\Aut\mathcal{SC}^{n-1}$ is
$$
1\vartriangleleft\SAut\mathcal{SC}^{n-1}\vartriangleleft
\Aut \mathcal{SC}^{n-1}
$$
with the Abelian normal subgroup
$\SAut\mathcal{SC}^{n-1}\cong\mathcal{O}_+( \mathcal{SC}^{n-2}_{\blc})$
and the Abelian quotient group
\begin{equation}\label{eq: factor}
\Aut\mathcal{SC}^{n-1}/\SAut \mathcal{SC}^{n-1}\cong
\mathbb{C}^*\times(\mathbb{Z}/n(n-1)\mathbb{Z})\,.
\end{equation}
\ethm

\bproof
Theorem \ref{Thm: com-ser} yields the inclusion
$\SAut\mathcal{SC}^{n-1}\subset [\Aut\mathcal{SC}^{n-1},\Aut\mathcal{SC}^{n-1}]$. The opposite
inclusion follows from (\ref{eq: factor}), which is in turn a
consequence of Theorem \ref{Thm: structure}. Hence the assertions
follow.
\eproof

Consider further the group $\Aut\mathcal{C}^n$. Note that the
quotient groups
\be\label{eq: factors}
(\Aut\mathcal{C}^n)/\mathcal{D}\cong\mathbb{C}^*\times
\mathbb{Z}\quad\mbox{and}\quad
\mathcal{D}/\SAut\mathcal{C}^n\cong\mathbb{C}^*\,
\, \ee
are Abelian, see Theorem \ref{Thm: structure}. Hence
\be\label{eq: incl}
[\Aut\mathcal{C}^n,\Aut\mathcal{C}^n]\subseteq
\mathcal{D},\quad\mbox{where}\quad [\mathcal{D}, \mathcal{D}] =
\SAut\mathcal{C}^n\,,
\ee
see Theorem \ref{Thm: com-ser}. More precisely, the following holds.

\bthm\label{Thm: commutator subgroup}
For $n>2$  we have
$[\Aut \mathcal{C}^n,\Aut \mathcal{C}^n]=\mathcal{D}$. Hence the
commutator series of the group $\Aut \mathcal{C}^n$ is
$$1\vartriangleleft \SAut\mathcal{C}^n\vartriangleleft  \mathcal{D}
\vartriangleleft  \Aut \mathcal{C}^n\,$$
with Abelian quotient groups, see {\rm (\ref{eq: factors})}.
\ethm

\bproof
By virtue of (\ref{eq: incl}) to establish the first
equality it suffices  to prove the inclusion $\mathcal{D}\subseteq
[\Aut\mathcal{C}^n,\Aut\mathcal{C}^n]$. We show that,
moreover, any element $F_0\in\mathcal{D}$ is a product of two
commutators in $\Aut\mathcal{C}^n$.

Indeed, choosing as before $F$ and $F'$ in $\mathcal{D}$ such that
$[F',F]\colon{Q\to}Q+b(Q)$ and replacing $F_0$ by $[F',F]^{-1}F_0$ we
may suppose that $F_0=F(\id, A_0)$, where $A_0\colon{y\mapsto}ty$ with
$t\in\mathbb{C}^*$.

Let $\widetilde{F}=F(S,A)$ and $\widetilde{F}'=F(S',A')$, where
%&
$$
S\colon Q^\circ\mapsto{sQ^\circ},\ \, S'\colon
Q^\circ\mapsto{s'Q^\circ},\ \,\mbox{and}\ \, A(Q^\circ)\colon y
\mapsto{sD_n^k(Q^\circ)y},\ \, A'(Q^\circ)\colon 
y\mapsto{s'D_n^{k'}}(Q^\circ)y
$$
%&
for $Q^\circ\in\mathcal{C}_{\blc}^{n-1}$ and $y\in\mathbb{C}$.
By (\ref{eq: commutator}) we obtain
%&
$$
{[\widetilde{F}',\widetilde{F}]}(Q)=F({\id},A''),\quad\mbox{where} \quad A'':y\mapsto
(s^{k'}{s'}^{-k})^{n(n-1)}y
$$ 
%&
does not depend on $Q^\circ\in\mathcal{C}_{\blc}^{n-1}$.
%&
Given $t\in\mathbb{C}^*$ we can find $s,s'\in\mathbb{C}^*$ and
$k,k'\in\mathbb{Z}$ such that $(s^{k'}{s'}^{-k})^{n(n-1)}=t$.
With this choice, $F_0=[\widetilde{F}',\widetilde{F}]$ and we are
done.
\eproof

\subsection{Torsion in $\Aut_0\,(\mathcal{Z}_{\blc}\times\mathbb{C})$}
 Let $\mathcal{Z}$ be one of the varieties $\mathcal{C}^n$, 
$\mathcal{SC}^{n-1}$, or $\Sigma^{n-1}$, where $n>2$, 
and let $\mathcal{Z_{\blc}}$ be the corresponding balanced variety. 
Denote by $G_\mathcal{Z}$ one of the groups
$\Aut\mathcal{C}^n$, $\Aut_0\,\mathcal{SC}^{n-1}$, and $\Aut\,\Sigma^{n-1}$. 
With this notation we have the following results.

\bthm\label{Thm: torsion-config-spaces}
The semisimple elements of the group $G_\mathcal{Z}$ 
are precisely the automorphisms of the form
%&
\begin{equation}\label{eq: torsion-config-spaces}
F\colon Q\mapsto{sQ^\circ+t\cdot\bc (Q)+t\cdot{b(Q^\circ)}-b(sQ^\circ)}
\quad\textup{for all}\ \,Q\in\mathcal{Z}\,,
\end{equation}
%&
where $Q^\circ=Q-\bc(Q)\in\mathcal{Z}_{\blc}$,  
$b\in\mathcal{O}_+(\mathcal{Z}_{\blc})$, and $s,t\in\mathbb{C}^*$, with
$s^{n(n-1)}=1$ when $\mathcal{Z}=\mathcal{SC}^{n-1}$. 
Such an element $F$ is torsion with $F^m=\id$ if and only if, in addition, 
$s^m=t^m=1$.
\ethm

\bproof Since
$\Aut\mathcal{C}^n/{\Aut}_0\,\mathcal{C}^n\cong\mathbb{Z}$, the
torsion elements of $\Aut\mathcal{C}^n$ are that of the neutral
component $\Aut_0\,\mathcal{C}^n$. Taking into account that the
group $\Aut\,\Sigma^{n-1}$ $(n>2)$ is connected, in all three
cases Proposition \ref{cor: form of ss elements} applies and yields
the result after a simple calculation using the description in
(\ref{eq: 2-torus}) and (\ref{eq: Cplus-actions}). \eproof

In Example \ref{Example: finite order automorphisms} we
construct some particular torsion elements of the group
$\Aut\mathcal{C}^n$. We show that for any
$b\in\mathcal{O}_+(\mathcal{C}^{n-1}_{\blc})$ there is an element
$F\in\Tors(\Aut\mathcal{C}^n)$ of the form
%&
\be\label{eq: torsion model} F\colon\,Q\mapsto{sQ^\circ}+t\bc(Q)
+b(Q^\circ),\quad\textup{where}\quad Q^\circ=Q-\bc(Q)\,.
\ee
%&
We use the following lemma. Its proof proceeds by induction on
$m$; we leave the details to the reader.

\begin{lem}\label{Lm: F{a,t,0,b}m}
Let $n>2$, and let $F\in{\Aut}_0\,\mathcal{C}^n$ be given by
\textup{(\ref{eq: torsion model})}. Letting $Q^\circ=Q-\bc(Q)$ for any
$m\in\mathbb{N}$ we have
%&
$$
F^m(Q)=s^m Q^\circ+t^m\bc(Q)+\sum_{j=0}^{m-1}
t^{m-j-1}b(s^jQ^\circ)\,.
$$
Consequently, $F^m=\id$ if and only if $s^m=t^m=1$ and the
function $b$ satisfies the equation
%&
\begin{equation}\label{eq: parameters of finite order automorphism}
\sum_{j=0}^{m-1} t^{m-j-1}b(s^jQ^\circ)=0 \ \ \textup{for any} \
\, {Q^\circ}\in\mathcal{C}^{n-1}_{\blc}\,.
\end{equation}
%&
\end{lem}

\brem\label{rem: comparison} It follows from Lemma \ref{Lm:
F{a,t,0,b}m} and Theorem \ref{Thm: torsion-config-spaces} that for $m\ge 2$ a
function $b\in\mathcal{O}(\mathcal{C}^{n-1}_{\blc})$ satisfies
(\ref{eq: parameters of finite order automorphism}) for a given pair  $(s,t)$ of
$m$th roots of unity if and only if it can be written
as $b(Q^\circ)=t{\widetilde b}(Q^\circ)-{\widetilde b}(sQ^\circ)$
for some ${\widetilde b}\in
\mathcal{O}(\mathcal{C}^{n-1}_{\blc})$. The inversion formula
%&
\begin{equation}\label{eq: inversion}
{\widetilde b}(Q^\circ)=
\sum_{j=0}^{m-1} \frac{m-j}{m} t^{m-j}b(s^{j-1}Q^\circ)\ \
\textup{for any} \
\, {Q^\circ}\in\mathcal{C}^{n-1}_{\blc}\,
\end{equation}
%&
allows to find such a function ${\widetilde b}\in \mathcal{O}(\mathcal{C}^{n-1}_{\blc})$ 
for a given solution $b\in \mathcal{O}(\mathcal{C}^{n-1}_{\blc})$ of
(\ref{eq: parameters of finite order automorphism}).
\erem

\begin{exas}[{\em Automorphisms of $\mathcal{C}^n$ of finite order\rm}]
\label{Example:
finite order automorphisms} \textup{(a)} For $m>1$, pick any
$b\in\mathcal{O}(\mathcal{C}^{n-1}_{\blc})$ and any $m$th root of
unity $t\ne{1}$. Then the automorphism $F\colon\,Q\mapsto
(Q-\bc(Q))+t\bc(Q)+b(Q-bc(Q))$ satisfies $F\ne\id$ and $F^m=\id$.
\vskip3pt

\textup{(b)} Let $b\in\mathcal{O}(\mathcal{C}^{n})$ be invariant
under the diagonal $(\Aff\mathbb{C})$-action on $\mathcal{C}^n$.
For instance,
$\displaystyle
b(Q)=c{D}_n^{-k}(Q)\sum\limits_{q',\,q''\in Q} (q'-q'')^{kn(n-1)}$
is such a function  for any $k\in\mathbb{N}$ and $c\in\mathbb{C}$.
Take any $m>2$, and let $s$ and $t$ be two distinct $m$th roots of
unity, where $t\neq 1$. Then the automorphism $F$ as in
(\ref{eq: torsion model}) satisfies $F\ne\id$ and $F^m=\id$.
\vskip3pt

\textup{(c)} Any automorphism $F$ of $\mathcal{C}^n$ of the form
$F\colon\,Q\mapsto -Q+b(Q-\bc(Q))$ with
$b\in\mathcal{O}(\mathcal{C}^{n-1}_{\blc})$ is an involution. For
instance, one can take 
%&
$$
b(Q)=cD_n^r(Q)\sum\limits_{\{q',\,q''\}\subset Q}
(q'-q'')^{2m}\quad\mbox{with}\quad r,m\in\mathbb{Z},\quad
|r|+|m|>0,\quad\mbox{and}\quad c\in\mathbb{C}\,.
$$
%&
\end{exas}

\section{The group $\Aut\mathcal{C}^n(\mathbb{C}^*)$: 
Zinde's Theorem }\label{sec:Zinde's Thm.}

In this section we provide a description of the group 
$\Aut\mathcal{C}^n(\mathbb{C}^*)$ according to Zinde \cite{Zin78}. 
This description will be used in the next section. 
For the convenience of the reader, we give an alternative proof in our setting. 
The original Zinde's Theorem contains a description 
of the biholomorphic automorphisms of $\mathcal{C}^n(\mathbb{C}^*)$ for 
$n>4$, 
and the structure of the group of all  such automorphisms. 
The proof in \cite{Zin78} is based on Zinde's part of Tame Map Theorem  
(see  the Introduction), 
which says that for $n>4$ any non-Abelian 
holomorphic endomorphism of $\mathcal{C}^n(\mathbb{C}^*)$ is
 tame (see  \cite[Theorem 7]{Zin78}). 
Our approach is quite different. 
Indeed, we deal with biregular automorphisms only. In this particular case, 
for any $n>2$ 
we provide a proof which does not refer to Tame Map Theorem. 
However, the starting point of both proofs is the same, see Corollary \ref{lem: Zin77}. 

To formulate our extension of Zinde's Theorem, we need a portion of notation.
Consider the function $h_n\in\mathcal{O}^\times(\mathcal{C}^{n}(\mathbb{C}^*))$, where
$h_n(Q)=D_n(Q)/(q_1\cdot...\cdot{q_n})^{n-1}$. It is invariant under 
the diagonal action of $\Aut\mathbb{C}^*$ 
on $\mathcal{C}^{n}(\mathbb{C}^*)$.
For $\epsilon\in\{1,-1\}$ and $Q=\{q_1,...,q_n\}\in\mathcal{C}^{n}(\mathbb{C}^*)$ we let 
$Q^\epsilon=\{q_1^\epsilon,\ldots,q_n^\epsilon\}$.
With this notation, we have the following description of the group 
$\Aut\mathcal{C}^n(\mathbb{C}^*)$.
\vskip4pt

\noindent {\bf Zinde's Theorem} (\cite[Theorem 8]{Zin78}).
{\em Let $n>2$. A map $F\colon\,\mathcal{C}^{n}(\mathbb{C}^*)
\to\mathcal{C}^{n}(\mathbb{C}^*)$ 
is an automorphism if and only if there exist $\epsilon\in\{1,-1\}$, 
$s\in\mathbb{C}^*$, and $k\in\mathbb{Z}$ 
such that
%&
\begin{equation}\label{eq: aut-form CnCstar}
F(Q)=sh_n^k(Q)Q^\epsilon \ \, \textup{\em for all} \ \, Q
\in\mathcal{C}^{n}(\mathbb{C}^*)\,.
\end{equation}
Furthermore, for $n>2$
$$\Aut\mathcal{C}^{n}(\mathbb{C}^*)\cong 
(\mathbb{C}^*\times \mathbb{Z})\rtimes(\mathbb{Z}/2\mathbb{Z})\,, $$
where the factors $\Z$ and $\Z/2\Z$ commute,   while for $n=2$
%&
$$\Aut \mathcal{C}^{2}(\mathbb{C}^*)
\cong \left((\mathbb{C}^*\times \mathbb{Z})\rtimes(\mathbb{Z}/2\mathbb{Z})\right)
\rtimes (\mathbb{Z}/2\mathbb{Z})\,.$$
%&
}

Thus, the case $n=2$ is exceptional: the automorphisms as in 
{\rm (\ref{eq: aut-form CnCstar}) }
form a subgroup of index 2 in the group 
$\Aut\mathcal{C}^{2}(\mathbb{C}^*)$; see Theorem 
\ref{thm: Zinde Thm for n=2} for a more precise description. 
The proof of Zinde's Theorem for $n>2$ is done in 
\ref{lem: Zin77-alg}-\ref{prop: end-Zinde-thm}. 

The following lemma is a generalization to the series B of the classical Artin Theorem 
(\cite[Theorem 3]{Artin}), which says that 
the pure braid group $\mathbf{PA}_{n-1}$ 
is a characteristic subgroup of  the Artin braid group $\mathbf{A}_{n-1}$. 
Recall that a subgroup $N$ of a group $G$ is called {\em characteristic} if $N$ is stable 
under all automorphisms of $G$.  

\begin{lem} \label{lem: Zin77-alg} 
Let $\mathbf{B}_{n}=\pi_1(\mathcal{C}^n(\mathbb{C}^*))$ 
be the Artin-Brieskorn braid group of type $B_n$, and let $\mathbf{PB}_n\subset \mathbf{B}_{n}$ 
be the kernel of the natural surjection $\mathbf{B}_{n}\to \mathbf{S}(n)$. 
Then $\mathbf{PB}_n$ is a characteristic subgroup of $\mathbf{B}_{n}$ for any $n>2$. 
\elem

\bproof For $n\ge 5$ the assertion follows from a more general result due to Zinde 
(\cite[Theorem 5]{Zin77}), which says that $\mathbf{PB}_n$ is invariant under any endomorphism of 
$\mathbf{B}_{n}$ with non-Abelian image. 
Alternatively, the assertion also follows for $n=3$ and any $n\ge 5$  from a result of Ivanov 
(\cite[Theorem 2]{Ivanov}), 
which generalizes Artin's Theorem to the braid groups of general surfaces. 
The result for $n=4$ is as well announced  in \cite{Ivanov} 
(a remark after Theorem 2), however, the proof of this was never published. 
We provide therefore yet another proof, which works for all $n\ge 3$. 

Our approach uses the following result. We let $\mathbf{WB}_n$ be the finite 
Coxeter group of type $B_n$, 
and $\mathbf{CB}_{n}$ (the `colored braid group') be the kernel of the natural surjection 
$\omega_n\colon \mathbf{B}_{n}\to \mathbf{WB}_n$. 
By a result of Cohen and Paris (\cite[Propositions 2.4 and 3.9]{CP03}), 
$\mathbf{CB}_{n}$ is a characteristic subgroup of $\mathbf{B}_{n}$ for any $n\in\N$. 

Let further  $\mathcal{E}_n\cong(\Z/2\Z)^n$ be the subgroup of $\mathbf{WB}_n$
generated by the orthogonal reflexions in $\mathbb{R}^n$ with mirrors being 
the coordinate hyperplanes. 
We have $\mathbf{WB}_n=\mathcal{E}_n\rtimes\mathbf{S}(n)$. If $\pi_n\colon  \mathbf{WB}_n\to \mathbf{S}(n)$ 
is the surjection with kernel $\mathcal{E}_n$, then 
$\mathbf{PB}_{n}=\ker (\pi_n\circ\omega_n)\supset\ker\omega_n=\mathbf{CB}_{n}$. 

Let $\alpha\in\Aut \mathbf{B}_{n}$. Since $\mathbf{CB}_{n}\subset \mathbf{B}_{n}$ 
is a characteristic  subgroup, we have 
$\alpha(\mathbf{CB}_{n})=\mathbf{CB}_{n}$. Hence $\alpha$ induces an automorphism 
 $\bar{\alpha}\in\Aut \mathbf{WB}_n$ 
of the quotient group. By a lemma of Franzsen (\cite[Lemma 2.6]{Franzsen}),
for any $n\ge3$ the subgroup $\mathcal{E}_n\subset \mathbf{WB}_n$ is characteristic. 
Thus, $\bar{\alpha}(\mathcal{E}_n)=\mathcal{E}_n$. 
Since $\mathbf{PB}_{n}=\omega_n^{-1}(\ker\pi_n)=\omega_n^{-1}(\mathcal{E}_n)$, 
then $\alpha(\mathbf{PB}_{n})=\mathbf{PB}_{n}$, 
and so, $\mathbf{PB}_{n}\subset \mathbf{B}_{n}$ is a characteristic  subgroup too. 
 \eproof

\brems\label{rem: Franzsen-n=2} 1. Lemma \ref{lem: Zin77-alg} does not hold any longer for $n=2$. 
Indeed, recall that the group $\mathbf{B}_{2}$ admits a presentation 
%&
$$
\mathbf{B}_{2}=\langle \sigma_1,\sigma_2\,|\,\sigma_1\sigma_2\sigma_1\sigma_2=\sigma_2\sigma_1\sigma_2\sigma_1\rangle
$$
%&
with standard generators $\sigma_1$ and $\sigma_2$.
Let $\rho\colon \mathbf{B}_{2}\to\mathbf{S}(2)$ be the standard surjection sending $\sigma_1$ to 
$\sigma=(1,2)$ and $\sigma_2$ to $\id$.
Thus, we have $\mathbf{PB}_{2}=\ker\rho=\langle \sigma_1^2,\sigma_2\rangle$. 
Let $\alpha\in\Aut \mathbf{B}_{2}$ be the involution  of $\mathbf{B}_{2}$
interchanging $\sigma_1$ and $\sigma_1$. It is called a {\em graph involution} in \cite{Franzsen}, since it is induced by an involution of the Dynkin graph of type 
$B_2$ (the only graph of series $B$, which admits a nontrivial automorphism).
Then  $\alpha(\mathbf{PB}_{2})=\langle\sigma_1,\sigma_2^2\rangle\neq \mathbf{PB}_{2}$.
Therefore, the  subgroup
$\mathbf{PB}_{2}\subset \mathbf{B}_{2}$ is not characteristic. 

2. Likewise, Franzsen's Lemma cited above does not hold for $n=2$. 
Indeed, we have 
%&
$$\mathbf{WB}_2=\mathcal{E}_2\rtimes\langle\sigma\rangle=\{\id, \varepsilon_1,\varepsilon_2, 
w_0,
\sigma, w_0\sigma, \varepsilon_1\sigma, \varepsilon_2\sigma\}\,,
$$
%&
where $\sigma=(1,2)$ acts  naturally on $\R^2$ as an orthogonal  reflexion, 
$\mathcal{E}_2=\langle\varepsilon_1,\varepsilon_2\rangle
\cong(\Z/2\Z)^2$ with
$\varepsilon_i$ being the reflexion  in $\R^2$ of sign change of the $i$th coordinate, 
$i=1,2$, and $w_0=\varepsilon_1\varepsilon_2=-\id$. \footnote{Note that the natural surjection 
$\mathbf{B}_2\to \mathbf{WB}_2$ sends the standard generators $\sigma_1$ and 
$\sigma_2$ of $\mathbf{B}_2$ to $\sigma$ and 
$\varepsilon_1$, respectively, and sends the generator $(\sigma_1\sigma_2)^2$ of the center 
$Z(\mathbf{B}_2)\cong\Z$ to $\varepsilon_1\varepsilon_2$.}
It is easy to check that the involution
%&
$$
\bar\alpha\in\Aut \mathbf{WB}_2,\quad \id\mapsto \id, w_0\mapsto w_0, \varepsilon_1 \leftrightarrow\sigma, 
\varepsilon_2 \leftrightarrow w_0\sigma, \varepsilon_1 \sigma \leftrightarrow\varepsilon_2 \sigma
$$ 
%&
does not stabilize the subgroup $\mathcal{E}_2$. Hence this subgroup is not  characteristic. 
(In fact, $\bar\alpha$ is induced by $\alpha$ as above via the natural surjection $\Aut \mathbf{B}_2\to\Aut \mathbf{WB}_2$.)
\erems

The next corollary follows from a more general result of Zinde if $n>4$, see \cite[\S 2]{Zin78}. 

\begin{cor}\label{lem: Zin77}  
For $n>2$ any 
automorphism  $F$ of $\mathcal{C}^{n}(\mathbb{C}^*)$ 
can be lifted  to an automorphism 
$\widetilde F$ of 
$\mathcal{C}_{\ord}^{n}(\mathbb{C}^*)$. 
\end{cor}

\bproof
By the monodromy theorem, a continuous selfmap $F$ of the base of the Galois covering 
$\pi\colon\mathcal{C}_{\ord}^{n}(\mathbb{C}^*)\to \mathcal{C}^{n}(\mathbb{C}^*)=
\mathcal{C}_{\ord}^{n}(\mathbb{C}^*)/\mathbf{S}(n)$ admits a lift to the cover if and only 
if the induced endomorphism $F_*$ of the fundamental group  
$\pi_1 (\mathcal{C}^{n}(\mathbb{C}^*))=\mathbf{B}_n$ 
stabilizes the subgroup 
$\pi_*(\pi_1(\mathcal{C}_{\ord}^{n}(\mathbb{C}^*)))=\mathbf{PB}_n$. 
Note that this endomorphism is well defined only up to conjugation, 
i.e., up to the left multiplication by an inner automorphism. 
However, the property to stabilize the normal subgroup 
$\mathbf{PB}_n$ does not depend of this ambiguity in the definition of $F_*$.
If $F$ is an automorphism, 
then $F_*(\mathbf{PB}_n)=\mathbf{PB}_n$ by Lemma \ref{lem: Zin77-alg}. 
Hence $F$ admits a lift $\widetilde F$ 
to an automorphism of $\mathcal{C}_{\ord}^{n}(\mathbb{C}^*)$.
\eproof

Note that this corollary does not hold for $n=2$; see Example \ref{exa: non-liftable autom-m}.

The following lemma should be well known. For the lack of a reference, 
we provide a short argument. 

\blem\label{lem: normalizer} Let $\pi\colon\widetilde{\mathcal{X}}\to\mathcal{X}$ 
be a Galois covering with a Galois group $G$. Suppose that any automorphism 
$\alpha\in\Aut\mathcal{X}$ 
admits a lift to an automorphism $\widetilde{\alpha}\in\Aut\widetilde{\mathcal{X}}$. 
Then the set of all possible such lifts coincides with the normalizer $N$ of $G$ in 
$\Aut\widetilde{\mathcal{X}}$, and $\Aut\mathcal{X}\cong N/G$.
\elem

\bproof An automorphism $\gamma\in\Aut\widetilde{\mathcal{X}}$ is a lift of some 
$\alpha\in\Aut\mathcal{X}$ if and only if $\gamma$ sends the fibers of $\pi$ into fibers. 
The $\pi$-fibers are the $G$-orbits, hence $\gamma$ acts on the family of the $G$-orbits via
%&
\be
\label{eq: lift} \gamma(G.{\tilde x})=G.\gamma({\tilde x})\quad\forall 
{\tilde x}\in \widetilde{\mathcal{X}}\,.
\ee
%&
 Letting $\tilde x'=\gamma({\tilde x})$ we can write (\ref{eq: lift}) in the form 
%&
\be
\label{eq: normalizing} \gamma\circ G\circ \gamma^{-1}({\tilde x'})=
G.{\tilde x'}\quad\forall {\tilde x'}\in \widetilde{\mathcal{X}}\,.
\ee
%&
Any $\gamma\in N$ satisfies (\ref{eq: normalizing}). 
Conversely, let $\gamma\in\Aut\widetilde{\mathcal{X}}$ satisfies (\ref{eq: normalizing}). Then 
 $\gamma\circ G\circ \gamma^{-1}$ consists of deck transformations. 
However, the deck transformations form
the Galois group $G\simeq\pi_1(\mathcal{X})/\pi_1(\widetilde{\mathcal{X}})$.
Hence $\gamma\circ G\circ \gamma^{-1}=G$, these groups being of the same cardinality. 
Thus, $\gamma\in N$. This proves the first assertion.
The proof of the second is easy and is left to the reader.
\eproof

\bnota\label{sit: subgroup-of-Aut}  Given a group $G$ and a subgroup
$S$ of $G$, we let $\Norm_G\,(S)$ denote the normalizer of $S$ in $G$. 
In the sequel
$N= {\rm Norm}_{\Aut \mathcal{C}_{\ord}^{n}
(\mathbb{C}^*)}(\mathbf{S}(n))$ 
stands for the normalizer of $\mathbf{S}(n)$ in 
$\Aut\mathcal{C}_{\ord}^{n}(\mathbb{C}^*)$.
\enota

\bcor\label{cor: configuration-lift} Consider the Galois covering 
$ \mathcal{C}_{\ord}^{n}(\mathbb{C}^*)\to \mathcal{C}^{n}(\mathbb{C}^*)$
with the Galois group $ \mathbf{S}(n)$ acting on $ \mathcal{C}_{\ord}^{n}(\mathbb{C}^*)$ 
via permutations of coordinates. 
Then $\Aut \mathcal{C}^{n}(\mathbb{C}^*)\cong N/\mathbf{S}(n)$.
\ecor

\bproof Indeed, the assumptions of Lemma \ref{lem: normalizer} are fulfilled 
in this example due to Lemma \ref{lem: Zin77}. Applying Lemma \ref{lem: normalizer} 
yields the result.\eproof

Thus, to prove Zinde's Theorem we have to determine the normalizer $N$. 
This is done below. 

To describe the automorphism group 
$\Aut \mathcal{C}_{\ord}^{n}(\mathbb{C}^*)$
we use the following direct product decomposition.

\begin{nota}\label{lem: Cstar-cylinder} We let as usual $\mathbb{C}^{**}=
\mathbb{C}^{*}\setminus\{1\}$.
For any $n\ge 1$ there is an isomorphism 
%&
\begin{equation}\label{eq: transform of forms} 
\mathcal{C}_{\ord}^{n}(\mathbb{C}^*)\stackrel{\simeq}{\longrightarrow} 
\mathcal{C}_{\ord}^{n-1}(\mathbb{C}^{**})\times\mathbb{C}^*,\,\,\, 
Q=(q_1,...,q_n)\mapsto (Q',y)= 
\left(\left(\frac{q_1}{q_n},...,\frac{q_{n-1}}{q_n}\right),q_n\right)\,
\end{equation}
%& 
with inverse $\eta\colon (Q',y)\mapsto (yQ',y)$, where 
$Q'\in \mathcal{C}_{\ord}^{n-1}(\mathbb{C}^{**})$ and 
$y\in\mathbb{C}^*$.
It is equivariant with respect to the $\mathbb{C}^*$-actions on $\mathcal{C}_{\ord}^{n}(\mathbb{C}^*)$ 
via $ Q\mapsto sQ$  and on  the product $\mathcal{C}_{\ord}^{n-1}(\mathbb{C}^{**})\times\mathbb{C}^*$ 
via  $(Q',y)\mapsto (Q',sy)$, where $s\in\mathbb{C}^*$.
The first projection ${\rm pr}_1:\mathcal{C}_{\ord}^{n}(\mathbb{C}^*)
\to \mathcal{C}_{\ord}^{n-1}(\mathbb{C}^{**}),\,Q\mapsto Q'$, 
is the orbit map of the $\mathbb{C}^*$-actions on
 $\mathcal{C}_{\ord}^{n}(\mathbb{C}^*)$. 
 \end{nota}

The following simple lemma is a crucial point of our approach.

\begin{lem}\label{lem: preservation-of-cylinder} For $n\ge 1$ 
the image of any regular map 
$f\colon\mathbb{C}^*\to\mathcal{C}_{\ord}^{n}(\mathbb{C}^*)$ 
is contained in an orbit of the $\mathbb{C}^*$-action on 
$\mathcal{C}_{\ord}^{n}(\mathbb{C}^*)$. 
\end{lem}

\begin{proof} Since any morphism $\mathbb{C}^*\to\mathbb{C}^{**}$ 
is constant, also 
${\rm pr}_1\circ f\colon\mathbb{C}^*\to \mathcal{C}_{\ord}^{n-1}(\mathbb{C}^{**})
\subset (\mathbb{C}^{**})^{n-1}$ is.
This yields the assertion.
\end{proof}

The next corollary is straightforward. 

\begin{cor}\label{cor: C*-orbits} For $n\ge 1$ any automorphism 
$F\in \Aut \mathcal{C}_{\ord}^{n}(\mathbb{C}^*)$ 
preserves the family of the  
$\mathbb{C}^*$-orbits in $ \mathcal{C}_{\ord}^{n}(\mathbb{C}^*)$
and induces an automorphism $\phi\in  \Aut\mathcal{C}_{\ord}^{n-1}(\mathbb{C}^{**})$ 
such that ${\rm pr}_1\circ F= \phi\circ {\rm pr}_1$. The correspondence $F\mapsto\phi$  
defines a surjective homomorphism $\rho: \Aut \mathcal{C}_{\ord}^{n}(\mathbb{C}^*)\to 
\Aut\,\mathcal{C}_{\ord}^{n-1}(\mathbb{C}^{**})$.
\end{cor}

Due to the following proposition, the automorphisms of  
$\mathcal{C}_{\ord}^{n}(\mathbb{C}^*)$ are ``triangular'' in the sense close to that 
of Section \ref{sc: abstract scheme}; cf. Corollary \ref{Crl: our automorphisms are triangular}.

\begin{prop}\label{cor: sdp-decomposition} {\rm (a)} For $n\ge 1$
 any automorphism 
$F\in \Aut ( \mathcal{C}_{\ord}^{n-1}(\mathbb{C}^{**})\times\mathbb{C}^*)
\cong \Aut \mathcal{C}_{\ord}^{n}(\mathbb{C}^*)$ can be written as $F(x,y)=
(Sx, A(x)y)$ for some $A\in \Mor \left(\mathcal{C}_{\ord}^{n-1}(\mathbb{C}^{**}), 
\Aut\,\mathbb{C}^{*}\right)$ and  
$S\in \Aut \mathcal{C}_{\ord}^{n-1}(\mathbb{C}^{**})$, for all $(x,y)\in 
\mathcal{C}_{\ord}^{n-1}(\mathbb{C}^{**})\times\mathbb{C}^*$. Vice versa, for
any such pair $(S,A)$  the above formula defines an automorphism  $F=  F(S,A)$ 
of the direct product $\mathcal{C}_{\ord}^{n-1}(\mathbb{C}^{**})\times\mathbb{C}^*$.

 {\rm (b)} There is a  decomposition 
%&
\begin{equation} \label{eq: sdp-decomposition}
\Aut \mathcal{C}_{\ord}^{n}(\mathbb{C}^*)
\cong
\Mor \left(\mathcal{C}_{\ord}^{n-1}(\mathbb{C}^{**}), 
\Aut \mathbb{C}^*\right)\rtimes  \Aut\,\mathcal{C}_{\ord}^{n-1}(\mathbb{C}^{**})\,.
\end{equation}
%&
\end{prop}

\begin{proof} The exact sequence
%&
\begin{equation}\label{eq: ex-sec-rho}
0\to\ker\rho\to \Aut \mathcal{C}_{\ord}^{n}(\mathbb{C}^*)\stackrel{\rho}{\longrightarrow } 
\Aut\,\mathcal{C}_{\ord}^{n-1}(\mathbb{C}^{**})\to 0\,
\end{equation}
%&
splits, with splitting homomorphism 
%&
$$
\Aut\,\mathcal{C}_{\ord}^{n-1}(\mathbb{C}^{**})\ni\phi\mapsto (\phi, \id)\in 
\Aut (\mathcal{C}_{\ord}^{n-1}(\mathbb{C}^{**})\times\mathbb{C}^*)\cong 
\Aut \mathcal{C}_{\ord}^{n}(\mathbb{C}^*)\,.
$$  
%&
Hence $ \Aut \mathcal{C}_{\ord}^{n}(\mathbb{C}^*)\cong\ker\rho\rtimes  
\Aut\,\mathcal{C}_{\ord}^{n-1}(\mathbb{C}^{**})$.
It remains to show that $\ker\rho\cong\Mor \left(\mathcal{C}_{\ord}^{n-1}(\mathbb{C}^{**}), 
\Aut\,\mathbb{C}^{*}\right)$. 

Indeed, any $F\in\ker\rho$ preserves every $\mathbb{C}^*$-orbit 
and induces 
an automorphism of this orbit. Identifying $ \mathcal{C}_{\ord}^{n}(\mathbb{C}^*)$ 
with the direct product $ \mathcal{C}_{\ord}^{n-1}(\mathbb{C}^{**})\times\mathbb{C}^*$ 
via isomorphism (\ref{eq: transform of forms}), we
obtain a morphism $A\colon \mathcal{C}_{\ord}^{n-1}(\mathbb{C}^{**})\to 
\Aut \mathbb{C}^*$ such that $F(x,y)=(x, A(x)y)$, where 
$x\in \mathcal{C}_{\ord}^{n-1}(\mathbb{C}^{**})$ and $y\in \mathbb{C}^*$. 
Vice versa, for every morphism 
$A\colon \mathcal{C}_{\ord}^{n-1}(\mathbb{C}^{**})\to \Aut \mathbb{C}^*$ this formula 
gives an automorphism $F\in\ker\rho$. 
\end{proof}

Our next aim is to describe the factors in (\ref{eq: sdp-decomposition}). 
We start with the factor $\Aut \mathcal{C}_{\ord}^{n-1}(\mathbb{C}^{**})$, 
see Lemma \ref{lem: Kaliman}. This lemma follows from a more general 
result of Kaliman  \cite{Kal76a}; 
see \cite[Theorem 16]{Lin11} for a proof. 

\begin{lem}\label{lem: Kaliman} $(\rm S.\ Kaliman \,\,\,$\cite{Kal76a}$)$ 
For  $n>1$ we have
%&
$\Aut  \mathcal{C}_{\ord}^{n-1}(\mathbb{C}^{**})\cong \mathbf{S}(n+2)\,.
$
\end{lem}

\brem\label{sit: Lin-Kaliman} A hint for the proof is as follows; see \cite[\S 3.2 and \S 4.10]{Lin11}.  
The diagonal action of the group 
$\mathbf{PSL}(2,\mathbb{C})$ on the configuration space 
$\mathcal{C}_{\ord}^{n+2}(\mathbb {P}^{1})$ commutes with the  
$\mathbf{S}(n+2)$-action via permutations of coordinates. 
There is an isomorphism 
$\mathcal{C}_{\ord}^{n-1}(\mathbb{C}^{**})\cong 
\mathcal{C}_{\ord}^{n+2}(\mathbb {P}^{1})/\mathbf{PSL}(2,\mathbb{C})$. 
The $\mathbf{S}(n+2)$-action on $\mathcal{C}_{\ord}^{n+2}(\mathbb {P}^{1})$ 
descends to an effective $\mathbf{S}(n+2)$-action on the quotient 
$\mathcal{C}_{\ord}^{n-1}(\mathbb{C}^{**})$. Using an explicit description 
of all non-constant morphisms $\mathcal{C}_{\ord}^{n-1}(\mathbb{C}^{**})\to\mathbb{C}^{**}$, 
it was shown in [ibid] that
the image of $\mathbf{S}(n+2)$ exhausts the whole automorphism group 
$\Aut  \mathcal{C}_{\ord}^{n-1}(\mathbb{C}^{**})\,.$ 
\erem

From Proposition 
\ref{cor: sdp-decomposition}(b) and Lemma \ref{lem: Kaliman} we obtain such a corollary.

\bcor\label{cor: identity-component} For $n\ge 2$  we have 
%&
\be\label{eq: units0}
{\Aut}\, \mathcal{C}_{\ord}^{n}(\mathbb{C}^*)\cong
\left(\mathcal{O}^\times(\mathcal{C}_{\ord}^{n-1}(\mathbb{C}^{**}))
\rtimes (\Z/2\Z)\right)\rtimes  \mathbf{S}(n+2)\,,
\ee
%&
where 
%&
\be\label{eq: units1}
\mathcal{O}^\times(\mathcal{C}_{\ord}^{n-1}(\mathbb{C}^{**}))\cong
\mathcal{O}^\times (\mathcal{C}_{\ord}^{n}
(\mathbb{C}^*))^{\mathbb{C}^*}\,.
\ee
\ecor

In the following lemma we describe the group of units on the configuration space 
$\mathcal{C}^n(\mathbb{C}^*)$ and its $\mathbb{C}^*$-stable part.

\begin{lem}\label{lem: units on Cstar-n} {\rm (a)} For $n\ge 2$  
we have $\mathcal{O}^\times (\mathcal{C}^n(\mathbb{C}^*)) =\mathbb{C}^*\rtimes \Z^2$, 
where the factor $\Z^2$ is freely generated by the invertible functions $D_n$ and $z_n$ on 
$\mathcal{C}^n(\mathbb{C}^*)\subset\mathbb{C}^n_{(z)}$. 

 {\rm (b)} Furthermore, a function 
$f\in \mathcal{O}^\times (\mathcal{C}^n(\mathbb{C}^*))$ is 
$\mathbb{C}^*$-invariant if and only if $f=sh_n^k$, where $s\in \mathbb{C}^*$, $h_n=D_n z_n^{1-n}$, 
and $k\in\Z$. Hence $$\mathcal{O}^\times
 (\mathcal{C}^n(\mathbb{C}^*))^{\mathbb{C}^*}\cong \mathbb{C}^*\times\Z\,.$$
 \end{lem} 

\begin{proof}
The configuration space $\mathcal{C}^{n}(\mathbb{C}^*)$ 
can be realized as the complement 
in the affine space $\mathbb{C}^{n}_{(z)}$ 
to the union of the coordinate hyperplane $z_n=0$ 
and the discriminant hypersurface $d_n(z)=0$. Hence
the discriminant $D_n=d_n\vert \mathcal{C}^{n}(\mathbb{C}^*)$ 
and the function 
$z_n(Q)={(-1)^n}\displaystyle{\prod\limits_{q\in{Q}}q}$ freely 
generate the quotient group 
$\mathcal{O}^\times (\mathcal{C}^n(\mathbb{C}^*))/
\mathbb{C}^*\cong
H^1(\mathcal{C}^n(\mathbb{C}^*);\Z)\cong\Z^2$.  
This proves (a). Now (b) follows readily, because $D_n(\lambda Q)=\lambda^{n(n-1)} Q$ 
and $z_n(\lambda Q)=\lambda^{n} Q$ 
for any $Q\in \mathcal{C}^n(\mathbb{C}^*)$ and $\lambda\in \mathbb{C}^*$.
\end{proof}

We need the following elementary lemma. 

\blem\label{lem: algebra} Let
$\rho\colon G\to H$ be a homomorphism of groups, and let $S$ be a subgroup of 
$G$.  
If $N=\Norm_{G}(S)$  and $\bar N=\Norm_{H}(\rho(S))$,
then $N\subset \rho^{-1}(\bar N)$. 
\elem

\bsit\label{sit: tilde h} By Lemma \ref{lem: units on Cstar-n}(b), the function 
$\widetilde  h_n=h_n\circ \pi$ generates the group 
$\mathcal{O}^\times(\mathcal{C}_{\ord}^{n}
(\mathbb{C}^*))^{\mathbb{C}^*\times\mathbf{S}(n)}/\mathbb{C}^*$.
 This function participates in the next lemma. 
\esit

\blem\label{lem: 3-points}  Let $n\ge 2$.  {\rm (a)}
For any $\epsilon\in\{1,-1\}$, $s\in\mathbb{C}^*$, 
and $k\in\mathbb{Z}$, the formula\footnote{Cf.\ (\ref{eq: aut-form CnCstar}).} 
%&
\begin{equation}\label{eq: aut-form CnCstar1}
F\colon Q\mapsto s\widetilde h_n^k(Q)Q^\epsilon, \ \, \textup{\em where} \ \, 
Q\in\mathcal{C}_{\ord}^{n}(\mathbb{C}^*)\,,
\end{equation}
%&
defines an automorphism $F\in\Aut \mathcal{C}_{\ord}^{n}(\mathbb{C}^*)$. 
All these automorphisms form
a subgroup $H_1\subset\Aut\mathcal{C}_{\ord}^{n}(\mathbb{C}^*)$ 
such that $H_1\cong (\mathbb{C}^*\times \mathbb{Z})\rtimes(\mathbb{Z}/2\mathbb{Z})$, 
where the factors $\Z$ and $\Z/2\Z$ commute. 

{\rm (b)}
The subgroup $H_1$ 
commutes with the Galois group $H_2=\mathbf{S}(n)$ of the covering 
$\pi\colon\mathcal{C}_{\ord}^{n}(\mathbb{C}^*)\to\mathcal{C}^{n}(\mathbb{C}^*)$. 
If  $H\subset\Aut\mathcal{C}_{\ord}^{n}(\mathbb{C}^*)$ is the  subgroup generated by 
$H_1$ and $H_2$, then 
%&
$$
H\cong H_1\times H_2\cong \left((\mathbb{C}^*\times 
\mathbb{Z})\rtimes(\mathbb{Z}/2\mathbb{Z})\right)\times\mathbf{S}(n)\,.
$$
%&

{\rm (c)} We have 
$H\subset N=
\Norm_{\Aut\mathcal{C}_{\ord}^{n}(\mathbb{C}^*)}\,(\mathbf{S}(n))$. 
\elem

\bproof 
The inclusion $H\subset N$ follows from the fact that the groups 
$H_1$ and $\mathbf{S}(n)$ commute.
The rest of the proof is easy, so we leave it to the reader. 
\eproof

\bsit\label{sit: remark-Zinde-thm} In Proposition \ref{prop: end-Zinde-thm} 
we show that, in fact, $N=H$  for any $n> 2$.   
Then by Corollary \ref{cor: configuration-lift},
$\Aut\mathcal{C}^{n}(\mathbb{C}^*)=
H/\mathbf{S}(n)\simeq(\mathbb{C}^*\times \mathbb{Z})
\rtimes(\mathbb{Z}/2\mathbb{Z})$, with the factors defined 
as in Zinde's Theorem. This gives a proof of Zinde's Theorem.
\esit

In the sequel we use the following portion of notation. 

\bnota\label{sit: decomposition-again} 
Recall that by Corollary \ref{cor: identity-component},  for $n> 1$
%&
\be\label{eq: units3}
{\Aut}\, \mathcal{C}_{\ord}^{n}(\mathbb{C}^*)\cong
\left(\mathcal{O}^\times \left(\mathcal{C}_{\ord}^{n}
(\mathbb{C}^*)\right)^{\mathbb{C}^*}
\rtimes (\Z/2\Z)\right)\rtimes  \mathbf{S}(n+2)= 
\widetilde H_1\rtimes \widetilde H_2\,,
\ee
%&
where $ \widetilde H_1=\mathcal{O}^\times \left(\mathcal{C}_{\ord}^{n}
(\mathbb{C}^*)\right)^{\mathbb{C}^*}
\rtimes (\Z/2\Z)$
and $ \widetilde H_2=\mathbf{S}(n+2)$. Letting 
$ \widetilde N_2=\Norm_{ \widetilde H_2}\,(\rho(\mathbf{S}(n)))$, 
where $\rho\colon \widetilde H_1\rtimes \widetilde H_2\to  \widetilde H_2$ 
is the quotient morphism, by Lemma \ref{lem: algebra} we obtain 
$ \rho(N)\subset\widetilde N_2$.
 In Lemma  \ref{lem: N2} we describe the subgroup
$\widetilde N_2\subset \widetilde H_2$.
\enota

\bnota\label{sit: involutions}
In the sequel we deal with the following involutions in 
${\Aut}\, \mathcal{C}_{\ord}^{n}(\mathbb{C}^*)$. We let 
$\sigma_i=(i.i+1)\in  \mathbf{S}(n)$, $i=1,\ldots,n-1$, where 
$\mathbf{S}(n)$ acts on $\mathcal{C}_{\ord}^{n}(\mathbb{C}^*)$ via 
permutations of coordinates. Clearly, 
$\langle \sigma_1,\ldots,\sigma_k\rangle=\mathbf{S}(k+1)$ for any 
$k=1,\ldots,n-1$. 
For 
$Q=(q_1,\ldots,q_n)\in\mathcal{C}_{\ord}^{n}(\mathbb{C}^*)$ we let
%&
\be\label{eq: iota-tau}
\iota\colon Q\mapsto Q^{-1},\quad 
\tau\colon Q\mapsto q_n^{-2}Q,\quad \upsilon\colon Q\mapsto q_n^2Q^{-1},
\quad\mbox{and}\quad \sigma'=q_{n-1}^{-1}q_n\sigma_{n-1}(Q)\,.
\ee
%&
We have 
$\upsilon=\tau\circ\iota=\iota\circ\tau$. In fact, $\upsilon$ and
$\sigma'$ represent  the commuting transpositions 
$(n,n+2)\in \mathbf{S}(n+2)$ and $(n-1,n+1)\in\mathbf{S}(n+2)$, 
respectively, under the 
$\mathbf{S}(n+2)$-action on $\mathcal{C}_{\ord}^{n}(\mathbb{C}^*)$ 
(see Remark \ref{sit: Lin-Kaliman}).
\enota

In the next lemma we describe the normalizer 
$ \widetilde N_2=\Norm_{ \widetilde H_2}\,(\rho(\mathbf{S}(n)))$, where 
$\widetilde H_2=\mathbf{S}(n+2)\cong
\Aut\mathcal{C}_{\ord}^{n-1}(\mathbb{C}^{**})$. 

\blem \label{lem: N2} \begin{itemize}\item[(a)]
In the  notation of \ref{sit: involutions} we have 
%&
$$
\rho(\mathbf{S}(n))=
\langle \sigma_1,\ldots,\sigma_{n-2},\sigma'
\rangle={\rm Stab}_{\mathbf{S}(n+2)}(n)\cap 
{\rm Stab}_{\mathbf{S}(n+2)}(n+2)
\cong\mathbf{S}(n)\,.
$$
%&
\item[(b)] Furthermore, 
%&
$$ 
\widetilde N_2=\langle \sigma_1,\ldots,\sigma_{n-2}, \sigma', 
\upsilon
\rangle=\rho(\mathbf{S}(n))\times \langle
\upsilon
\rangle\cong \mathbf{S}(n)\times(\Z/2\Z)\,.
$$
%&
\end{itemize}
\elem

\bproof (a) It is easily seen that  $\rho(\sigma_i)=\sigma_i$, $i=1,\ldots,n-2$, 
while $\rho(\sigma_{n-1})=\sigma'\notin \mathbf{S}(n)$. Hence
the intersection of $\mathbf{S}(n)\subset 
{\Aut}\, \mathcal{C}_{\ord}^{n}(\mathbb{C}^*)$ with the subgroup
$\mathbf{S}(n+2)$ as in (\ref{eq: units3}) coincides with 
$\mathbf{S}(n-1)=\langle \sigma_1,\ldots,\sigma_{n-2}
\rangle={\rm Stab}_{\mathbf{S}(n)}(n)\subset \mathbf{S}(n)$. In particular,
 the restriction $\rho|_{\mathbf{S}(n-1)}$ is identical. Now (a) follows.

(b) Let $\sigma\in \widetilde N_2\subset \mathbf{S}(n+2)$, i.e., 
$\sigma$ normalizes the subgroup 
$\rho(\mathbf{S}(n))\subset\mathbf{S}(n+2)$.
Then $\sigma$ respects the orbit structure of $\rho(\mathbf{S}(n))$ acting 
on $\{1,\ldots, n+2\}$. 
Hence it preserves the orbit $\{1,\ldots, n-1, n+1\}$ and sends 
the fixed points $n$ and $n+2$ of $\rho(\mathbf{S}(n))$ 
into fixed points. 
This implies (b). 
\eproof

\bsit\label{sit: 4cases} By Lemma \ref{lem: algebra} we have 
$\rho(N)\subset\widetilde N_2$. Then by
Lemma \ref{lem: N2}(b)
and decomposition (\ref{sit: decomposition-again}),
any automorphism $F\in N$ admits one 
of the following presentations:
\begin{itemize}
\item[(i)] $F\colon Q\mapsto f(Q)\cdot\rho(\sigma)(Q)$,
\item[(ii)] $F\colon Q\mapsto f(Q)\cdot (\rho(\sigma)\circ\upsilon) (Q)$,
\item[(iii)] $F\colon Q\mapsto f(Q)\cdot (\tau\circ\rho(\sigma))(Q)$,
\item[(vi)] $F\colon Q\mapsto f(Q)\cdot (\tau\circ\rho(\sigma)\circ\upsilon) (Q)$,
\end{itemize}
where $Q\in\mathcal{C}_{\ord}^{n}(\mathbb{C}^*)$, 
$f\in\mathcal{O}^\times
(\mathcal{C}_{\ord}^{n}(\mathbb{C}^*))^{\mathbb{C}^*}$, 
$\sigma\in \mathbf{S}(n)$, and
the involutions $\tau$ and $\upsilon$ are defined in (\ref{eq: iota-tau}). Note that 
$\tau$ generates the factor $\Z/2\Z$ in 
(\ref{eq: units3}).
\esit

In the following proposition 
we prove the equality $N=H$, thus finishing the proof of Zinde's Theorem, 
see \ref{sit: remark-Zinde-thm}.

\bprop\label{prop: end-Zinde-thm}  Let $n>2$. Then for $F\in N$ the cases 
$(ii)$ and $(iii)$ never happen.
Furthermore, any $F\in N$ as in $(i)$ and $(iv)$, respectively,
can be given via 
%&
$$
F\colon Q\mapsto 
s\tilde h_n^k(Q)\cdot \sigma(Q)\quad\mbox{and}\quad F\colon Q\mapsto 
s\tilde h_n^k(Q)\cdot (\iota\circ \sigma)(Q)\,,
$$ respectively,
where $Q\in\mathcal{C}_{\ord}^{n}(\mathbb{C}^*)$, 
$s\in\mathbb{C}^*$, $k\in\Z$,
$\sigma\in \mathbf{S}(n)$, and $\iota$ is defined in $(\ref{eq: iota-tau})$.
Consequently, $N=H$.
\eprop

\bproof We begin with the following observations. 
Note first that the multiplication by a function $f\in\mathcal{O}^\times
(\mathcal{C}_{\ord}^{n}(\mathbb{C}^*))$ and the $\mathbf{S}(n)$-action 
on 
$\mathcal{C}_{\ord}^{n}(\mathbb{C}^*)$ commute. Furthermore,
any $\sigma\in  \mathbf{S}(n)$
can be written as a word, say,   $w(\sigma_1,\ldots,\sigma_{n-1})$.
Hence $\rho(\sigma)=w(\sigma_1,\ldots,\sigma_{n-2},\sigma')$, where 
$\sigma'\in\mathbf{S}(n+2)$ and 
$\sigma'=\rho(\sigma_{n-1})=q_{n-1}^{-1}q_n\sigma_{n-1}$,
see (\ref{eq: iota-tau}) and the proof of Lemma \ref{lem: N2}(a). Since 
$q_{n-1}^{-1}q_n\in\mathcal{O}^\times
(\mathcal{C}_{\ord}^{n}(\mathbb{C}^*))^{\mathbb{C}^*}$, 
it follows that 
$\rho(\sigma)=g\cdot w(\sigma_1,\ldots,\sigma_{n-1})$ for a certain
function $g\in\mathcal{O}^\times
(\mathcal{C}_{\ord}^{n}(\mathbb{C}^*))^{\mathbb{C}^*}$.

For $F$ in (i), the latter yields a presentation 
$F\colon Q\mapsto \tilde f(Q) \cdot\sigma(Q)$, where 
$\tilde f\in \mathcal{O}^\times
(\mathcal{C}_{\ord}^{n}(\mathbb{C}^*))^{\mathbb{C}^*}$.  
If $F\in N$, then also 
$F\circ\sigma^{-1}\in N$, because $\sigma\in\mathbf{S}(n)\subset N$.  
However,
$F\circ\sigma^{-1}\colon Q\mapsto \tilde f(Q)Q$ normalizes the subgroup 
$\mathbf{S}(n)\subset \Aut\mathcal{C}_{\ord}^{n}(\mathbb{C}^*)$ 
if and only if $ \tilde f\in \mathcal{O}^\times
(\mathcal{C}_{\ord}^{n}(\mathbb{C}^*))^{\mathbb{C}^*}$ is
$\mathbf{S}(n)$-invariant, i.e., if and only if $\tilde f=s\tilde h_n^k$ for some 
$s\in\mathbb{C}^*$ and $k\in\Z$, see \ref{sit: tilde h}. Thus $F\in H$.

The same argument allows to write $F$ in (iv) as 
$$F\colon Q\mapsto \tilde f(Q)\cdot (\tau\circ\sigma\circ\upsilon) (Q)
={\tilde f}(Q)\cdot (\tau\circ\upsilon\circ\sigma) (Q)=
{\tilde f}(Q)\cdot (\iota\circ\sigma) (Q)\,,$$ where 
$\tilde f\in \mathcal{O}^\times
(\mathcal{C}_{\ord}^{n}(\mathbb{C}^*))^{\mathbb{C}^*}$.  
We use here the equality
$\iota=\tau\circ\upsilon$ (see \ref{sit: involutions}) and the fact that 
$\upsilon$ and $\sigma$ commute. Assume further that $F\in N$.
Since both $\sigma$ and $\iota$ belong to $N$, then 
$Q\mapsto {\tilde f}(Q)Q$ is in $N$, too. This implies as before that 
$\tilde f=s\tilde h_n^k$ for some 
$s\in\mathbb{C}^*$ and $k\in\Z$, and so, $F\in H$.

It remains to eliminate the possibility that some
$F$ in (ii) or in (iii) belongs to $N$. 
With the same reasoning as before, any $F$ in (ii) can be presented  as 
%&
$$
F\colon Q\mapsto \tilde f(Q)\cdot (\sigma\circ\upsilon) (Q)=
\tilde f(Q)\cdot (\iota\circ\sigma\circ\tau) (Q)=
q_n^{-2}\tilde f(Q)(\iota\circ\sigma) (Q)\,,
$$
%&
where 
$\tilde f\in \mathcal{O}^\times
(\mathcal{C}_{\ord}^{n}(\mathbb{C}^*))^{\mathbb{C}^*}$.  
Suppose to the contrary that $F\in N$. Then also $\tilde F=F\circ
 (\sigma^{-1}\circ\iota) \in N$, where 
$\tilde F\colon Q\mapsto q_n^{-2}\tilde f(Q)Q$. Hence 
for any $\sigma\in \mathbf{S}(n)$ there exists $\sigma''\in\mathbf{S}(n)$ 
 such that $\tilde F\circ\sigma=\sigma''\circ \tilde F$. Thus, for any 
$Q=(q_1,\ldots,q_n)\in  \mathcal{C}_{\ord}^{n}(\mathbb{C}^*)$,
%&
\be\label{eq: norm-f}
q_{\sigma(n)}^{-2}\tilde f(\sigma(Q)) \sigma(Q)=
q_n^{-2} \tilde f(Q)\sigma''(Q)\,.
\ee
Let $Q$ be chosen so that $|q_i|< |q_j|$ for $i<j$. Then
both $\sigma$ and $\sigma''$ are uniquely determined by the images 
$\sigma(Q)$ and
$\sigma''(Q)$. Since by (\ref{eq: norm-f}) these 
two sequences are proportional, we have 
$ \sigma= \sigma''$. So, (\ref{eq: norm-f})
is equivalent to 
%&
\be\label{eq: norm-f1}
q_{\sigma(n)}^{-2} f(\sigma(Q)) =q_n^{-2} f(Q)\quad\forall 
\sigma\in\mathbf{S}(n),\,\,\,\forall Q\in \mathcal{C}_{\ord}^{n}
(\mathbb{C}^*)\,.
\ee
%&
We claim that there is no function $\tilde f\in \mathcal{O}^\times \left(\mathcal{C}_{\ord}^{n}
(\mathbb{C}^*)\right)^{\mathbb{C}^*}$ satisfying (\ref{eq: norm-f1}). Indeed,  
these equalities mean that the function $\tilde f/q_n^2$ is  
$\mathbf{S}(n)$-invariant. 
Hence $\tilde f/q_n^2$ descends to a function, say, $g\in \mathcal{O}^\times
(\mathcal{C}^{n}(\mathbb{C}^*))$. By Lemma \ref{lem: units on Cstar-n}(a) 
we have $g=sD_n^kz_n^l$ for some 
$s\in\mathbb{C}^*$ and $k,l\in\Z$. Therefore, $\tilde f=g\circ\pi= 
sq_n^2D_n^kz_n^l$. Since $\tilde f$ is $\mathbb{C}^*$-invariant, we obtain 
$n(n-1)k+nl+2=0$. The latter equality is impossible whatever are $k,l\in\Z$, 
because 
by our assumption $n>2$, in particular, $n\not |\, 2$. This proves our claim.

The possibility that some $F$ in (iii) belongs to $ N$ can be ruled out in a similar way. 
We leave the details to the reader.
\eproof

\bsit\label{sit: n=2} Next we turn to the case $n=2$. Example \ref{exa: non-liftable autom-m} 
shows that Zinde's Theorem does not hold any more for  $n=2$ (however, it is evidently true for $n=1$). 
Note that for any $n\ge 1$, the automorphisms of $\mathcal{C}_{\ord}^{n}(\mathbb{C}^*)$ 
as in (\ref{eq: aut-form CnCstar}) form a subgroup of 
the group $\Aut \mathcal{C}^{n}(\mathbb{C}^*)$. We let 
$\Aut_{\rm Zin}\, \mathcal{C}^{n}(\mathbb{C}^*)$ denote this subgroup. 
Our proof of Zinde's Theorem shows actually that for any $n\ge 1$, an automorphism 
$F\in\Aut \mathcal{C}^{n}(\mathbb{C}^*)$ admits a lift to 
$\widetilde F\in\Aut \mathcal{C}_{\ord}^{n}(\mathbb{C}^*)$ if and only if 
$F\in\Aut_{\rm Zin}\, \mathcal{C}^{n}(\mathbb{C}^*)$. 
In the following example we consider an automorphism 
$U\in\Aut \mathcal{C}^{2}(\mathbb{C}^*)$, which
does not admit a lift to $\mathcal{C}_{\ord}^{2}(\mathbb{C}^*)$. Hence 
$\Aut_{\rm Zin}\, \mathcal{C}^{2}(\mathbb{C}^*)\neq\Aut \mathcal{C}^{2}(\mathbb{C}^*)$. 
In particular, Corollary \ref{lem: Zin77} 
does not hold any longer for $n=2$. \esit

\bexa\label{exa: non-liftable autom-m} 
The affine surface 
$\mathcal{C}^{2}(\mathbb{C}^*)$ is isomorphic to the complement
$\mathbb{C}^2_z\setminus (\Gamma_1\cup \Gamma_2)$, where the plane affine curves 
$\Gamma_1$ and $\Gamma_2$ are 
given in the coordinates $(z_1,z_2)$ by equations 
$z_1^2-4z_2=0$ and $z_2=0$, respectively. Consider the involution 
$U\in\Aut\mathcal{C}^{2}(\mathbb{C}^*)$ 
given in these coordinates as
%&
$$
U\colon Q=(z_1,z_2)\mapsto U(Q)=\left(z_1, z_1^2/4 - z_2\right).
$$
%&
Then $U$ extends to a triangular automorphism 
of the plane $\mathbb{C}^2_z$ interchanging $\Gamma_1$ and $\Gamma_2$. 
Choose vanishing loops $\sigma_1$ and $\sigma_2$
of $\Gamma_1$ and $\Gamma_2$, respectively, with the base point $(4,2)\in\mathbb{C}^2_z$ fixed by $U$. 
Then these loops are 
interchanged by $U$. According to Brieskorn (\cite{Brieskorn})
the classes of these loops (denoted by the same letters) 
are the standard generators of the Artin-Brieskorn braid group 
$\mathbf{B}_{2}=\pi_1(\mathcal{C}^{2}(\mathbb{C}^*))=
\pi_1(\mathbb{C}^2_z\setminus (\Gamma_1\cup \Gamma_2))$. The induced 
graph automorphism $U_*$ of $\mathbf{B}_{2}$
interchanges 
$\sigma_1$ and $\sigma_2$, see Remark \ref{rem: Franzsen-n=2}.1.
It does not stabilize the pure braid group 
$\mathbf{PB}_{2}=\langle\sigma_1^2,\sigma_2\rangle\subset \mathbf{B}_{2}$; indeed, 
$U_*(\mathbf{PB}_{2})=\langle\sigma_1,\sigma_2^2\rangle\neq \mathbf{PB}_{2}$. Hence 
$U$ cannot be lifted to $\mathcal{C}_{\ord}^{2}(\mathbb{C}^*)$; 
see the proof of Corollary \ref{lem: Zin77}.

Alternatively, the latter can be seen as follows.
Assuming that there exists a lift $\widetilde U$ of $U$ to $\mathcal{C}_{\ord}^{2}(\mathbb{C}^*)$, it must 
be given for $Q\in \mathcal{C}_{\ord}^{2}(\mathbb{C}^*)$ by
%&
$$
\widetilde U\colon  Q=(q_1,q_2)\mapsto 
\left(\frac{q_1+q_2}{2}\pm \sqrt{q_1q_2},\, \frac{q_1+q_2}{2}\mp 
\sqrt{q_1q_2}\right)
\,.
$$
%&
However, due to ramifications this  formula does not  define a morphism of 
$\mathcal{C}_{\ord}^{2}(\mathbb{C}^*)$.

Thus, $U\in  \Aut \mathcal{C}^{n}(\mathbb{C}^*)\setminus \Aut_{\rm Zin}\, \mathcal{C}^{n}(\mathbb{C}^*)$, 
i.e., $U$ is not one of the automorphisms as in (\ref{eq: aut-form CnCstar}). So, Zinde's Theorem
does not hold any longer for $n=2$. 
\eexa

The next result completes the description of the groups  
$\Aut \mathcal{C}^{n}(\mathbb{C}^*)$, $n\ge 2$.

\bthm\label{thm: Zinde Thm for n=2} We have $\Aut \mathcal{C}^{2}(\mathbb{C}^*)
=\Aut_{\rm Zin}\, \mathcal{C}^{2}(\mathbb{C}^*)
\rtimes\langle U\rangle$, where $U\in\Aut \mathcal{C}^{2}(\mathbb{C}^*)$ 
is the involution 
as in Example {\rm \ref{exa: non-liftable autom-m}}.  \ethm

In the proof we use the following simple lemma.  
Consider the graph automorphism $\bar \alpha\in \Aut \mathbf{WB}_{2}$ 
as in Remark \ref{rem: Franzsen-n=2}.2. Given a group $G$, we 
let ${\rm Inn}\,G$ stand for the group of inner automorphisms of $G$.  

\blem\label{lem: Franzsen-2} {\rm (Franzsen, \cite[p.\ 21]{Franzsen})} We have 
$\Aut \mathbf{WB}_{2}= ({\rm Inn}\, 
\mathbf{WB}_{2})\rtimes\langle\bar \alpha\rangle$. 
In particular, the outer automorphism group ${\rm Out}\,\mathbf{WB}_{2}$ 
is a cyclic group of order 2 generated by the image of the graph automorphism 
$\bar \alpha$.  
\elem

Since the proof is not given in \cite{Franzsen}, we provide a short argument. 

\bproof In the notation of Remark \ref{rem: Franzsen-n=2}.2 
the conjugacy classes of  
$\mathbf{WB}_{2}$ are 
%&
\be\label{eq: conj-classes}
C_0=\{e\},\,\, C_1=\{\varepsilon_1\varepsilon_2\},\,\,C_2
=\{\varepsilon_1\sigma, \varepsilon_2\sigma\},\,\,C_3=
\{\varepsilon_1, \varepsilon_2\},\,\,C_4
=\{\sigma, \varepsilon_1\varepsilon_2\sigma\}\,.
\ee
%&
Any automorphism $\beta\in\Aut \mathbf{WB}_{2}$ induces a permutation 
of the set $(C_0, C_1, C_2, C_3, C_4)$. It is easily seen that $\beta$
preserves the classes 
$C_0, C_1, C_2$ 
and either preserves or interchanges 
$C_3$ and $C_4$. Thus, the  induced action of $\beta$ on the conjugacy 
classes yields a surjection 
$\mu\colon\Aut \mathbf{WB}_{2}\to \mathbf{S}(2)$. 
Clearly, $\bar \alpha\not\in\ker\mu$, whereas 
${\rm Inn}\,\mathbf{WB}_{2}\subset\ker\mu$. 
It suffices to show that ${\rm Inn}\,\mathbf{WB}_{2}=\ker\mu$. 
The latter holds indeed, since $C_2\cdot C_3=C_4$, 
and so, there are just 4 different possibilities for 
the action of an element $\beta\in\ker\mu$ on $\mathbf{WB}_{2}$. 
However, it is easy to check that all these possibilities 
can be realized by inner automorphisms.   
Finally, our argument shows that 
%&
$$
{\rm Out}\,\mathbf{WB}_{2}={\rm Aut}\,\mathbf{WB}_{2}/\ker\mu=
\mathbf{S}(2)\cong\Z/2\Z
$$
%&
 is generated 
by the image of $\bar \alpha$. 
\eproof

\brems\label{rem: action on conj} 1. Let $X$ be a connected and locally 
linearly 
connected topological space, and $p\in X$ a base point. 
Any continuous selfmap $f\colon X\to X$ induces an endomorphism 
$f_*\colon\pi_1(X,p)\to \pi_1(X,p)$, which is well defined up to a conjugation. 
That is, the coset $[f_*]$ of $f_*$  in ${\rm End} \,\pi_1(X,p)$ modulo 
the group ${\rm Inn}\, \pi_1(X,p)$ 
is well defined. 
Moreover,  the correspondence $f\mapsto [f_*]$ defines a homomorphism 
of semigroups 
${\rm End}\, X\to {\rm End} \,\pi_1(X,p)/{\rm Inn}\, \pi_1(X,p)$. 
which restricts further to a homomorphism 
%&
$${\rm Aut}\, X\to 
{\rm Out}\, \pi_1(X,p)={\rm Aut}\, \pi_1(X,p)/{\rm Inn} \,\pi_1(X,p)\,,
$$
%&
where ${\rm End}\, X$ and ${\rm Aut}\, X$ denote
(in this context) the semigroup of continuous selfmaps 
of $X$ and the group of homeomorphisms of $X$, respectively.

2. Furthermore, the action of $f_*$ is well defined on the set of 
conjugacy classes in $\pi_1(X,p)$. For $f\in {\rm Aut}\, X$ we let 
$f_{!}$ be the induced permutation of the conjugacy classes. 
The correspondence $f\to f_{!}$ yields a homomorphism
$ {\rm Aut}\, X\to {\rm Bij}\,C(\pi_1(X,p))$ to the permutation group 
of the set $C(\pi_1(X,p))$ of conjugacy classes of $\pi_1(X,p)$. 

3. Let $G$ be a group and $H\subset G$ be a subgroup. 
If $H$ is characteristic in $G$, then clearly $\Aut G$ and ${\rm Inn}\,G$ 
act on the quotient 
$G/N$, and there is a natural homomorphism 
${\rm Out}\,G\to {\rm Out}\,(G/N)$.
\erems

\noindent {\em Proof of Theorem} \ref{thm: Zinde Thm for n=2}. 
By the Cohen-Paris' Theorem \cite[Proposition 2.4]{CP03}, the kernel 
$\mathbf{CB}_{2}:=\ker(\omega_2\colon \mathbf{B}_{2}\to \mathbf{WB}_2)$ is a characteristic subgroup of the group $\mathbf{B}_{2}$,
that is, it is stable under the action  of $\Aut \mathbf{B}_{2}$. 
By virtue of Remarks \ref{rem: action on conj} and 
Lemma \ref{lem: Franzsen-2} this leads to the natural 
homomorphism
%&
\be\label{eq: out-homo}
\widehat\mu\colon {\rm Aut}\,  \mathcal{C}^{2}(\mathbb{C}^*)\to 
{\rm Out}\, 
\mathbf{B}_{2}\to {\rm Out}\, (\mathbf{B}_{2}/\mathbf{CB}_{2})=
 {\rm Out}\,  \mathbf{WB}_{2}\cong\Z/2\Z
\,
\ee
%&
in the realization 
$\mathbf{B}_{2}=\pi_1(\mathcal{C}^{2}(\mathbb{C}^*))$. 
The involution $U\in\Aut \mathcal{C}^{2}(\mathbb{C}^*)$ 
as in Example \ref{exa: non-liftable autom-m} induces the involution 
$U_*=\alpha\in \Aut \mathbf{B}_{2}$
of Remark \ref{rem: Franzsen-n=2}.1, 
which
descends to the graph involution 
$\bar\alpha\in\Aut \mathbf{WB}_{2}$ as 
in Lemma \ref{lem: Franzsen-2}, 
see also Remark \ref{rem: Franzsen-n=2}.2. 
Hence the image $\widehat\mu(U)$ generates the cyclic group $\Z/2\Z$ in (\ref{eq: out-homo}).

We claim that 
$\ker\widehat\mu={\Aut}_{\rm Zin}\, \mathcal{C}^{2}(\mathbb{C}^*)$. 
Indeed, in the notation as in (\ref{eq: conj-classes}) we have 
$\mathcal{E}_2=\langle\varepsilon_1,\varepsilon_2\rangle=
C_0\cup C_1\cup C_3\subset \mathbf{WB}_{2}$,
see Remark \ref{rem: Franzsen-n=2}.2. 
It follows from the proof of Lemma \ref{lem: Franzsen-2} 
that ${\rm Stab}_{\Aut \mathbf{WB}_{2}}(\mathcal{E}_2)=\ker\mu
={\rm Inn}\,\mathbf{WB}_{2}$, and so, 
${\rm Stab}_{{\rm Out} \,\mathbf{WB}_{2}}(\mathcal{E}_2)=\{\id\} $.  

Furthermore, $F\in\Aut\, \mathcal{C}^{2}(\mathbb{C}^*)$ belongs to
$\Aut_{\rm Zin}\, \mathcal{C}^{2}(\mathbb{C}^*)$ if and only 
if the induced automorphism
$F_*\in\Aut\mathbf{B}_{2}$
(whatever is its realization) stabilizes the subgroup 
$\mathbf{PB}_{2}\subset\mathbf{B}_{2}$,
see the discussion in \ref{sit: n=2}, 
if and only if the induced automorphism of $\mathbf{WB}_{2}$
stabilizes the subgroup $\mathcal{E}_2\subset\mathbf{WB}_{2}$, 
i.e., belongs to ${\rm Inn}\,\mathbf{WB}_{2}$, if and only if 
$F\in\ker\widehat\mu$, as claimed. 
  
Therefore, ${\Aut}_{\rm Zin}\, \mathcal{C}^{2}(\mathbb{C}^*)\subset{\Aut} \mathcal{C}^{2}(\mathbb{C}^*)$ 
is a subgroup of index 2 
which does not contain $U$. Now theorem follows. 
\qed

\section{The group $\Aut\Sigma^{n-2}_{\blc}$}\label{Sec: Aut-discriminant}
In this section we prove part (c) of Theorem
\ref{Thm: Kaliman-Lin-Zinde}. Let us recall this assertion.

\begin{thm}\label{Thm: aut balanced Sigma}
For $n>2$ we have
%&
$$
\Aut \Sigma^{n-2}_{\blc}\cong\mathbb{C}^*\,,
$$
%&
where $s\in\mathbb{C}^*$ acts on  $Q\in\Sigma^{n-2}_{\blc}$ via $Q\mapsto{s}Q$.
\end{thm}

For the proof we need some preparation.
Recall (see e.g. \cite{Arnold}, \cite{Napolitano}) that for $n\ge 4$ the singular locus 
$\sing\,\Sigma^{n-1}
=\Sigma^{n-1}\setminus\reg\,\Sigma^{n-1}$ of $\Sigma^{n-1}$ is the
union\footnote{This is not a stratification of $\sing\,\Sigma^{n-1}$ since 
$\Sigma_{\rm Maxw}^{n-2}\cap\Sigma_{\rm cau}^{n-2}\ne\varnothing$.} of the 
{\em Maxwell stratum} $\Sigma_{\rm Maxw}^{n-2}$ 
and the {\em Arnold caustic} $\Sigma_{\rm cau}^{n-2}$ 
defined by
%&
\begin{equation}\label{eq: maxw-cau}
\Sigma_{\rm Maxw}^{n-2}=p(\{q_{n-2}=q_{n-1}=q_n\})\quad\textup{and}
\quad\Sigma_{\rm cau}^{n-2}=p(\{q_{n-3}=q_{n-2},\,q_{n-1}=q_n\})\,,
\end{equation}
%&
where $p$ is the projection (\ref{eq: projection p}). So, 
$\Sigma_{\rm Maxw}^{n-2}$ and $\Sigma_{\rm cau}^{n-2}$ 
consist, respectively, of all unordered $n$-multisets $Q\subset\mathbb{C}$ that can be written as 
$Q=\{q_1,...,q_{n-3},u,u,u\}$ and $Q=\{q_1,...,q_{n-4},u,u,v,v\}$.
\vskip6pt

{\em Proof of Theorem} \ref{Thm: aut balanced Sigma}. We start 
with the case where $n>4$.
Consider the isomorphism 
%&
$$
\phi\colon\reg\,\Sigma^{n-2}_{\blc}\stackrel{\cong}{\longrightarrow}
\mathcal{C}^{n-2}(\mathbb{C}^{*})\,
$$
%&
as in  (\ref{eq: phi-isomorphism})
of Lemma 
\ref{Lm: reg Sigma(n-2)blc=C(n-2)(C*)} with inverse $\psi=\phi^{-1}$ 
given by (\ref{eq: psi-isomorphism}). We can associate to any 
$F\in\Aut\mathcal{C}^{n-2}(\mathbb{C}^{*})$ an automorphism 
$\widetilde F=\phi^{-1}\circ F\circ\phi\in\Aut(\reg\,\Sigma^{n-2}_{\blc})$. 
By Zinde's Theorem, 
for $n>4$ the  automorphism $F$ is given by (\ref{eq: aut-form CnCstar}).
We have to show that,  for $F$ as in (\ref{eq: aut-form CnCstar}),
the automorphism $\widetilde F$ extends to an automorphism of
$\Sigma^{n-2}_{\blc}$ if and only if $k=0$ and $\varepsilon=1$,
that is, if and only if $F\in\Aut (\mathcal{C}^{n-2}(\mathbb{C}^*))$ 
belongs
to the identity component $\Aut_0
(\mathcal{C}^{n-2}(\mathbb{C}^*)) \cong\mathbb{C}^*$.

Note that the $\mathbb{C}^*$-action $Q\mapsto{sQ}$
($s\in\mathbb{C}^*$, $Q\in\mathcal{C}^{n-2}(\mathbb{C}^*)$) on
$\mathcal{C}^{n-2}(\mathbb{C}^*)$  induces  a $\mathbb{C}^*$-action
on $\reg\,\Sigma^{n-2}_{\blc}$ given  again by $Q\mapsto{sQ}$  ($s\in\mathbb{C}^*$). The latter
$\mathbb{C}^*$-action extends to $\Sigma^{n-2}_{\blc}$ so that the
origin $\bar 0\in\Sigma^{n-2}_{\blc}$ is a unique fixed point.
This  fixed point lies in the closure of any one-dimensional $\mathbb{C}^*$-orbit.

Thus, without loss of generality we may
restrict to the case, where $s=1$ in  (\ref{eq: aut-form
CnCstar}), and so, $F=F_{k,\varepsilon}\colon Q\mapsto h_{n-2}^kQ^\varepsilon$.

The ($\Aut\mathbb{C}^*$)-invariant function $h_{n-2}\in
\mathcal{O}^\times(\mathcal{C}^{n-2}(\mathbb{C}^*))$ yields an
invertible regular function $g\Def h_{n-2}\circ\phi$ on
$\reg\,\Sigma^{n-2}_{\blc}$.
An automorphism  $F_k\colon Q\mapsto h_{n-2}^kQ$  of
$\mathcal{C}^{n-2}(\mathbb{C}^*)$ ($k\in\mathbb{Z}$) induces the
automorphism $\widetilde{F}_k\colon Q\mapsto{g^k}Q$ of
$\reg\,\Sigma^{n-2}_{\blc}$.

The subgroup $\Aut_0 (\mathcal{C}^{n-2}(\mathbb{C}^*))
\cong\mathbb{C}^*$ of $\Aut (\mathcal{C}^{n-2}(\mathbb{C}^*))$ 
being normal,
an automorphism $F\in\Aut (\mathcal{C}^{n-2}(\mathbb{C}^*))$
sends any $\mathbb{C}^*$-orbit in $\mathcal{C}^{n-2}(\mathbb{C}^*)$
into a $\mathbb{C}^*$-orbit  of the same dimension.  Since the function $h_{n-2}$ is constant
along the  $\mathbb{C}^*$-orbits,
the multiplication $Q\mapsto h_{n-2}^kQ$ preserves
each $\mathbb{C}^*$-orbit. Hence the automorphism
$F_{k,\varepsilon}\colon Q\mapsto h_{n-2}^kQ^\varepsilon$ sends the
$\mathbb{C}^*$-orbits in $\mathcal{C}^{n-2}(\mathbb{C}^*)$
into $\mathbb{C}^*$-orbits.
It follows that $\widetilde{F}_{k,\varepsilon}$ also sends
the $\mathbb{C}^*$-orbits in $\reg\,\Sigma^{n-2}_{\blc}$ into $\mathbb{C}^*$-orbits.

The involution $Q\mapsto Q^{-1}$ on
$\mathcal{C}^{n-2}(\mathbb{C}^*)$ sends any $\mathbb{C}^*$-orbit into
another such orbit interchanging the punctures, while  the multiplication
$Q\mapsto h_{n-2}^kQ$ preserves the punctures. Hence $\widetilde{F}_{k,\varepsilon}$ interchanges
the punctures of the $\mathbb{C}^*$-orbits in $\reg\,\Sigma^{n-2}_{\blc}$
if and only if $\varepsilon=-1$.

On the other hand, if $\widetilde{F}\in\Aut (\reg\,\Sigma^{n-2}_{\blc})$ admits
an extension,  say $\bar F$, to an automorphism of  $\Sigma^{n-2}_{\blc}$,  
then $\bar F$ should fix the origin. Indeed, $\bar{F}$ normalizes the $\mathbb{C}^*$-action 
on  $\Sigma^{n-2}_{\blc}$, hence it preserves the unique $\mathbb{C}^*$-fixed point 
$0\in\Sigma^{n-2}_{\blc}$.
This point is a unique
common point of the $\mathbb{C}^*$-orbit closures.
Hence $\widetilde{F}$ cannot interchange the punctures of the
$\mathbb{C}^*$-orbits in $\reg\,\Sigma^{n-2}_{\blc}$.
This proves that $\varepsilon=1$ for such an extendable $\widetilde{F}$.

The function  $h_{n-2}\in\mathcal{O}^\times(\mathcal{C}^{n-2}(\mathbb{C}^*))$ 
can be regarded as the
rational function $d_{n-2}(z)/z_{n-2}^{n-3}$ on $\mathbb{C}^{n-2}_{(z)}$, 
where $z_{n-2}=(-1)^{n-2}\prod_{i=1}^{n-2} q_i$. 
It has pole along the coordinate
hyperplane $z_{n-2}=0$, and $h_{n-2}^{-1}$ has pole
along the discriminant hypersurface $\Sigma^{n-3}=\{d_{n-2}=0\}$.
It follows by (\ref{eq: phi-isomorphism}) that $g$ regarded as a rational
function  on $\Sigma^{n-2}_{\blc}$ has pole along the caustic
$\Sigma^{n-3}_{\rm cau, \blc}=\Sigma^{n-2}_{\rm cau}\cap 
\Sigma^{n-2}_{\blc}$, and $g^{-1}$
has pole along the Maxwell stratum $\Sigma^{n-3}_{\rm
Maxw, \blc}=\Sigma^{n-2}_{\rm
Maxw}\cap\Sigma^{n-2}_{\blc}$, see (\ref{eq: maxw-cau}). Anyway,
the automorphism $\widetilde{F}_k\colon Q\mapsto g_{n}^kQ$ of
$\reg\,\Sigma^{n-2}_{\blc}$ does not admit an extension to an
automorphism of $\Sigma^{n-2}_{\blc}$ unless $k=0$ in (\ref{eq:
aut-form CnCstar}). Thus,  
$\widetilde{F}_{k,\varepsilon} \in\Aut (\reg\,\Sigma^{n-2}_{\blc})$ 
admits an extension
to an automorphism of  $\Sigma^{n-2}_{\blc}$ if and only if $k=0$ 
and $\varepsilon =1$, 
as stated. This ends the proof in the case $n>4$.

Let now $n=4$. The automorphism group 
$\Aut \mathcal{C}^{2}(\mathbb{C}^*)$ is described in Theorem 
\ref{thm: Zinde Thm for n=2}. Due to this theorem, any automorphism
$F\in\Aut \mathcal{C}^{2}(\mathbb{C}^*)$ can be written either as 
$F=F'\circ U$ or as $F=F'$, 
where $F'\colon Q\mapsto sg_{2}^kQ^\varepsilon$
is as in (\ref{eq: aut-form CnCstar}), and 
$U\in\Aut \mathcal{C}^{2}(\mathbb{C}^*)$ 
is the involution 
as in Example \ref{exa: non-liftable autom-m}. In the second case 
we have as before that 
$\widetilde{F}=\phi^{-1}\circ F\circ\phi \in\Aut (\reg\,\Sigma^{2}_{\blc})$ 
admits an extension
to an automorphism of  $\Sigma^{n-2}_{\blc}$ if and only if 
$\varepsilon=1$, $k=0$, and so, $F\in\mathbb{C}^*$. 

Assume further that $F=F'\circ U$.
The identity component $\mathbb{C}^*$ of 
$\Aut \mathcal{C}^{2}(\mathbb{C}^*)$ being normal, $F$ 
preserves the family of $\mathbb{C}^*$-orbits in 
$ \mathcal{C}^{2}(\mathbb{C}^*)$. So does 
$\widetilde F=\phi^{-1}\circ F\circ\phi\in\Aut(\reg\,\Sigma^{2}_{\blc})$ 
as well, since
$\phi\colon\reg\,\Sigma^{2}_{\blc}\stackrel{\cong}{\longrightarrow}
\mathcal{C}^{2}(\mathbb{C}^{*})$ is 
$\mathbb{C}^*$-equivariant. 

Likewise in Example \ref{exa: non-liftable autom-m}, we realize 
$\mathcal{C}^{2}(\mathbb{C}^{*})$ as 
$\mathbb{C}^{2}_{(z)}\setminus (C_1\cup C_2)$ with $C_1=\{z_2=0\}$ and 
$C_2=\{z_1^2-4z_2=0\}$. 
Since $U$ extends to an automorphism 
of the ambient affine plane $\mathbb{C}^2$ (denoted by the same letter), 
it sends any 
$\mathbb{C}^*$-orbits in 
$ \mathcal{C}^{2}(\mathbb{C}^*)$ into another one without 
interchanging the punctures. The same arguments as before prove 
that $\widetilde F$ 
does not extend to an automorphism of $\Sigma^{2}_{\blc}$ unless 
$\varepsilon=1$ and 
$k=0$. Applying a suitable element of the $\mathbb{C}^*$-action on 
$\Sigma^{2}_{\blc}$, we may  consider that also $s=1$, and so, 
$F'=\id$ and $F=U$. 

We claim that $\widetilde F=\widetilde U=\phi^{-1}
\circ U\circ\phi\in\Aut(\reg\,\Sigma^{2}_{\blc})$ cannot be
extended to an automorphism of $\Sigma^{2}_{\blc}$, and so, 
$ \Aut\Sigma^{2}_{\blc}$ reduces to its identity component 
$\mathbb{C}^*$, as required.
Indeed, suppose that $\widetilde U$ does extend to an  automorphism of 
$\Sigma^{2}_{\blc}$, which will be denoted
by the same symbol $\widetilde U$.
Observe that the morphism
$\psi=\phi^{-1}\colon\mathcal{C}^{2}(\mathbb{C}^{*})\to 
\reg\,\Sigma^{2}_{\blc}$ as in  (\ref{eq: psi-isomorphism})
extends naturally to 
a birational morphism $\bar\psi\colon\mathbb{C}^{2}\to 
\Sigma^{2}_{\blc}$.
The latter morphism fits in the commutative diagram 
%&
$$
\CD
\mathbb{C}^2 @ > {U} >> {\mathbb{C}^2}\\
@V{\bar\psi}VV @VV{\bar\psi}V\\
{\Sigma^{2}_{\blc}} @ > {\widetilde U}>>{\Sigma^{2}_{\blc}}
\endCD
$$
%&

The morphism $\bar\psi$ is  surjective and sends 
$C_1$ to the Maxwell stratum 
$\Sigma^{1}_{\rm Maxw, \blc}$ and  $C_2$ to the Arnold caustic 
$\Sigma_{\rm cau,\blc}^{1}$. Furthermore, the restriction 
$\psi|_{C_1}\colon C_1\to \Sigma^{1}_{\rm Maxw, \blc}$ is a bijection, 
while
$\psi|_{C_2}\colon C_2\to \Sigma^{1}_{\rm cau, \blc}$ has degree 2. 
However, the existence of a diagram of morphisms with these properties
contradicts the fact that the involution $U$ 
interchanges the curves $C_1$ and $C_2$. This completes the proof 
in the case $n=4$.

In the remaining case $n=3$ the plane affine curve 
$\Sigma^{1}_{\blc}$ is a standard
cuspidal cubic, and so, 
$\Aut\Sigma^{1}_{\blc}=\mathbb{C}^*$,
as desired.
\qed

\section{The group $\Aut\mathcal{C}^{n-1}_{\blc}(\mathbb{C})$ revisited}
\label{sec: KLZ-thm revisited}
In this section we  give an alternative proof  of Theorem 
\ref{Thm: Kaliman-Lin-Zinde}(a)
following the lines of the proof of Zinde's Theorem in Section 
\ref{sec:Zinde's Thm.}. For the reader's convenience  we recall
this statement. To make a link to Zinde's Theorem, it will be convenient to replace $n$ by $n+1$ in 
Theorem \ref{Thm: Kaliman-Lin-Zinde}(a).

\begin{thm}\label{thm: 5.1-a} For any $n>1$ we have 
$\Aut\mathcal{C}^{n}_{\blc}(\mathbb{C})\cong\mathbb{C}^*$. 
Any automorphism $F\in\Aut\mathcal{C}^{n-1}_{\blc}(\mathbb{C})$ is of the form $Q\mapsto sQ$, 
where  $s\in\mathbb{C}^*$ and
$Q\in\mathcal{C}^{n-1}_{\blc}(\mathbb{C})$. \end{thm}

The reader can find  two different proofs of this result in Sections \ref{sec:Zinde's Thm.} and \ref{Sec: 
Holomorphic endomorphisms of Cblc(n-1)}. Both of them refer to Tame Map Theorem. The  
alternative proof given below avoids addressing 
Tame Map Theorem. In turn, by Proposition \ref{prop: weak-TMT}, Tame Map Theorem in the particular case of biregular automorphisms of 
the configuration spaces 
$\mathcal{C}_{\blc}^{n}(\mathcal{C})$ and $\mathcal{C}^{n}(\mathcal{C}^*)$  
can be derived from 
Theorem \ref{thm: 5.1-a} 
and Zinde's Theorem proven in Section \ref{sec:Zinde's Thm.}, 
respectively.

\bnota\label{not: iso-inverse}  For $Q=(q_1,\ldots,q_{n})\in{\mathcal{C}_{\ord}^{n}(\mathbb{C}^*)}$ 
we let as before $\bc(Q)=\frac{1}{n}\sum_{i=1}^{n} q_i$. Consider the map 
%&
\begin{equation} \label{eq: isom}
\widetilde\varphi\colon \mathcal{C}_{\ord}^{n}(\mathbb{C}^*)\to\mathcal{C}_{\blc,\ord}^{n}(\mathbb{C}),
\quad Q\mapsto \left(Q-\frac{n}{n+1}\bc(Q),
 -\frac{n}{n+1}\bc(Q)\right).
\end{equation}
%&
This map is an isomorphism with inverse 
%&
\begin{equation} \label{eq: isom-1}
{\widetilde\varphi}^{-1}\colon \mathcal{C}_{\blc,\ord}^{n}(\mathbb{C})\to
\mathcal{C}_{\ord}^{n}(\mathbb{C}^*),\quad (q_1,\ldots,q_{n+1})\mapsto 
(q_1-q_{n+1},\ldots,q_{n}-q_{n+1})\,.
\end{equation}
%&
\enota

\begin{prop}\label{lem: equivariant diagram} 
For $n>1$ the isomorphism ${\widetilde\varphi}$ as in (\ref{eq: isom}) 
fits in the commutative diagram
%&
\begin{equation}\label{CD: diagram over T(0,m+3) to Com(C**)}
\CD
\mathcal{C}_{\ord}^{n-1}(\mathbb{C}^{**})\times\mathbb{C}^*
@> \eta >> \mathcal{C}_{\ord}^n(\mathbb{C}^*)
          @>\widetilde{\varphi} >> \mathcal{C}_{\blc,\ord}^{n}(\mathbb{C}) \\
@.%@/SE//\rho'/ 
@V/\mathbf{S}(n)VV     @VV/\mathbf{S}(n+1) V  \\
@.\mathcal{C}^n(\mathbb{C}^*)
          @> \widetilde{\psi} >> \mathcal{C}_{\blc}^{n}(\mathbb{C})\\
\endCD
\end{equation}
%&
where \begin{itemize}
\item[(i)]
the vertical columns are Galois coverings with the Galois groups $\mathbf{S}(n)$ and 
$\mathbf{S}(n+1)$, 
respectively, acting by permutations of coordinates. Furthermore, $\eta$ is an isomorphism as in 
{\rm (\ref{eq: transform of forms})}, and $\widetilde\psi$ is an unramified $n$-sheeted covering; 
\item[(ii)] $\widetilde\varphi$ conjugates $\mathbf{S}(n)$ with the stabilizer 
${\rm Stab}_{\mathbf{S}(n+1)}(n+1)\subset \mathbf{S}(n+1)$;
\item[(iii)] the morphisms in
 {\rm (\ref{CD: diagram over T(0,m+3) to Com(C**)})} are equivariant 
with respect to the free $\mathbb{C}^*$-actions  $Q\mapsto sQ$, $s\in\mathbb{C}^*$, 
on each of the four corners of the square,  and the standard $\mathbb{C}^*$-action on the second 
factor of $\mathcal{C}_{\ord}^{n-1}(\mathbb{C}^{**})\times\mathbb{C}^*$ identical on the first factor; 
\item[(iv)] the automorphisms of each of the varieties in 
{\rm (\ref{CD: diagram over T(0,m+3) to Com(C**)})} 
preserve the family of the $\mathbb{C}^*$-orbits on this variety. 
\end{itemize}
\end{prop}

\begin{proof}
The assertions (i)--(iii) can be verified without difficulty; (iv) follows from (iii) due to Corollary \ref{cor: C*-orbits}.
\end{proof} 

\begin{sit}\label{sit: scheme-of-proof} The proof of Theorem \ref{thm: 5.1-a} 
starts as follows. Given $F\in\Aut\mathcal{C}_{\blc}^{n}(\mathbb{C})$, 
we lift it to an automorphism $\tilde F$ of $\mathcal{C}_{\blc,\ord}^{n}(\mathbb{C})$, and 
then conjugate with an automorphism $\tilde F'$ of  
$\mathcal{C}_{\ord}^{n-1}(\mathbb{C}^{**})\times\mathbb{C}^*$. Due to Proposition 
\ref{cor: sdp-decomposition}(a),
the resulting automorphism $\tilde F'$
has a triangular form $\tilde F'\colon (Q',y)\mapsto (SQ',A(Q')y)$, see Notation 
\ref{lem: Cstar-cylinder}.
The next lemma makes this first step possible.
\end{sit}

\begin{lem} Any automorphism $F$ of $\mathcal{C}_{\blc}^{n}(\mathbb{C})$
admits a lift to an automorphism $\tilde F$ of $\mathcal{C}_{\blc, \ord}^{n}(\mathbb{C})$. 
\end{lem}

\begin{proof} By virtue of (\ref{eq: 3 cylinders}) we have 
$\mathcal{C}_{\blc}^{n}(\mathbb{C})\times\mathbb{C}\cong\mathcal{C}^{n+1}(\mathbb{C})$. 
Hence $\pi_1(\mathcal{C}_{\blc}^{n}(\mathbb{C}))\cong \pi_1(\mathcal{C}^{n+1}(\mathbb{C}))=
\mathbf{A}_{n}$ is the Artin braid group with $n+1$ strands. 
Similarly, the isomorphism 
$\mathcal{C}_{\blc,\ord}^{n}(\mathbb{C})\times\mathbb{C}\cong\mathcal{C}_{\ord}^{n+1}(\mathbb{C})$ 
yields an isomorphism 
$\pi_1(\mathcal{C}_{\blc,\ord}^{n}(\mathbb{C}))\cong\pi_1(\mathcal{C}_{\ord}^{n+1}(\mathbb{C}))
=\mathbf{PA}_{n}$,
where as before $\mathbf{PA}_{n}\subset \mathbf{A}_{n}$ is the pure braid group on $n+1$ strands. 
By virtue of  
Artin's Theorem (\cite[Theorem 3]{Artin})\footnote{See \cite[Theorem 10]{Lin11} 
for a  more general result.}, $\mathbf{PA}_{n}$ is a characteristic subgroup of $\mathbf{A}_{n}$. 
Therefore, the induced automorphism $F_*$ of $\mathbf{A}_{n}$  (which is well defined modulo an inner automorphism of $\mathbf{A}_{n}$) 
preserves the subgroup $\mathbf{PA}_{n}$. 
Now the assertion follows by the monodromy theorem. 
 \end{proof} 

\bnota\label{not: normalizer-bis} Consider the subgroup 
%&
$$
\widetilde{\mathbf{S}}=\widetilde\varphi^{-1}\mathbf{S}(n+1)
\widetilde\varphi\subset\Aut\mathcal{C}_{\ord}^{n}(\mathbb{C}^*)
$$
%& 
conjugated with $\mathbf{S}(n+1)\subset\Aut\mathcal{C}_{\blc,\ord}^{n}(\mathbb{C})$  
via $\widetilde\varphi$. 
We let $\widetilde N=\Norm_{\Aut\mathcal{C}_{\ord}^{n}(\mathbb{C}^*)}(\widetilde{\mathbf{S}})$
be the normalizer of $\widetilde{\mathbf{S}}$ in $\Aut\mathcal{C}_{\ord}^{n}(\mathbb{C}^*)$.
\enota

 By Lemma \ref{lem: normalizer}, $\widetilde N$ is isomorphic to the subgroup in 
$\Aut\mathcal{C}_{\blc,\ord}^{n}(\mathbb{C})$ of all lifts of the automorphisms from 
$\Aut\mathcal{C}_{\blc}^{n}(\mathbb{C})$. The next corollary of Lemma  \ref{lem: normalizer} 
is immediate.

\begin{cor}\label{cor: multiplication-by-cst} We have $\Aut\mathcal{C}_{\blc}^{n}(\mathbb{C})\cong 
\widetilde N/\mathbf{S}(n+1)$.
\ecor

\bsit\label{sit: proof-of-thm-8.1}  In Proposition \ref{prop: end-of-the-proof-thm-8.1} we show that 
$\widetilde N=\mathbb{C}^*\times\mathbf{S}(n+1)$, and so, by Corolary \ref{cor: multiplication-by-cst}, 
$\Aut\mathcal{C}_{\blc}^{n}(\mathbb{C})\cong\mathbb{C}^*$, where in both cases we mean the standard 
diagonal action of  $\mathbb{C}^*$ on $\mathcal{C}_{\blc,\ord}^{n}(\mathbb{C})$ and on 
$\mathcal{C}_{\blc}^{n}(\mathbb{C})$, respectively. 
This proves Theorem \ref{thm: 5.1-a}. \esit

\brem\label{rem: involutions-1} We let as before $\sigma_i=(i,i+1)\in\mathbf{S}(k)$ for $k>i$. 
Note that for any $i=1,\ldots,n-1$ the natural actions of the transposition $\sigma_i$ on the varieties 
in the upper line of (\ref{CD: diagram over T(0,m+3) to Com(C**)}) are mutually conjugated via the isomorphisms 
$\eta$ and $\widetilde\varphi$. In other words, the natural $\mathbf{S}(n)$-action on 
$\mathcal{C}_{\ord}^{n}(\mathbb{C}^*)$ is conjugated with the 
${\rm Stab}_{\mathbf{S}(n+1)}(n+1)$-action on $\mathcal{C}_{\blc,\ord}^{n}(\mathbb{C})$ and the 
$\left({\rm Stab}_{\mathbf{S}(n+2)}(n+1)\cap{\rm Stab}_{\mathbf{S}(n+2)}(n+2)\right)$-action on 
$\mathcal{C}_{\ord}^{n-1}(\mathbb{C}^{**})\times\mathbb{C}^*$, where $\mathbf{S}(n+2)$ 
is identified with 
$\Aut\mathcal{C}_{\ord}^{n-1}(\mathbb{C}^{**})\times\{\id\}\subset
\Aut(\mathcal{C}_{\ord}^{n-1}(\mathbb{C}^{**})
\times\mathbb{C}^*)$, see Lemma \ref{lem: Kaliman}.

The transposition $\sigma_n=(n,n+1)\in\mathbf{S}(n+2)$ acts on 
$\mathcal{C}_{\ord}^{n-1}(\mathbb{C}^{**})$ via $$Q=(q_1,\ldots,q_{n-1})\mapsto \alpha(Q)=
(\alpha(q_1),\ldots,\alpha(q_{n-1}))\,,$$ where 
$Q\in\mathcal{C}_{\ord}^{n-1}(\mathbb{C}^{**})$ and $\alpha(z)=1-z$ 
for $z\in\mathbb{C}$. The action of the transposition 
$\sigma_n\in\mathbf{S}(n+1)\subset\Aut\mathcal{C}_{\blc,\ord}^{n}(\mathbb{C})$ 
on $\mathcal{C}_{\blc,\ord}^{n}(\mathbb{C})$ gives rise to the involution, say, 
$\varrho$ on $\mathcal{C}_{\ord}^{n}(\mathbb{C}^*)$, where
%&
\be\label{eq: invol-sigma-n}
\varrho=\widetilde\varphi^{-1}\circ\sigma_n\circ\widetilde\varphi\in\widetilde{\mathbf{S}},
\quad Q=(q_1,\ldots,q_n)\mapsto (q_1-q_n,\ldots,q_{n-1}-q_n,-q_n)\,.
\ee
%& 
In turn, $\varrho$ gives rize to  the involution $(\sigma_n,\delta)$ on 
$\mathcal{C}_{\ord}^{n-1}(\mathbb{C}^{**})\times\mathbb{C}^*$, 
where $\delta(z)=-z$ for $z\in\mathbb{C}^*$. Thus, in the notation of \ref{sit: decomposition-again}, 
%&
\be\label{eq: sigma-n} 
\rho(\varrho)=\rho(\widetilde\varphi^{-1}\circ\sigma_n\circ\widetilde\varphi) =\rho(\sigma_n,\delta)
=\sigma_n\in\mathbf{S}(n+2)=\widetilde H_2\,,
\ee
%&
where $\rho: \Aut\mathcal{C}_{\ord}^{n}(\mathbb{C}^*)=\widetilde H_1
\rtimes\widetilde H_2\to\widetilde H_2$ is the natural surjection. \erem

These observations lead to the following lemma.

\blem\label{lem: N2-bis} Letting $\widetilde N_2=\Norm_{\mathbf{S}(n+2)}
(\rho(\widetilde{\mathbf{S}}))\subset\widetilde H_2= \mathbf{S}(n+2)$ we have
%&
\be\label{eq: norm-3} \widetilde N_2=\rho(\widetilde{\mathbf{S}})=
{\rm Stab}_{\mathbf{S}(n+2)}(n+2)\subset \widetilde N\,.
\ee
%& 
\elem

\bproof 
According to Remark 
\ref{rem: involutions-1}, the isomorphism 
$\widetilde\varphi$ conjugates 
$\sigma_i\in\mathbf{S}(n+1)\subset 
\Aut\mathcal{C}_{\blc,\ord}^{n}(\mathbb{C})$ to 
$\sigma_i\in\mathbf{S}(n)\subset 
\Aut\mathcal{C}_{\ord}^{n}(\mathbb{C}^*)$ for 
$i=1,\ldots,n-1$, and $\sigma_n$ to 
$\delta\circ\sigma_n\in\Aut\mathcal{C}_{\ord}^{n}(\mathbb{C}^*)$, 
where $\delta\in\mathbb{C}^*\subset \widetilde N$, 
$\delta\colon Q\mapsto -Q$, see (\ref{eq: sigma-n}). 
The latter projects further to $\sigma_n\in\mathbf{S}(n+2)$ 
under $\rho$,  see again (\ref{eq: sigma-n}). 
Hence $\rho(\widetilde{\mathbf{S}})=
{\rm Stab}_{\mathbf{S}(n+2)}(n+2)\subset\mathbf{S}(n+2)$. 
The normalizer of the subgroup 
${\rm Stab}_{\mathbf{S}(n+2)}(n+2)\cong\mathbf{S}(n+1)$ in 
$\mathbf{S}(n+2)$ coincides with this subgroup. 
These observations yield the equalities in (\ref{eq: norm-3}). 

We have $\widetilde{\mathbf{S}}=\widetilde\varphi^{-1}\mathbf{S}(n+1)
\widetilde\varphi\not\subset\mathbf{S}(n+2)=\widetilde H_2$. 
Indeed, $\widetilde\varphi^{-1}\sigma_n\widetilde\varphi=\delta\sigma_n$ by (\ref{eq: sigma-n}).
Nevertheless, $\rho (\widetilde{\mathbf{S}})= {\rm Stab}_{\mathbf{S}(n+2)}(n+2)
\subset \widetilde N$, since 
 $\mathbb{C}^*\subset \widetilde N$ and $\widetilde{\mathbf{S}}\subset \widetilde N$ 
by definition of $\widetilde N$, see \ref{not: normalizer-bis}. Hence 
$\mathbb{C}^*\cdot\widetilde{\mathbf{S}}\subset \widetilde N$, 
and, in particular, $\sigma_n=\rho(\delta\sigma_n)=\delta(\delta\sigma_n)\in \widetilde N$, because 
$\delta\sigma_n\in\widetilde{\mathbf{S}}\subset \widetilde N$. It remains to note that also 
$\mathbf{S}(n-1)\subset\widetilde{\mathbf{S}}\subset \widetilde N$, see Remark \ref{rem: involutions-1}, 
and so, $\rho (\widetilde{\mathbf{S}})=\langle\mathbf{S}(n-1), 
\rho(\delta\sigma_n)\rangle\subset \widetilde N$. 
\eproof

\bsit\label{sit: 4cases-bis} By Lemma \ref{lem: algebra} we have 
$\rho(\widetilde N)\subset\widetilde N_2$. Due to
Lemma \ref{lem: N2-bis}
and decomposition (\ref{eq: units3}), 
any automorphism $F\in \widetilde N$ admits one 
of the following presentations:
%&
\be\label{eq: i-ii}
{\rm (i)}\quad F\colon  Q\mapsto f(Q)\cdot\sigma(Q)\quad\mbox{and}\quad {\rm (ii)}\quad
F\colon  Q\mapsto f(Q)\cdot (\tau\circ\sigma)(Q)\,,
\ee
%& 
where $Q\in \mathcal{C}_{\ord}^{n}(\mathbb{C}^*)$, $f\in\mathcal{O}^\times 
(\mathcal{C}_{\ord}^{n}(\mathbb{C}^*))^{\mathbb{C}^*}$, 
$\sigma\in {\rm Stab}_{\mathbf{S}(n+2)}(n+2)=\mathbf{S}(n+1)\subset\mathbf{S}(n+2)$, 
and
$\tau\colon Q\mapsto q_n^{-2}Q$ as in (\ref{eq: iota-tau}) generates the factor $\Z/2\Z$ 
in (\ref{eq: units3}).
\esit

\blem\label{lem: PSL2-invar} Let $F\in \widetilde N$ be as in $ (i)$ (as in $(ii)$, respectively). 
Then the function $f$ ($q_n^{-2}f$, respectively) is $\widetilde{\mathbf{S}}$-invariant. 
\elem

\bproof 
We have:  $F\in \widetilde N$ if and only if
$F\circ\sigma^{-1}\in \widetilde N$. Indeed, $\sigma\in\rho(\widetilde{\mathbf{S}})\subset 
\widetilde N$ 
by Lemma \ref{lem: N2-bis}.  
Thus, it suffices to prove the assertions for the automorphisms $F\in \widetilde N$ of the forms 
%&
$$
{\rm (i')}\quad F\colon Q\mapsto f(Q)Q\quad\mbox{and}\quad {\rm (ii')}\quad
F\colon  Q\mapsto q_n^{-2}f(Q)Q\,,
$$
%&
where $Q\in \mathcal{C}_{\ord}^{n}(\mathbb{C}^*)$ and $f\in\mathcal{O}^\times 
(\mathcal{C}_{\ord}^{n}(\mathbb{C}^*))^{\mathbb{C}^*}$.

To this end, consider the quotient morphism 
%&
$$
\theta\colon\mathcal{C}_{\ord}^{n+2}(\mathbb{P}^1)
\stackrel{/{\bf PSL}(2,\mathbb{C})}{\longrightarrow}\mathcal{C}_{\ord}^{n-1}(\mathbb{C}^{**})\,
$$
%&
with respect to the natural diagonal ${\bf PSL}(2,\mathbb{C})$-action on 
$\mathcal{C}_{\ord}^{n+2}(\mathbb{P}^1)$ (see Remark \ref{sit: Lin-Kaliman}). 
It admits a section 
%&
$$
\mathcal{C}_{\ord}^{n-1}(\mathbb{C}^{**})\ni Q=(q_1,\ldots,q_{n-1})\mapsto 
(q_1,\ldots,q_{n-1},0,1,\infty)\in\mathcal{C}_{\ord}^{n+2}(\mathbb{P}^1)\,.
$$
%&
The latter leads to a ${\bf PSL}(2,\mathbb{C})$-equivariant factorization
%&
$$
\theta\colon\mathcal{C}_{\ord}^{n+2}(\mathbb{P}^1)\stackrel{\cong}{\longrightarrow}
\mathcal{C}_{\ord}^{n-1}(\mathbb{C}^{**})\times{\bf PSL}(2,\mathbb{C})
\stackrel{{\rm pr}_1}{\longrightarrow}\mathcal{C}_{\ord}^{n-1}(\mathbb{C}^{**})\,.
$$
%&
The $\mathbf{S}(n+2)$-action on $\mathcal{C}_{\ord}^{n-1}(\mathbb{C}^{**})$
lifts naturally to the direct product, and then also to $
\mathcal{C}_{\ord}^{n+2}(\mathbb{P}^1)$, where it acts via permutations of coordinates, 
see Remark \ref{sit: Lin-Kaliman}. 
Consider natural isomorphisms
%&
$$
\aligned
\mathcal{O}^\times(\mathcal{C}_{\ord}^n (\mathbb{C}^*))^{\mathbb{C}^*}\cong 
\mathcal{O}^\times(\mathcal{C}_{\ord}^{n-1}(\mathbb{C}^{**})\times\mathbb{C}^*)^{\mathbb{C}^*}\cong
\mathcal{O}^\times(\mathcal{C}_{\ord}^{n-1}(\mathbb{C}^{**}))\\
\cong\theta^*(\mathcal{O}^\times(\mathcal{C}_{\ord}^{n-1}(\mathbb{C}^{**})))=
\mathcal{O}^\times(\mathcal{C}_{\ord}^{n+2}(\mathbb{P}^1))^{{\bf PSL}(2,\mathbb{C})}\,.
\endaligned 
$$
%&
Any function $f\in\mathcal{O}^\times(\mathcal{C}_{\ord}^n (\mathbb{C}^*))^{\mathbb{C}^*}$
 lifts through this chain of isomorphisms to a function 
$\tilde f\in \mathcal{O}^\times(\mathcal{C}_{\ord}^{n+2}(\mathbb{P}^1))^{{\bf PSL}(2,\mathbb{C})}$. 
Respectively, an automorphism $F\in\widetilde N$ as in (i$'$) lifts first to a triangular automorphism 
%&
$$
F(\id, A)\in \Aut (\mathcal{C}_{\ord}^{n-1}(\mathbb{C}^{**})\times\mathbb{C}^*),\quad Q=(Q',q)\mapsto
(Q', f(Q')q)\,\,\,\mbox{for all}\,\,\, Q\in\mathcal{C}_{\ord}^{n-1}(\mathbb{C}^{**})\times\mathbb{C}^*,
$$
%&
 and then to
%&
$$
\tilde F\in\Aut (\mathcal{C}_{\ord}^{n+2}(\mathbb{P}^1)\times\mathbb{C}^*),\quad Q=(Q',q)\mapsto
(Q', \tilde f(Q')q)\,\,\,\mbox{for all}\,\,\,
Q\in\mathcal{C}_{\ord}^{n+2}(\mathbb{P}^1)\times\mathbb{C}^*\,,
$$
%&
where, as before, $\tilde f\in 
\mathcal{O}^\times(\mathcal{C}_{\ord}^{n+2}(\mathbb{P}^1))^{{\bf PSL}(2,\mathbb{C})}$ is the lift of $f$. 
Since $F\in\widetilde N$, the automorphism $
\tilde F$ normalizes the subgroup $\mathbf{S}(n+1)={\rm Stab}_{\mathbf{S}(n+2)}(n+2)\subset 
\mathbf{S}(n+2)$ acting naturally on $\mathcal{C}_{\ord}^{n+2}(\mathbb{P}^1)\times\mathbb{C}^*$ 
identically on the second factor. 
Hence for any $\sigma\in \mathbf{S}(n+1)$ there exists $\sigma'\in \mathbf{S}(n+1)$ 
such that 
$\tilde F\circ\sigma=\sigma'\circ\tilde F$, where for any $Q=(Q',q)\in\mathcal{C}_{\ord}^{n+2}
(\mathbb{P}^1)\times\mathbb{C}^*$ we have
%&
$$
\tilde F\circ\sigma\colon Q\mapsto \left(\sigma(Q'),\tilde f(\sigma(Q'))q\right)\quad\mbox{and}
\quad \sigma'\circ\tilde F\colon Q
\mapsto \left(\sigma'(Q'),\tilde f(Q')q\right)\,.
$$
%&
Thus the equality $\tilde F\circ\sigma=\sigma'\circ\tilde F$ holds if and only if $\sigma=\sigma'$ and 
$\tilde f\circ\sigma=\tilde f$. Hence $\tilde f$ is  
$\mathbf{S}(n+1)$-invariant, and $\tilde F$ commutes with the $\mathbf{S}(n+1)$-action
on $\mathcal{C}_{\ord}^{n+2}(\mathbb{P}^1)\times\mathbb{C}^*$. Since $\widetilde{\mathbf{S}}=
\eta\mathbf{S}(n+1)\eta^{-1}\subset\Aut\mathcal{C}_{\ord}^{n}(\mathbb{C}^*)$ 
(see (\ref{CD: diagram over T(0,m+3) to Com(C**)})),
it follows that the function $f\in\mathcal{O}^\times(\mathcal{C}_{\ord}^n(\mathcal{C}^*))^{\mathcal{C}^*}$ 
is $\widetilde{\mathbf{S}}$-invariant. 

For an automorphism $F\in\widetilde N$ as in (ii$'$), the same argument  shows that the function 
$q_n^{-2}f$ is $\widetilde{\mathbf{S}}$-invariant, as reqired.
\eproof

The next proposition 
ends the proof of Theorem \ref{thm: 5.1-a}. 

\bprop\label{prop: end-of-the-proof-thm-8.1} In the notation as in \ref{sit: proof-of-thm-8.1}--\ref{lem: N2-bis} 
we have $\widetilde N=\mathbb{C}^*\times\mathbf{S}(n+1)$, where the subgroup 
$\mathbb{C}^*\subset \widetilde N$ acts on $\mathcal{C}_{\ord}^{n}(\mathbb{C}^*)$ via  
$(s, Q)\mapsto sQ$ for $s\in\mathbb{C}^*$ and $Q\in \mathcal{C}_{\ord}^{n}(\mathbb{C}^*)$. 
\eprop

\bproof 
Consider first $F\in\widetilde N$ as in (\ref{eq: i-ii}(i)). By Lemma \ref{lem: PSL2-invar} 
the function $f$  in (\ref{eq: i-ii}(i)) is $\widetilde{\mathbf{S}}$-invariant.  
Hence $f$ is $\mathbf{S}(n)$-invariant, where $\mathbf{S}(n)\subset\widetilde{\mathbf{S}}$ 
acts on $\mathcal{C}_{\ord}^{n}(\mathbb{C}^*)$
via permutations of coordinates. It follows that $f=s\tilde h_n^k$ for some $s\in\mathbb{C}^*$ 
and $k\in\Z$, see \ref{sit: tilde h}. Furthermore, $f$ is $\sigma_n$-invariant, where $\sigma_n\in\mathbf{S}(n+1)$ 
acts on $\mathcal{C}_{\ord}^{n}(\mathbb{C}^*)$ via the involution $\varrho$ in (\ref{eq: invol-sigma-n}).
The latter invariance translates as the identity
%&
$$
s\tilde h_n^k(Q)=sh_n^k(\varrho(Q))\,,
$$
%&
which is definitely wrong unless $k=0$, and so, $f=s\in\mathbb{C}^*$. 
Thus, in case (i) we have $F\in\mathbb{C}^*$. 
 
It remains to eliminate the possibility that 
some $F$ as in (\ref{eq: i-ii}(ii)) belongs to $\widetilde N$. 
By Lemma  \ref{lem: PSL2-invar}, in this case the function 
$q_n^{-2}f\in \mathcal{O}^\times
(\mathcal{C}_{\ord}^{n}(\mathbb{C}^*))$ is 
$\widetilde{\mathbf{S}}$-invariant, and, in particular, 
$\mathbf{S}(n)$-invariant, while $f$ is $\mathbb{C}^*$-invariant. 
However, for $n>2$ these lead to a contradiction
 in the same way as in the proof of Proposition \ref{prop: end-Zinde-thm}, 
case (ii). Indeed, an argument in this proof shows that for $n=2$ the condition 
$q_n^{-2}f\in \mathcal{O}^\times
(\mathcal{C}_{\ord}^{n}(\mathbb{C}^*))^{\mathbb{C}^*\times\mathbf{S}(n)}$ implies the equality $q_n^{-2}f=sd_2^k/z_2^{k+1}$ for some $s\in \mathbb{C}^*$ and $k\in\Z$. The latter function is also $\varrho$-invariant. Using (\ref{eq: invol-sigma-n}), this can be
translated as the identity 
%&
$$
(-1)^kq_1^{3k+1}=(q_1-q_2)^{3k+1}\quad\mbox{for all}\quad (q_1,q_2)\in
\mathcal{C}_{\ord}^{2}(\mathbb{C}^*)\,,
$$
%&
which is definitely wrong whatever is $k\in\Z$.
\eproof
 
\section{The group $\Aut\mathcal{SC}^2_{\rm blc}$}\label{sec: Kaliman for n=4}
In this section we prove Kaliman's Theorem \ref{Thm: Kaliman-Lin-Zinde}(b) 
in the remaining case $n=4$. Let us repeat this statement.

\begin{thm}\label{thm: Kaliman for n=4}
 We have $\Aut\mathcal{SC}^2_{\rm blc}\cong\Z/12\Z$, 
where $\zeta\in\Z/12\Z$ acts on 
$\mathcal{SC}^2_{\rm blc}$ via $Q\mapsto \zeta Q$ for $Q\in \mathcal{SC}^2_{\rm blc}$. 
\end{thm}

The proof is done in Subsection \ref{ss: proof of Kaliman theorem  for n=4}.  
In Subsection \ref{ss: elliptic-pencil}
we construct an elliptic fibration $\mathcal{SC}^2_{\rm blc}\to\mathcal{SC}^1_{\rm blc}$ 
over 
an elliptic curve $\mathcal{SC}^1_{\rm blc}$. Its fibers and the base curve are 
smooth
affine plane cubics with one place at infinity. 
The group $\mu_{12}$ of the roots of unity of order 12 acts naturally on $\mathcal{SC}^2_{\rm blc}$
preserving the fibration. The action of 
$-1\in\mu_{12}$ yields the hyperelliptic involution on each fiber. 
We show in Subsection \ref{ss: proof of Kaliman theorem  for n=4} 
that any automorphism $F\in\Aut\mathcal{SC}^2_{\rm blc}$ 
acts as an element of $\mu_{12}$. This gives a proof of Theorem \ref{thm: Kaliman for n=4}. 

\subsection{Elliptic fibration of $\mathcal{SC}^2_{\rm blc}$}
\label{ss: elliptic-pencil}
Given a quartic polynomial
%&
\begin{equation}\label{eq: polyn-f}
f(X) =X^4 + z_2X^2 + z_3X + z_4\,
\end{equation} 
%&
we consider its cubic resolvent\footnote{In fact, the choice of a cubic 
resolvent is irrelevant for our purposes.}
%&
$$
R_3(X) = X^3 +v_1X^2 + v_2X +v_3\,,
$$
%&
where 
%&
$$
v_1=-z_2,\,\,v_2=- 4z_4,\quad\mbox{and}\quad v_3= 4z_2z_4-z_3^2\,.
$$
%&
If $q_1,\ldots,q_4$ are the roots of $f$, then, up to reordering,  
the roots of $R_3$ are
$$
\lambda_1=q_1q_2 +q_3q_4,\,\, \lambda_2=q_1q_3 +q_2q_4,
\quad\mbox{and}\quad \lambda_3=q_1q_4 +
q_2q_3\,.
$$
 We have discr$\,R_3 =$ discr $f$.

The Tschirnhausen transformation gives the cubic polynomial
\begin{equation}\label{eq: polyn-g}
g(Y) =R_3(Y+z_2/3)=
Y^3 + u_2Y +u_3,\,
\end{equation}
where
\begin{equation}\label{eq: u2}
u_2=-z_2^2/3- 4z_4
\quad\mbox{and}\quad
u_3 =
8z_2z_4/3
-2z_2^3/27-z_3^2\,.
\end{equation}
Once again, we have discr $g =$ discr$\,R_3 =$ discr $f$. 

The balanced special configuration space $\mathcal{SC}^2_{\rm blc}$ can be realized as
the surface  in $\mathbb{C}^3$ with coordinates
$(z_2,z_3,z_4)$
given by equation discr $f=1$, and the  balanced special configuration
space $\mathcal{SC}^1_{\rm blc}$ as the curve  in
$\mathbb{C}^2$  with coordinates $(u_2, u_3)$ given by equation 
%&
\begin{align}\label{eq: discr g}
{\rm discr}\,g=-(4u_2^3+27u_3^2)=1\,,
\end{align} 
%&
where $f$ and $g$ are as in (\ref{eq: polyn-f}) and 
(\ref{eq: polyn-g}) respectively.
 Clearly, 
$\mathcal{SC}^2_{\blc}$ is a smooth affine surface and
$\mathcal{SC}^1_{\blc}$  is 
a smooth affine elliptic cubic curve with one place at infinity 
and  zero $j$-invariant.
Since discr $g =$ discr $f$, the correspondence 
%&
$$\pi\colon \mathcal{SC}^2_{\blc}\to \mathcal{SC}^1_{\blc},\quad
f\mapsto g\,, 
$$
%&
yields a surjective morphism 
 given by formulas (\ref{eq: u2}). 

\begin{lem}\label{lem: elliptic fibration} In the notation as before, 
the morphism
$
\pi\colon  \mathcal{SC}^2_{\blc}\to \mathcal{SC}^1_{\blc}\,
$
yields an elliptic fibration on $\mathcal{SC}^2_{\blc}$. 
The fiber $E(P)=\pi^*(P)$ over a point $P=(u_2, u_3)\in\mathcal{SC}^1_{\blc}$ is a 
smooth, reduced 
elliptic cubic with one place at infinity and with 
$j(E(P))=2^8 3^3 u_2^3$.
\end{lem}

\bproof Since discr $g =$ discr $f$, 
plugging the expressions for $u_2$ and $u_3$ from (\ref{eq: u2}) into 
(\ref{eq: discr g}) gives the following equation of $\mathcal{SC}^2_{\blc}$ in $\mathbb{C}^3_{(z)}$,
where $z=(z_2,z_3,z_4)$:
%&
\begin{align}\label{eq: S} 
4(-z_2^2/3- 4z_4)^3+27
(8z_2z_4/3
-2z_2^3/27-z_3^2)^2=-1\,.
\end{align} 
%&
We can equally realize $\mathcal{SC}^2_{\blc}$ in $\mathbb{C}^5_{(z,u)}$, where 
$u=(u_2,u_3)$,  as intersection of three hypersurfaces given 
by equations (\ref{eq: u2}) and (\ref{eq: discr g}).
It is convenient however to simplify the latter system by
eliminating the variable $z_4$ from (\ref{eq: u2}). In this way we
arrive at the relation 
\begin{align}\label{eq: z4}
z_3^2  =
Z_2^3+g_2Z_2+g_3\,,
\end{align} 
where 
$$
Z_2=-2z_2/3, \,\, \,\,g_2=u_2,\quad\mbox{and}\quad g_3=-u_3\,.
$$
This yields an embedding of $\mathcal{SC}^2_{\blc}$ in $\mathbb{C}^4$ with coordinates 
$(z_2,z_3,u_2,u_3)$ onto the complete intersection of two hypersurfaces given by 
(\ref{eq: discr g}) and (\ref{eq: z4}). Then the morphism $\pi\colon \mathcal{SC}^2_{\blc}\to \mathcal{SC}^1_{\blc}$
coincides with the restriction to $\mathcal{SC}^2_{\blc}$ of the standard projection 
$(z_2,z_3,u_2,u_3)\mapsto (u_2,u_3)$. 

For a general point $(u_2,u_3)\in\mathbb{C}^2$ equation 
(\ref{eq: z4}) defines a smooth elliptic 
 cubic curve $E=E(u_2,u_3)$ in $\mathbb{C}^2_{(z_2,z_3)}$
with one place at infinity.
The member of this family of elliptic curves in $\mathbb{C}^4$ over a point $(u_2,u_3)\in\mathbb{C}^2$ is  singular 
if and only if the polynomial
 $h(Z)=Z_2^3+g_2Z_2+g_3$ has a multiple root, 
if and only if its discriminant vanishes, i.e.,
\begin{align}\label{eq: discr h}
{\rm discr}\,h=-(4g_2^3+27g_3^2)=-(4u_2^3+27u_3^2) =0\,.
\end{align} 
Any nonsingular member is a reduced and irreducible elliptic cubic with one place at infinity.
Since (\ref{eq: discr g}) and (\ref{eq: discr h}) are incompatible, 
the induced elliptic fibration on the surface
 $\mathcal{SC}^2_{\blc}\subset \mathbb{C}^4$ has no degenerate fiber, as stated.

Finally, the formula for the $j$-invariant in the lemma is the classical one, 
where the denominator disappears because of the equalities 
discr $h=$ discr $g=1$. 
\eproof

\subsection{Proof of Theorem \ref{thm: Kaliman for n=4}}\label{ss: proof of Kaliman theorem  for n=4}
In this subsection we give a proof of Theorem \ref{thm: Kaliman for n=4}. 

We let $\mu_n$ stand for the group of $n$th roots 
of unity acting on $\mathbb{C}$ in a natural way. 
The group
$\mu_3\times\mu_2=\mu_6\cong\Z/6\Z$ acts on $\mathcal{SC}^1_{\blc}$ via
$$
(u_2,u_3)\mapsto (\zeta u_2, \xi u_3),\quad\mbox{where}\quad 
\zeta\in\mu_3,\,\,\,\xi\in\mu_2\,,
$$
or, in other terms, via
\begin{equation}\label{eq: action-Gamma}
(u_2,u_3)\mapsto (\theta^2 u_2, \theta^3 u_3),\quad\mbox{where}\quad \theta\in\mu_6\,.
\end{equation}
Recall that $j(\mathcal{SC}^1_{\blc})=0$ and $\Aut\mathcal{SC}^1_{\blc}=\mu_6$ with the action of $\mu_6$ on $\mathcal{SC}^1_{\blc}$ 
as in (\ref{eq: action-Gamma}); see, e.g., 
\cite[ Ch.\ 12, Remark 4.8]{Husemoller} or \cite[Ch.\ IV, Corollary 4.7]{Ha}.
 
 We claim that any 
 automorphism $F\in\Aut \mathcal{SC}^2_{\blc}$ preserves the elliptic fibration  
$\pi\colon \mathcal{SC}^2_{\blc}\to\mathcal{SC}^1_{\blc}$ as in Lemma \ref{lem: elliptic fibration}. 
Indeed, our family $E(P)$, $P\in\mathcal{SC}^1_{\blc}$, 
is not isotrivial, 
i.e., the $j$-invariant $j(E(P))=2^8 3^3 u_2^3$ is a non-constant 
function of $P\in\mathcal{SC}^1_{\blc}$. 
If $P=P_0^\pm=(0,\pm u_3)\in\mathcal{SC}^1_{\blc}$, where $u_3=i\sqrt{3}/9$, 
then $j(E(P))=0$, and if $P=(u_2,0)\in\mathcal{SC}^1_{\blc}$, 
where $u_2^3=-1/4\in\mathbb{R}$, 
then $j(E(P))=12^3=1728$. For a general point 
$P\in\mathcal{SC}^1_{\blc}$, the 
$j$-invariant $j(E(P))$
 is different from $0$ and $1728$.

Assume that there is an automorphism $F\in\Aut \mathcal{SC}^2_{\blc}$ that does not preserve
 the fibration $\pi$. 
Then for a general fiber $E(P)$, the morphism
$\pi\circ F|_{E(P)}\colon E(P)\to\mathcal{SC}^1_{\blc}$ is non-constant. 
It extends to the projectivizations $\overline{E(P)}$ 
and $\overline{\mathcal{SC}^1_{\blc}}$ of $E(P)$ and $\mathcal{SC}^1_{\blc}$, respectively,
sending the point at infinity of $E(P)$ to the point at infinity  of $\mathcal{SC}^1_{\blc}$.  
Thus $\pi\circ F|_{E(P)}\colon E(P)\to\mathcal{SC}^1_{\blc}$ is an isogeny. 
However, the set of all elliptic curves isogeneous to $\mathcal{SC}^1_{\blc}$ is countable, 
see \cite[Ch.\ IV, Exercise 4.9.b]{Ha}. 
This yields a contradiction.

Thus for any $P\in\mathcal{SC}^1_{\blc}$ we have $F(E(P))=E(P')$ for some point $P'\in\mathcal{SC}^1_{\blc}$. 
The correspondence $\varphi\colon P\mapsto P'$ defines an automorphism of $\mathcal{SC}^1_{\blc}$. 
This gives a homomorphism $\rho\colon\Aut \mathcal{SC}^2_{\blc}\to\Aut \mathcal{SC}^1_{\blc}$, 
$F\mapsto\varphi$, where $\varphi\in
\Aut\mathcal{SC}^1_{\blc}=\mu_6$ acts as in (\ref{eq: action-Gamma}) 
for a certain $\theta\in\mu_6$. 

Consider the action  on $\mathcal{SC}^2_{\blc}$ of the cyclic group $\mu_{12}\cong\Z/12\Z$ via
\begin{equation}\label{eq: action-S}
(z_2,z_3,u_2,u_3)\mapsto (\zeta^2z_2,\zeta^3z_3,\zeta^4u_2,\zeta^6u_3),\quad\mbox{where}
\quad \zeta\in\mu_{12}\,.
\end{equation}
The projection $\pi\colon \mathcal{SC}^2_{\blc}\to\mathcal{SC}^1_{\blc}$ induces the surjection $\mu_{12}\to\mu_6$, $\zeta\mapsto\zeta^2$. 
It follows that $\rho\colon\Aut \mathcal{SC}^2_{\blc}\to\mu_6$ is  a surjection. Theorem \ref{thm: Kaliman for n=4} claims that in fact 
$\Aut \mathcal{SC}^2_{\blc}=\mu_{12}$. To confirm this claim, it suffices to establish the equality  $\ker\rho=\mu_2$, where $\mu_2$ 
acts on $\mathcal{SC}^2_{\blc}$ via (\ref{eq: action-S}) with  $\zeta\in\{1,-1\}$.  
This action restricts  to the hyperelliptic involution on any fiber of $\pi$.

Suppose that $F\in\ker\rho$, i.e., $\varphi= {\rm id}$, and so, $F$ preserves each fiber of $\pi$. 
For a fiber  $E=E(P)$, 
the automorphism $\alpha=F|_E\in\Aut E$ extends to an automorphism $\bar\alpha$ of the 
projectivization $\bar E$ of $E$. The extended automorphism $\bar\alpha$ fixes the unique 
point of $\bar E$ at infinity. This point is a flex of the cubic $\bar E$ that can be chosen 
for zero of the group low on $\bar E$.  
Any automorphism of $\bar E$ that fixes the zero point  
is a group automorphism. 
Hence also $\bar\alpha$ is. 

For a general point $P\in\mathcal{SC}^1_{\blc}$, the value $j(E(P))$ is different from $0$ 
and 
$1728$. So,  the hyperelliptic involution of $E(P)$ is a unique non-identical 
automorphism 
preserving the point at infinity, see \cite[Ch.\ IV, Corollary 4.7]{Ha}. 
Thus $\alpha=F|_E\in\Aut E$
is either identical or the hyperelliptic involution. Hence the kernel 
$\ker\rho=\mu_2$ is contained in $\mu_{12}$.
It follows that the group $\Aut \mathcal{SC}^2_{\blc}$ coincides with $\mu_{12}$ acting on 
$\mathcal{SC}^2_{\rm blc}$ via (\ref{eq: action-S}), where 
the latter action is induced by the action $Q\mapsto \zeta Q$ 
for $\zeta\in\mu_{12}$ and 
$Q\in\mathcal{SC}^2_{\rm blc}$. This proves Theorem \ref{thm: Kaliman 
for n=4}. \qed

\vskip 3mm

The following 
example is essentially borrowed in \cite[\S 8, 2.2]{Lin79} 
(cf.\ \cite[\S 14.1]{Lin04b}). 
It shows that Kaliman's theorem 
on endomorphisms does not hold any longer for 
$n=4$, although it does hold for automorphisms  in this case as well, see Theorem  
\ref{thm: Kaliman for n=4}. 

\begin{exa} \label{ex: Lin's counterexample} 
Consider 
the endomorphism 
%&
\be\label{eq: F}
F\colon \mathcal{SC}^2_{\blc}\stackrel{\pi}{\longrightarrow} 
\mathcal{SC}^1_{\blc}\stackrel{\varphi}{\longrightarrow} 
E(P_0^+)\hookrightarrow \mathcal{SC}^2_{\rm blc}\,,
\ee
%& 
where 
$\varphi$ is an isomorphism of $\mathcal{SC}^1_{\blc}$ 
onto the fiber $E(P_0^+)$  over the point $P_0^+=(0,i\sqrt{3}/9)\in\mathcal{SC}^1_{\blc}$.
Such an isomorphism exists since $j(E(P_0^+))=j(\mathcal{SC}^1_{\blc})=0$. 
In contrast with the case $n\neq 4$ in 
Kaliman's theorem, $F$ is not an automorphism. 

More explicitly, an endomorphism $F$ as in (\ref{eq: F}) can be given as follows. Note that the curve $E(P_0^+)$ on the surface $\mathcal{SC}^2_{\blc}\subset\mathbb{C}^3_{(z)}$
is contained in the complete intersection given  by equations (\ref{eq: S}) of $\mathcal{SC}^2_{\blc}$ and $u_2=0$. The latter intersection is a disjoint union of the fibers $E(P_0^\pm)$ of $\pi$ with $j=0$. The equation $u_2=0$ is equivalent to $12z_4+z_2^2=0$, see (\ref{eq: u2}). Consider the family $(F_{a,b})$ of endomorphisms of $\mathbb{C}^3_{(z)}$  defined by
%&
$$
F_{a,b}\colon (z_2,z_3,z_4)\stackrel{\pi}{\mapsto} (u_2,u_3)\stackrel{\phi}{\mapsto} \left(au_2,bu_3,-\frac{(au_2)^2}{12}\right)\,,
$$
%&
where $a,b\in\mathbb{C}$ and
$u_2, u_3$ are given by formulas (\ref{eq: u2}). In other terms,
%& 
$$
F_{a,b}\colon f=X^4+z_2X^2+z_3X+z_4\mapsto F_{a,b}(f)=X^4+au_2X^2+bu_3X-\frac{(au_2)^{2}}{12}\,.
$$
%&
A simple computation shows that $\discr(F_{a,b}(f))=-(1/27)(8a^3u_2^3+27b^2u_3^2)^2$. Choosing the constants $a$ and $b$ 
such that $a^3=\frac{i3\sqrt{3}}{2}$ and $b^2=i3\sqrt{3}$ we obtain the equality
%&
$$
-(1/27)(8a^3u_2^3+27b^2u_3^2)^2=(4u_2^3+27u_3^2)^2\,,
$$
%& 
and so, 
%&
$$
\discr(F_{a,b}(f))=(\discr(g))^2=(\discr(f))^2\,,
$$
%&
 see (\ref{eq: discr g}).
With this choice of constants,
$F=F_{a,b}$ yields an endomorphism with one-dimensional fibers of every one of the spaces $\mathcal{C}^3_{\blc}$, $\mathcal{SC}^2_{\blc}$, and $\Sigma^2_{\blc}\subset\mathbb{C}^3_{(z)}$. Moreover, we have $F(\mathcal{SC}^2_{\blc})\subset E(P_0^+)\cup E(P_0^-)$. Acting with the subgroup $\mu_2\times\mu_3\subset\mu_{12}=\Aut\mathcal{SC}^2_{\blc}$ 
(see (\ref{eq: action-S})) we can achieve in addition that $F(\mathcal{SC}^2_{\blc})=E(P_0^+)$, and so,  $F$ fits in  (\ref{eq: F}).
\end{exa}

The next example is an algebraic counterpart of the previous one.

\begin{exa} \label{ex: Lin's counterexample-2} 

Recall that for $n>4$, the subgroup 
$\mathbf{A}'_{n-1}\cap\mathbf{PA}_{n-1}\subset \mathbf{A}'_{n-1}$ 
is stable under any endomorphism of $\mathbf{A}'_{n-1}$, 
see the discussion following Theorem \ref{Thm: Kaliman-Lin-Zinde}. 
This does not hold any longer for $n=4$. Indeed, recall the presentation 
(\cite[Theorem 2.1]{G-L69})
%&
$$
\mathbf{A}'_3=\big\langle s,t,u,v\,|\,usu^{-1}=t, \,\,
utu^{-1}=t^2s^{-1}t, \,\,
vsv^{-1}=s^{-1}t,\,\,
vtv^{-1}=(s^{-1}t)^3s^{-2}t\big\rangle\,,
$$
%&
where 
%&
$$
s=\sigma_3\sigma_1^{-1},\,\,
t=\sigma_2\sigma_3\sigma_1^{-1}\sigma_2^{-1},\,\,
u=\sigma_2\sigma_1^{-1}, \,\,
v=\sigma_1\sigma_2\sigma_1^{-2}\,.
$$
%&
We have $\mathbf{A}'_3=T\rtimes V$, where $T=\langle s,t\rangle\cong
\mathbb{F}_2$ is a normal subgroup of $\mathbf{A}'_3$ 
and $V=\langle u,v\rangle\cong
\mathbb{F}_2$. Consider the composition $f=i\circ p\in\End \mathbf{A}'_3$, 
where $p\colon \mathbf{A}'_3\to V=\mathbf{A}'_3/T$ is the quotient 
morphism  and $i\colon V\stackrel{\cong}{\longrightarrow} T
\hookrightarrow  \mathbf{A}'_3$ an isomorphism onto the subgroup $T$.
We have
%&
$$
f\colon s\mapsto 1, \,\, t\mapsto 1, \,\,u \mapsto s, \,\, v\mapsto t\,.
$$
%&
We claim that $uv\in \mathbf{PA}_3$, whereas
$f(uv)=st\not\in \mathbf{PA}_3$. Indeed, the images of $u$ and $v$ in
the alternating group $\mathbf{A}(4)\subset\mathbf{S}(4)$ 
are mutually inverse three-cycles, while the images of 
$s$ and $t$ are products of independent transpositions, which
generate the Klein four-group $K\subset \mathbf{A}(4)$. 
Thus, the subgroup $\mathbf{A}'_3\cap\mathbf{PA}_3$ of 
$\mathbf{A}'_3$ is not stable under $f$.

We claim that, likewise $f$, 
the endomorphism $F_*\in\End\pi_1(\mathcal{SC}_{\blc}^2)
=\End\mathbf{A}'_3$ 
does not preserve 
the intersection $\mathbf{A}'_3\cap\mathbf{PA}_3$. 
Implicitly, this follows from the proof of Kaliman's Theorem in 
\cite[Theorem 12.13]{Lin04b}.
Indeed, this proof shows that, if $H$ is an endomorphism of 
$\mathcal{SC}_{\blc}^2$ such that $H_*$ preserves 
the subgroup $\mathbf{A}'_3\cap\mathbf{PA}_3$, then 
$H\in\Aut\mathcal{SC}_{\blc}^2$. Let us give a direct proof of our claim,
which uses some ideas from \cite{G-L69}.

The homotopy exact sequence 
of the fiber bundle $\pi\colon \mathcal{SC}^2_{\blc}\to\mathcal{SC}^1_{\blc}$ with general 
fiber $E=E(P)$ is
%&
$$
1\to \pi_1(E)\to \pi_1(\mathcal{SC}^2_{\blc})
\stackrel{\pi_*}{\longrightarrow} 
\pi_1(\mathcal{SC}^1_{\blc})\to 1\,.
$$
%&
 This leads to 
a  semi-direct product decomposition 
$\mathbf{A}'_3\cong \pi_1(\mathcal{SC}^2_{\blc})=B\rtimes A$, where 
$A=\pi_1(\mathcal{SC}^1_{\blc})\cong \mathbb{F}_2$ and 
$B=\pi_1(E)\cong\mathbb{F}_2$, cf.\ 
\cite[Corollary 2.7]{G-L69} (cf.\ also \cite[\S 14.1]{Lin04b}).
Let us show that $T= B$. 

Indeed, $T$ is 
the intersection of the members of the lower 
central series of the group
$\mathbf{A}'_3$, see \cite[Theorem 2.10.a]{G-L69}. 
Hence the image of $T$ in the quotient group $\mathbf{A}'_3/B\cong A$ is 
the intersection of the members of the lower 
central series of the group $A\cong \mathbb{F}_2$. The latter intersection
is trivial due to a theorem of Magnus (\cite{Magnus}; 
see also \cite[Ch.\ IX, \S 36]{Kurosh}). 
Thus $T\subset B$.
Hence $V=\mathbf{A}'_3/T\cong (B/T)\rtimes A$. Due 
to another theorem of Magnus 
(\cite[\S 5, VIII]{Magnus}; see also \cite[Theorem 41.52]{Neumann}), 
a free group of finite rank is Hopfian, i.e., it does not admit 
an isomorphic proper quotient group.  Since $A\cong V\cong\mathbb{F}_2$, 
this implies that $B/T=1$, and so, $T=B$. 

It follows from our
 constructions 
that the endomorphisms 
$F_*$ and  $f$ of $\mathbf{A}'_3$ with the same image $T=B$ and the same kernel
  $T=B$ differ by an 
automorphism, say, $\alpha$ of $T$. The image of $T$ 
in the alternating group 
$\mathbf{A}(4)$ is the  Klein four-group $K$.
The images of $\alpha(s)=F_*(u)$ and $\alpha(t)=F_*(v)$ generate $K$. 
Hence the image of $F_*(uv)=\alpha(st)$ in $K$ is different from $1$. 
Once again, we have $uv\in \mathbf{PA}_3$, while 
$F_*(uv)\not\in\mathbf{PA}_3$. Thus, the subgroup $\mathbf{A}'_3\cap\mathbf{PA}_3$ of 
$\mathbf{A}'_3$ is not stable under the endomorphism 
$F_*\in\End \mathbf{A}'_3$.
 \end{exa}

\section{Holomorphic endomorphisms of the balanced configuration space}
\label{Sec: Holomorphic endomorphisms of Cblc(n-1)}

In this section, $\mathcal{O}_{\hol}^{\times}(\mathcal{Z})$ stands for
the multiplicative group of the algebra $\mathcal{O}_{\hol}(\mathcal{Z})$
of all holomorphic functions on a complex space $\mathcal{Z}$, and  
$\mathcal{O}_{\hol,+}(\mathcal{Z})$ for its additive group.

By (\ref{eq: 3 cylinders}), any holomorphic endomorphism $f$ of 
$\mathcal{C}^{n-1}_{\blc}$
extends to a holomorphic endomorphism of $\mathcal{C}^n$. Such an
extension is non-Abelian whenever $f$ is non-Abelian\footnote{The latter 
means
 that the image of the induced endomorphism of the corresponding 
fundamental group 
is non-Abelian, see the Introduction.}. 
The {\em minimal} extension $F$ given by $F(Q)=f(Q-\bc(Q))$ for all 
$Q\in\mathcal{C}^n$ 
maps $\mathcal{C}^n$ to 
$\mathcal{C}^{n-1}_{\blc}\subset\mathcal{C}^n$, see (\ref{eq: pi}).

Among the affine transformations of $\mathbb{C}$ acting diagonally on
$\mathcal{C}^n$, only the elements of the multiplicative subgroup
$\mathbb{C}^*\subset\Aff\mathbb{C}$ fixing the origin
$0\in\mathbb{C}$ preserve the balanced configuration space
$\mathcal{C}^{n-1}_{\blc}\subset\mathcal{C}^n$. Let $\mathcal{S}$
denote this $\mathbb{C}^*$-action on each of the spaces
$\mathcal{C}^n$ and $\mathcal{C}^{n-1}_{\blc}$, and let
$\mathcal{O}_{\hol}^{\mathcal{S}}(\mathcal{C}^{n-1}_{\blc})$ be the
subalgebra of $\mathcal{O}_{\hol}(\mathcal{C}^{n-1}_{\blc})$ of 
all $\mathcal{S}$-invariant functions.

For any configuration ${Q^\circ}\in\mathcal{C}^{n-1}_{\blc}$ its
$\mathbb{C}^*$-stabilizer
$\textup{St}_{\mathbb{C}^*}(Q^\circ)=\{\zeta\in\mathbb{C}^*\,|\
\zeta\cdot {Q^\circ}={Q^\circ}\}$ is a cyclic rotation subgroup in
$\mathbb{C}^*$ of order\,$\le{n}$ permuting elements of $Q^\circ$.
If $n\ge 3$, it follows that the set
$\{Q^\circ\in\mathcal{C}^{n-1}_{\blc}\mid
\textup{St}_{\mathbb{C}^*}(Q^\circ)\ne\{1\}\}$  is a Zariski
closed subset in $\mathcal{C}^{n-1}_{\blc}$ of dimension $1$ and
$\{Q^\circ\in\mathcal{C}^{n-1}_{\blc}\mid
\textup{St}_{\mathbb{C}^*}(Q^\circ)=\{1\}\}$ is a Zariski open
dense subset of $\mathcal{C}^{n-1}_{\blc}$.

\begin{defi}\label{def: C*-tame map}
We say that a holomorphic self-map $f$ of
$\mathcal{C}^{n-1}_{\blc}$ is $\mathbb{C}^*$-{\em tame}, if there
is a holomorphic function
$h\colon\,\mathcal{C}^{n-1}_{\blc}\to\mathbb{C}^*$ such that
$f(Q^\circ)=h(Q^\circ)\cdot{Q}^\circ$ for all
$Q^\circ\in\mathcal{C}^{n-1}_{\blc}$.
\vskip4pt

Note that the cohomology group
$H^1(\mathcal{C}^{n-1}_{\blc},\mathbb{Z})\cong\mathbb{Z}$ of the
Stein manifold $\mathcal{C}^{n-1}_{\blc}$ is generated by the
cohomology class of the discriminant $D_n|_{\mathcal{C}^{n-1}_{\blc}}$ 
(see (\ref{eq: discriminant of Q})) restricted to
$\mathcal{C}^{n-1}_{\blc}$. Hence any function
$h\in\mathcal{O}_{\rm hol}^\times(\mathcal{C}^{n-1}_{\blc})$ can be
written as $h=e^{\chi}D_n^m$ with some
$\chi\in\mathcal{O}_{\hol}(\mathcal{C}^{n-1}_{\blc})$ and
$m\in\mathbb{Z}$.
\end{defi}

The results below, stated in \cite{Lin72b} and \cite[Sec. 8.2.1]{Lin79},
are simple consequences of the analytic counterpart of Tame Map Theorem (see \cite{Lin04b} 
or \cite{Lin11} for the proof) and the facts mentioned above.

\begin{thm}\label{Thm: Non-Abelian endomorphisms of Cblc(n-1)}
For $n>4$ every non-Abelian holomorphic self-map $f$ of
$\mathcal{C}^{n-1}_{\blc}$ is ${\mathbb C}^*$-tame, i.e., it can be given by
%&
\begin{equation}\label{eq: general form of f: Cblc(n-1) to Cblc(n-1)}
f({Q^\circ})=\mathcal{S}_{e^{\chi({Q^\circ})}
D_n^m({Q^\circ})}{Q^\circ}= e^{\chi({Q^\circ})}
D_n^m({Q^\circ})\cdot {Q^\circ} \ \ \text{\rm for all} \ \
{Q^\circ}\in\mathcal{C}^{n-1}_{\blc}\,,
\end{equation}
%&
where $\chi\in\mathcal{O}_{\hol}(\mathcal{C}^{n-1}_{\blc})$
and $m\in\mathbb{Z}$.
\end{thm}

\begin{proof}
The map $f$ admits a holomorphic non-Abelian extension 
$F\colon\mathcal{C}^n\to\mathcal{C}^{n-1}_{\blc}\subset\mathcal{C}^n$
defined by $F(Q)=f(Q-\bc(Q))$ for all $Q\in\mathcal{C}^n$. 
By the analytic version of Tame Map Theorem, $F(Q)=A(Q)Q+B(Q)$ for all 
$Q\in\mathcal{C}^n$ with $A\in\mathcal{O}^\times_{\hol}(\mathcal{C}^n)$ and
$B\in\mathcal{O}_{\hol}(\mathcal{C}^n)$. Since
$\mathcal{C}^{n-1}_{\blc}\subset\mathcal{C}^n$ and
$\bc({Q^\circ})=0$ for any ${Q^\circ}\in\mathcal{C}^{n-1}_{\blc}$,
we see that
%&
$$
f({Q^\circ})=a({Q^\circ}){Q^\circ}+b({Q^\circ})\ \ \textup{for
all} \ \ {Q^\circ}\in\mathcal{C}^{n-1}_{\blc}\,,
%$$
%&
\ \ \textup{where} \ \
%&
%$$
a=A|_{\mathcal{C}^{n-1}_{\blc}} \ \ \textup{and} \ \
b=B|_{\mathcal{C}^{n-1}_{\blc}}\,.
$$
%&
Moreover, $b=0$. Indeed, the condition $\bc(f(Q^\circ))=\bc(Q^\circ)=0$
implies that
%&
$$
b({Q^\circ})=a({Q^\circ})\bc({Q^\circ})+b({Q^\circ})
=\bc(a({Q^\circ}){Q^\circ}+b({Q^\circ}))=\bc(f({Q^\circ}))=0\ \
$$
%&
for all $Q^\circ\in\mathcal{C}^{n-1}_{\blc}$ and
%&
$$
a=e^\chi D_n^m \ \ \textup{for some} \ m\in\mathbb{Z} \
\textup{and} \
\chi\in\mathcal{O}_{\hol}(\mathcal{C}^{n-1}_{\blc})\,.
$$
%&
This proves (\ref{eq: general form of f: Cblc(n-1) to Cblc(n-1)}).
\end{proof}

\begin{thm}\label{Thm: properties of C*-tame maps}
Let $n\ge{3}$, and let
$f=f_{\chi,m}\colon\,\mathcal{C}^{n-1}_{\blc}\to\mathcal{C}^{n-1}_{\blc}$
be   a holomorphic map as in
\textup{(\ref{eq: general form of f: Cblc(n-1) to Cblc(n-1)})}.
Then the following hold.

\textup{(a)} The map $f$ is surjective\,\footnote{In view of
Theorem \ref{Thm: Non-Abelian endomorphisms of Cblc(n-1)}, for $n>4$ {\em any} 
non-Abelian holomorphic endomorphism of $\mathcal{C}^{n-1}_{\blc}$ is surjective.}, 
and the set $f^{-1}(Q^\circ)$ is discrete for any
$Q^\circ\in\mathcal{C}^{n-1}_{\blc}$. This set consists of all
points $\omega\cdot Q^\circ$, where $\omega\in\mathbb{C}^*$ is any
root of the system of equations
%&
\begin{equation}\label{eq: equation for preimage}
\omega^{mn(n-1)+1}e^{\chi(\omega\cdot Q^\circ)} D_n^m(Q^\circ) \cdot
Q^\circ=Q^\circ\,,
\end{equation}
%&
which always has solutions.
\vskip3pt

\textup{(b)} The map $f$ is proper $($in the complex topology$)$ 
if and only if $\chi\in\mathcal{O}_{\hol}^{\mathcal{S}}({\mathcal{C}^{n-1}_{\blc}})$.
In this case $f\colon\,\mathcal{C}^{n-1}_{\blc}\to\mathcal{C}^{n-1}_{\blc}$ 
is a finite unramified cyclic holomorphic covering of degree
$N=mn(n-1)+1$. The corresponding normal subgroup
$f_*(\pi_1(\mathcal{C}^{n-1}_{\blc}))$ of index $N$ in the Artin
braid group $\mathbf{A}_{n-1}=\pi_1(\mathcal{C}^{n-1}_{\blc})$ consists of all the elements
$g=\sigma_{i_1}^{m_1}\cdot\ldots\cdot\sigma_{i_q}^{m_q}\in \mathbf{A}_{n-1}$
such that $N$ divides $m_1+\ldots+m_q$,  \textup{}where
$\{\sigma_1,\ldots,\sigma_{n-1}\}$ is the standard system of
generators of $\mathbf{A}_{n-1}$. Every two such coverings of the same degree
are equivalent. \vskip3pt

\textup{(c)} The map $f$ is a biholomorphic automorphism of
$\mathcal{C}^{n-1}_{\blc}$ if and only if it is of the form
$f({Q^\circ})=e^{\chi({Q^\circ})}\cdot {Q^\circ}$ for any
${Q^\circ}\in\mathcal{C}^{n-1}_{\blc}$ and some
$\chi\in\mathcal{O}_{\hol}^{\mathcal{S}}({\mathcal{C}^{n-1}_{\blc}})$.
Every automorphism is isotopic to the identity and
$\Aut_{\hol}\mathcal{C}^{n-1}_{\blc}\cong
\mathcal{O}_{\hol,+}^{\mathcal{S}}({\mathcal{C}^{n-1}_{\blc}})/2\pi
i\mathbb{Z}$. \vskip3pt

\textup{(d)} If $f$ is regular, then $\chi=\const$ and so
$f({Q^\circ})=cD_n^m({Q^\circ})\cdot {Q^\circ}$ for all
${Q^\circ}\in\mathcal{C}^{n-1}_{\blc}$, where $c\in\mathbb{C}^*$
and $m\in\mathbb{Z}$. Every biregular automorphism $f$ of
$\mathcal{C}^{n-1}_{\blc}$ is of the form $f({Q^\circ})=s\cdot
{Q^\circ}$, ${Q^\circ}\in\mathcal{C}^{n-1}_{\blc}$, where
$s\in\mathbb{C}^*$. In particular, the group of all biregular automorphisms 
$\Aut\mathcal{C}^{n-1}_{\blc}$ is isomorphic to $\mathbb{C}^*$.
\end{thm}

\begin{proof}
\textup{(a)} Given a configuration $Q^\circ\in\mathcal{C}^{n-1}_{\blc}$, we
set
%&
\begin{equation}\label{eq: function psi on C*}
\psi_{Q^\circ}(\omega)\Def\omega^{mn(n-1)+1}e^{\chi(\omega\cdot
{Q^\circ})} D_n^m(Q^\circ)\ \ \textup{for any} \
\omega\in\mathbb{C}^*\,.
\end{equation}
%&
Clearly
$\psi_{Q^\circ}\in\mathcal{O}^\times_{\hol}(\mathcal{C}^{n-1}_{\blc})$
and $\psi_{Q^\circ}\ne\const$, since $mn(n-1)+1\ne{0}$ and
$e^{\chi(\omega\cdot{Q^\circ})}$ cannot be a non-constant rational
function of $\omega\in\mathbb{C}^*$. Hence, by the Picard theorem,
$\psi_{Q^\circ}(\mathbb{C}^*)=\mathbb{C}^*$. According to
(\ref{eq: general form of f: Cblc(n-1) to Cblc(n-1)}), we have
%&
$$
f(\omega\cdot{Q^\circ})=e^{\chi(\omega\cdot{Q^\circ})}
D_n^m(\omega\cdot{Q^\circ})\cdot {\omega\cdot Q^\circ}
=\omega^{mn(n-1)+1}e^{\chi(\omega\cdot{Q^\circ})}
D_n^m(Q^\circ)\cdot Q^\circ=\psi_{Q^\circ}(\omega)\cdot{Q^\circ}\,.
$$
%&
Thus, taking $\omega\in\mathbb{C}^*$ such that
$\psi_{Q^\circ}(\omega)=1$, we see that
$Q^\circ\in{f}(\mathcal{C}^{n-1}_{\blc})$. Hence $f$ is
surjective. Furthermore, all such $\omega$ satisfy the system of
equations (\ref{eq: equation for preimage}). Since the stabilizer
$\textup{St}_{\mathbb{C}^*}(Q^\circ)$ is finite, all solutions
$\omega$ of (\ref{eq: equation for preimage}) form a finite union
of countable discrete subsets of $\mathbb{C}^*$. Thus the set
$f^{-1}(Q^\circ)$ is countable and discrete. \vskip3pt

\textup{(b)} If $f$ as in (\ref{eq: general form of f: Cblc(n-1)
to Cblc(n-1)}) is proper then $f^{-1}(Q^\circ)$ is finite for any
$Q^\circ\in\mathcal{C}_{\blc}^{n-1}$. This is possible only when
the exponent $\chi(\omega\cdot Q^\circ)$ in (\ref{eq: equation for
preimage}) and (\ref{eq: function psi on C*}) does not depend on
$\omega\in\mathbb{C}^*$, i.e., the function $\chi$ is
$\mathcal{S}$-invariant. Then, for any fixed $Q^\circ$, the function
(\ref{eq: function psi on C*}) takes the form
%&
$$
\psi_{Q^\circ}(\omega)=\widetilde{\psi}_{Q^\circ}(\omega)
\Def\omega^{mn(n-1)+1}e^{\chi(Q^\circ)}
D_n^m(Q^\circ)\,.
$$
%&
The latter function is homogeneous of degree $N=mn(n-1)+1$, 
and the equation $\widetilde{\psi}_{Q^\circ}(\omega)=1$ has
precisely $N$ distinct roots $\omega_1,\ldots,\omega_N$. 
If the stabilizer $\textup{St}_{\mathbb{C}^*}(Q^\circ)$ is trivial, then
$f^{-1}(Q^\circ)$ consists on $N$ distinct points
$\omega_1{Q^\circ},\ldots,\omega_N{Q^\circ}$. If
$\textup{St}_{\mathbb{C}^*}(Q^\circ)\ne\{1\}$, then, to find the preimage
 $f^{-1}(Q^\circ)$, we have to solve the inclusion
$\widetilde{\psi}_{Q^\circ}(\omega)\in\textup{St}_{\mathbb{C}^*}(Q^\circ)$.
Fix some $\omega_0$ such that $\widetilde{\psi}_{Q^\circ}(\omega_0)=1$,
take any solution $\omega\in\mathbb{C}^*$ of the above inclusion, and
set $\lambda=\omega/\omega_0$. 
Then 
%&
\begin{equation}\label{eq: equation for lambda}
\lambda^N=\left(\frac{\omega}{\omega_0}\right)^N
=\frac{\widetilde{\psi}_{Q^\circ}(\omega)}
{\widetilde{\psi}_{Q^\circ}(\omega_0)}
=\widetilde{\psi}_{Q^\circ}(\omega)\in\textup{St}_{\mathbb{C}^*}(Q^\circ)\,.
\end{equation}
%&
The preimage $f^{-1}(Q^\circ)$ of $Q^\circ$ consists of all configurations
$\omega{Q^\circ}=\omega_0\lambda{Q^\circ}$, where $\lambda$ runs
over all solutions of the inclusion (\ref{eq: equation for lambda}).
All such configurations $\omega_0\lambda{Q^\circ}$
form a periodic sequence $\omega_0\lambda^k{Q^\circ}$, $k\in\mathbb{Z}_{\ge\,0}$, 
with period $N$; therefore, this sequence contains precisely $N$ distinct elements. 
It follows easily from these 
facts that $f\colon\mathcal{C}^{n-1}_{\blc}\to\mathcal{C}^{n-1}_{\blc}$ is an
unramified cyclic covering of degree $N$.

The proof of  the other assertions in (b) is easy, and
we leave it to the reader.
\vskip3pt

\textup{(c)} For a given
$\chi\in\mathcal{O}^{\mathcal{S}}_{\hol}(\mathcal{C}_{\blc}^{n-1})$ we let
 $f_1({Q^\circ})=e^{\chi({Q^\circ})}{Q^\circ}$ and
$f_2({Q^\circ})=e^{-\chi({Q^\circ})}{Q^\circ}$.
It follows from the $\mathcal{S}$-invariance of $\chi$ that
$f_1(f_2(Q^\circ))=f_2(f_1(Q^\circ))=Q^\circ$ for every
${Q^\circ}\in\mathcal{C}_{\blc}^{n-1}$. Thus $f_1$ and $f_2$ are
mutually inverse biholomorphic automorphisms of
$\mathcal{C}_{\blc}^{n-1}$. To prove the converse note that any
automorphism is a proper map. According to Theorem \ref{Thm:
Non-Abelian endomorphisms of Cblc(n-1)} (formula (\ref{eq: general
form of f: Cblc(n-1) to Cblc(n-1)})) and part \textup{(b)}, such a
map is of the form ${Q^\circ}\mapsto e^{\chi({Q^\circ})}{Q^\circ}$
with
$\chi\in\mathcal{O}^{\mathcal{S}}_{\hol}(\mathcal{C}_{\blc}^{n-1})$.
The other two assertions of part \textup{(c)} are clear.
\vskip3pt

\textup{(d)} A map as in  (\ref{eq: general form of f: Cblc(n-1)
to Cblc(n-1)}) is regular if and only if $\chi=\const$, i.e.,
%&
$$
f({Q^\circ})=s D_n^m({Q^\circ})\cdot {Q^\circ} \ \ \textup{for
all} \ \ {Q^\circ}\in\mathcal{C}^{n-1}_{\blc}\,, \ \
\textup{where} \ \ s\in\mathbb{C}^* \ \ \textup{and} \ \
m\in\mathbb{Z}\,.
$$
%&
It is a biregular automorphism of $\mathcal{C}^{n-1}_{\blc}$ if
and only if $m=0$. Hence
$\Aut\mathcal{C}^{n-1}_{\blc}\cong\mathbb{C}^*$.
\end{proof}

\begin{rem}[{\em Dimension of the image}\rm]
\label{Rm: dimension of image} In what follows we assume that $n>4$.
According to \cite[Theorem 14]{Lin11}, for $X=\mathbb{C}$ or $\mathbb{P}^1$
and any non-Abelian holomorphic endomorphism $F$ of $\mathcal{C}^n(X)$ 
we have $\dim_{\mathbb{C}}F(\mathcal{C}^n(X)) \ge{n-\dim_{\mathbb{C}}(\Aut X)+1}$. 
Moreover, by \cite[Remark 7]{Lin11} or theorems \ref{Thm: Non-Abelian endomorphisms of Cblc(n-1)} 
and \ref{Thm: properties of C*-tame maps}\,(a) above, for $X=\mathbb{C}$ the composition $\pi\circ{F}$ 
of any non-Abelian holomorphic endomorphism $F$ of $\mathcal{C}^n$ with the projection 
$\pi\colon\mathcal{C}^n\to\mathcal{C}^{n-1}_{\blc}$ is surjective, so that 
$\dim_{\mathbb{C}}{F}(\mathcal{C}^n)\ge{n-1}$. Clearly, the latter  bound  cannot be improved. 
Seemingly, for $X=\mathbb{P}^1$ no example of $F$ with 
$\dim_{\mathbb{C}}F(\mathcal{C}^n(\mathbb{P}^1))<n$
is known. Zinde (\cite{Zin78}) proved that  for $X=\mathbb{C}^*$  any non-Abelian 
holomorphic endomorphism of $\mathcal{C}^n(\mathbb{C}^*)$ is surjective.
\end{rem}
%%%%%%%%%%%%%%%%%%%%%%%%%%%%%%%%%%%%%%%%%%%%%%%%%%%%%%%%%
%%%%%%%%%%%%%%%%%%%%%%%%%%%%%%%%%%%%%%%%%%%%%%%%%%%%%%%%%%%%%%%%

\end{document}

%% file: cd.tex
\catcode`\@=11

 \def\AMSTeXfeatures{\Plainheads
   \let\current@vert=\AMS@vert}

 \def\Plainheads{\sh@ftdiam=0.05em
   \getlabeldims
   \let\vshaftfill=\plnvsolidfill
   \let\hshaftfill=\plnhsolidfill
   \let\th@rhead=\plnrhead
   \let\th@lhead=\plnlhead
   \let\th@dnhead=\plndnhead
   \let\th@uphead=\plnuphead}

 \def\glet{\global\let}

 \def\LaTeXfeatures{\catcode`\@=11
   \ifx\@clnwd\undefined \nol@g
      \input ltxcode.tex \dol@g \fi
   \ltxheads \let\current@vert=\new@vert
   \providelto \catcode`\@=\active}

 \def\nol@g{\def\wlog{\edef\garbage}}
 \def\dol@g{\let\wlog=\wl@g} \let\wl@g=\wlog
 \nol@g % Silences allocations with \newbox, etc.

 \newbox\ltobox
 \def\providelto{{\setbox\z@=
   \hbox{$\to$}\minharrlen=\wd\z@
   \global\setbox\ltobox=\hbox{$\activeat>>>$}}
   \def\lto{\mathrel{\copy\ltobox}}}

 \def\ltxheads{\sh@ftdiam=\@wholewidth
   \getlabeldims
   \let\vshaftfill= \ltxvsolidfill
   \let\hshaftfill=\ltxhsolidfill
   \let\th@rhead=\ltxrhead
   \let\th@lhead=\ltxlhead
   \let\th@dnhead=\ltxdnhead
   \let\th@uphead=\ltxuphead}
 {\catcode`\@=\active
   \gdef@#1{\csname #1\string@at\endcsname}
   \glet\activeat=@}
 \def\def@#1{\expandafter\def\csname #1@at\endcsname}

 \def@>#1>#2>{\@rrow R{#1}{#2}}
 \def@<#1<#2<{\@rrow L{#1}{#2}}
 \def@ V#1V#2V{\@rrow V{#1}{#2}}
 \def@ A#1A#2A{\@rrow A{#1}{#2}}
 \def@/#1/#2/#3/{\@rrow{#1}{#2}{#3}}
 \def@.{\ifodd\row\ifmmode\noharrow
     \else\leavevmode.\spacefactor3000 \fi
   \else\novarrow\fi}
 \def@={\ifodd\row\harrow\hequalfill{}{}%
   \else\varrow\vequalfill{}{}\fi}
 \def@:#1{\ifx=#1\harrow\deffill{}{}%
   \else\leavevmode\null:#1\fi}
 \def@|{\current@vert}
  \def\AMS@vert{\varrow\vequalfill{}{}}
  \def\new@vert#1|#2|{\ifodd\row
   \let\nextarrow\vertexvarrow
   \else\let\nextarrow\varrow\fi
   \nextarrow\vshaftfill{#1}{#2}}
 \def@-{\ifmmode\let\next\hl@ne
   \else\let\next\AMSatdash \fi \next}
  \def\hl@ne#1-#2-{\harrow\hshaftfill{#1}{#2}}
  \def\AMSatdash{\let\next\relax\leavevmode
    \def\next@{\ifx\next-%
      \def\next-{\futurelet\next\nextii@}%
     \else\def\next{\hbox{-}}\fi\next}%
    \def\nextii@{\ifx\next-\def\next-{\hbox{---}}%
      \else\def\next{\hbox{--}}\fi\next}%
    \futurelet\next\next@}
 \def@(#1){\tweenarrows{#1}}
 \def@[#1]{\setsp@n#1\relax\activeat}
 \def\fiberbox{\hbox{$\vcenter{\hr@le\hbox{\vr@le
   \kern1ex\vbox{\kern1.2ex}\vr@le}\hr@le}$}}
  \def\hr@le{\hrule height \sh@ftdiam}
  \def\vr@le{\vrule width \sh@ftdiam}
 \def@+#1+#2+#3+{\ifodd\row \harrow{#1}{#2}{#3}%
   \else \varrow{#1}{#2}{#3}\fi}

 \def\Dnarrfill{\vequalfill\Dnhe@d}
 \def\Uparrfill{\Uphe@d\vequalfill}
 \def\ontofill{\rtarrfill\kern-0.3em %2\he@dwd
   \th@rhead\kern 0.3em} %new def

 \def\rtarrfill{\hshaftfill\th@rhead}
 \def\ltarrfill{\th@lhead\hshaftfill}
 \def\dnarrfill{\vshaftfill\th@dnhead}
 \def\uparrfill{\th@uphead\vshaftfill}
 \def\hequalfill{\plnhfill=}
 \def\deffill{:\plnhfill=}
 \def\plnvextfill#1{\setbox\z@
   \hbox{\the\textfont3 #1}%
   \dimen@=\dp\z@\advance\dimen@\ht\z@
   \copy\z@ \kern-\dimen@ %-\dp\z@
   \cleaders\copy\z@ \vfill
   \kern-\dimen@ %-\dp\z@
   \box\z@}
 \def\plnhfill#1{$\m@th\mkern-1.5mu\mathord#1\mkern-6mu
    \cleaders\hbox{$\mkern-2mu\mathord#1\mkern-2mu$}\hfill
    \mkern-6mu\mathord#1\mkern-1.5mu$}
 \def\vequalfill{\plnvextfill{\char'167}}
 \def\plnvsolidfill{\plnvextfill{\char'077}}
 \def\plnhsolidfill{\plnhfill-}
 \def\ltxhsolidfill{\leaders\hrule height\topofshaft depth\botofshaft
   \hfill}
 \def\ltxvsolidfill{\leaders\vrule width\sh@ftdiam\vfill}
 \def\hdashfill{\hd@sh\wd@sh
   \xleaders \hbox{\wd@sh\hd@sh\wd@sh}\hfill
   \wd@sh\hd@sh}
 \def\vdashfill{\vd@sh\wd@sh
   \xleaders \vbox{\wd@sh\vd@sh\wd@sh}\vfill
   \wd@sh\vd@sh}
 \def\dashed{\ifinmeasureCD\else
    \ifodd\row\option{\let\hshaftfill=\hdashfill}%
   \else\option{\let\vshaftfill=\vdashfill}\fi\fi}
 \newdimen\CDstrutht  \newdimen\CDstrutdp
 \newdimen\CDstrutlen \CDstrutlen=\CDstrutht
   \advance\CDstrutlen by \CDstrutdp
 \def\CDstrut{\vrule
   height \ifnum\row=1 \z@\else\CDstrutht \fi
   depth \ifnum\row=\numrows \z@ \else\CDstrutdp \fi
   width\z@}
 \newdimen\CDarrsurr \CDarrsurr=0.375em
 \newdimen\CDdashlen
 \newdimen\CDvarrlen \CDvarrlen=1.5\baselineskip
 \newdimen\minharrlen %Used outside CD's
  \setbox\z@\hbox{$\longrightarrow$} \minharrlen=\wd\z@
 \newdimen\minCDharrlen \minCDharrlen=2.5em %825 2pc %2.5pc
\newdimen \minc@lwd
\def\findminc@lwd{\minc@lwd=2\CDarrsurr
  \advance\minc@lwd\minCDharrlen}
 \newdimen\sh@ftdiam

 \newdimen\labelsurr \labelsurr=1.25 em

\newcount\sp@ncnt \sp@ncnt=\@ne
\newcount\sp@ncnt@ \sp@ncnt@=\@ne
\newdimen\@rrwd \newdimen\@rrdp

 \def\adjustbot#1{\option{\advance\@rrdp#1\relax}}
\def\pushvertex#1{\global\p@shlen#1\relax
   \global\let\maybepush=\dopush}

 \newdimen\p@shlen \p@shlen=\z@

 \let\maybepush=\relax
 \def\dopush{\ifinmeasureCD %omitted by accident
   \advance\locdimen by -\p@shlen %AL
   \else\advance \@rrwd by -\p@shlen \fi %AL
   \global\let\maybepush=\relax \global\p@shlen=\z@\relax}
 \def\span@ne{\global\sp@ncnt=\@ne\relax}
 \def\setsp@n#1#2{\global\sp@ncnt=#1\relax
   \ifx\relax#2\relax\else\global\sp@ncnt@=#2\relax\fi}

 \def\plnrhead{\llap{$\rightarrow\mkern-1.5mu$}}
 \def\plnlhead{\rlap{$\mkern-1.5mu\leftarrow$}}

 \def\clap#1{\hbox to \z@{\hss #1\hss}}

 \def\plndnhead{\hbox{\the\textfont3 \char'171}}
 \def\plnuphead{\hbox{\the\textfont3 \char'170}}
 \def\Dnhe@d{\hbox{\the\textfont3 \char'177}}
 \def\Uphe@d{\hbox{\the\textfont3 \char'176}}

 \def\ltxrhead{\raise\@xisheight
   \llap{\smash{\@linefnt\@getrarrow(1,0)}}}
 \def\ltxlhead{\raise\@xisheight
   \rlap{\@linefnt\@getlarrow(-1,0)}}
 \def\ltxuphead{\setbox\z@=\rlap{%
   \kern\@halfwidth\@linefnt\char'66}%
   \copy\z@\kern-\ht\z@}
 \def\ltxdnhead{\setbox\z@=\rlap{%
   \kern\@halfwidth\@linefnt\char'77}%
   \ht\z@=\z@\box\z@}

 \def\wd@sh{\kern0.5\CDdashlen}
 \def\hd@sh{\vrule height\topofshaft depth\botofshaft
    width\CDdashlen}
 \def\vd@sh{\hrule height\CDdashlen
   depth\z@ width\sh@ftdiam}
\def\xylist{14{3434}13{2414}12{1723}%
  23{1413}34{1153}11{0867}43{0707}%
  32{0580}21{0414}31{0291}41{0}}
\newcount\tgtcnt@
\def\find@xyargs{\dimen@=\@rrdp
  \advance\dimen@ by \CDstrutlen
  \tgtcnt@=\dimen@ \dimen@=\@rrwd %\relax
  \divide\dimen@ by \@m %\relax
  \divide \tgtcnt@ by \dimen@ %\relax
  \expandafter\testxy\xylist\relax
  \unitlength=\@xarg\@rrdp
  \divide\unitlength by\@yarg\relax}
\def\testxy#1#2#3{\ifnum\tgtcnt@>#3
    \@xarg=#1\relax \@yarg=#2\relax
    \let\next=\ignorerest
  \else\let\next\testxy\fi\next}
\def\ignorerest#1\relax{\relax}

\let\scalefactor=\@ne
\def\SWarrow{\find@xyargs\vector
  (-\@xarg,-\@yarg)\scalefactor\hskip-\wd\@linechar}
\def\NWarrow{\find@xyargs\vector
  (-\@xarg,\@yarg)\scalefactor\hskip-\wd\@linechar}
\def\NEarrow{\find@xyargs\vector
  (\@xarg,\@yarg)\scalefactor}
\def\SEarrow{\find@xyargs\vector
  (\@xarg,-\@yarg)\scalefactor}
\def\rightupline{\find@xyargs\@linelen=\scalefactor
     \unitlength\@sline}
\def\rightdownline{\find@xyargs\@yarg=-\@yarg\relax
     \@linelen=\scalefactor\unitlength\@sline}

\def\Sim{\ifodd\row\setbox\z@=\hbox{$\sim$}\dimen@=\ht\z@
 \advance\dimen@ by -\@xisheight
  \vbox{\box\z@\kern-\@xisheight\kern\dimen@}%
  \else\hbox{$\wr$}\fi}

\def\harrow#1#2#3{\inmeasureCDtrue\findminarrwd
  {#2}{#3}{\sp@ncnt\minharrlen}\inmeasureCDfalse\span@ne
  \mathrel{\hbox{\options\hplace{#1}\ulabel{#2}\dlabel{#3}}}}

\def\noharrow{\harrow\hfill{}{}}
\def\vertexvarrow#1#2#3{\findarrdp \@rrwd=\z@ \setsp@n\@ne\@ne
  \vbox to \z@{\kern-1.2\CDstrutht
  \rlap{\options\vplace{#1}\llabel{#2}\rlabel{#3}}\vss}}

\newif\ifinmeasureCD
\def\measurelabel#1{\setbox\z@
  \hbox{$\scriptstyle#1\kern\labelsurr$}%
  \ifdim\wd\z@>\@rrwd \@rrwd=\wd\z@\fi}
\def\findminarrwd#1#2#3{\@rrwd=#3\relax
   \measurelabel{#1}\measurelabel{#2}}
\def\findCDarrwd#1#2{\@rrwd=\minCDharrlen
   \measurelabel{#1}\measurelabel{#2}%
  }

\newcount\row \row=\@ne \newcount\col \col=\@ne %96 initialized
 \newcount\numrows
\numrows=\@ne
 \newcount\numcols
\newcount\arrspan \newdimen\vrtxhalfwd  \newbox\tempbox

\def\DANABUG{\advance\col by \@ne
 \@rrwd=\minCDharrlen
  \advance\@rrwd by \vrtxhalfwd
  \advance\@rrwd by \CDarrsurr
  \ifnum\col>\numcols \numcols=\col
     \newlocdimen{col\the\col}\locdimen=\@rrwd %AL
  \else \ifdim\@rrwd>\c@l \c@l=\@rrwd\fi\fi}

\def\drop#1\\{%\noharrow %caused by DANABUG
  \findvrtxhalfsum\DANABUG\advance\row by 2 \measureinit}

\def\measureinit{\col=\@ne \vrtxhalfwd=-\CDarrsurr\arrspan=\@ne\@rrwd=\z@
   \setbox\tempbox=\hbox\bgroup$}
\def\measure{%CR \bgroup
  \let\harrow\measureCDarrow
  \let\CDCR=\measureCR %CR
   \findminc@lwd
  \inmeasureCDtrue
  \row=\@ne \numcols=\z@ \measureinit}

\def\endmeasure{\findvrtxhalfsum\DANABUG
  \numrows=\row %CR \egroup
  \inmeasureCDfalse}

\def\newlocdimen#1{\advance\dimenc@unt by \@ne
  \ifnum\dimenc@unt<\insc@unt
     \else\errmessage{No room for the CD}\fi
  \dimendef\locdimen=\dimenc@unt
  \expandafter\dimendef\csname#1\endcsname=\dimenc@unt}

 \def\r@wc@l{\csname row\the\row col\the\col\endcsname}
 \def\c@l{\csname col\the\col\endcsname}

 \def\findvrtxhalfsum{$\egroup
  \newlocdimen{row\the\row col\the\col}%%AL
  \locdimen=\vrtxhalfwd %AL
  \vrtxhalfwd=0.5\wd\tempbox %\maybes@ve %8231
  \advance\vrtxhalfwd by \CDarrsurr
  \advance\locdimen by \vrtxhalfwd %AL
  \advance\@rrwd by \locdimen %AL
  \maybepush
  \divide\@rrwd by \arrspan\relax
  \ifdim\@rrwd<\minc@lwd
    \ifnum\col>\@ne \@rrwd=\minc@lwd\fi \fi
  \loop %94 \edef\c@l{\csname col\the\col\endcsname}
    \ifnum\col>\numcols \numcols=\col
       \newlocdimen{col\the\col}% %AL
       \locdimen=\@rrwd %AL
    \else \ifdim\@rrwd>\c@l \c@l=\@rrwd\fi \fi
   \ifnum\arrspan>\@ne
      \advance\arrspan by -1 \advance\col by \@ne
  \repeat }

 \def\measureCDarrow#1#2#3{\findvrtxhalfsum
   \arrspan=\sp@ncnt\relax\global\sp@ncnt=1\relax
   \advance\col by \@ne
   \findCDarrwd{#2}{#3}%
   \setbox\tempbox=\hbox\bgroup$}

 \newcount\dr@tn \dr@tn=\z@
 \def\locate#1:#2{\ifinmeasureCD\else
   \count@=-#1
   \multiply\count@ by 2
   \advance\count@ by #2
   \dimen@=\count@\@rrwd
   \ifnum\dr@tn=\@ne\relax \else\dimen@=-\dimen@ \fi
   \dimen@i=\@rrdp
   \ifnum\dr@tn>\z@\advance\dimen@i by \CDstrutlen \fi
   \dimen@i=\count@\dimen@i
   \count@=#2 \multiply\count@ by 2
   \divide\dimen@ by \count@
   \divide\dimen@i by \count@
   \lift\dimen@i\nudge\dimen@\fi}
\def\betweenCDrows{\advance\row by \@ne \col=\@ne
\options}

\def\hbegin{\hbox\bgroup\kern\c@l \kern-\r@wc@l$}
\def\hend{$\glet\maybepush\relax \CDstrut\egroup}
\def\vbegin{\setbox\tempbox=\hbox\bgroup$}
\def\vend{$\egroup\ht\tempbox=\z@\dp\tempbox\CDvarrlen
  \box\tempbox}
\def\setCD{\let\harrow=\setCDarrow
  \let\CDCR=\setCR %CR
  \row=\@ne \col=\@ne \hbegin}
\let\endsetCD=\hend %AT (Assume CD ends with hmaterial)

\def\findarrwd{\@rrwd=\z@ \count@=\col \advance\count@ by\sp@ncnt
  \loop\ifnum\count@>\col \advance\count@ by -1
      \advance\@rrwd by\csname col\the\count@\endcsname\repeat}
\def\setCDarrow#1#2#3{\kern\CDarrsurr\advance\col by \@ne
  \findarrwd \advance\@rrwd by -\r@wc@l
  \@rrdp=\z@ %&0211 (It might be used by \locate).
  \maybepush
  \advance\col by -\@ne \advance\col by \sp@ncnt \span@ne
  \hbox to \@rrwd{\options
   \@rrwd=\scalefactor\@rrwd\hss
   \hplace{#1}\ulabel{#2}\dlabel{#3}\hss}%
   \kern\CDarrsurr}

\newdimen\labspacei %96 use subscript min of TeXbook 13a, p.444
\newdimen\labspaceii %96 Note many letters stick down below their boxes.

\newdimen\@xisheight
  \@xisheight=\the\fontdimen22\textfont2
\newdimen\labelskip
  \labelskip=\the\fontdimen10\textfont3 %2pt
\newdimen\topofshaft
\newdimen\botofshaft
\newdimen\botofulabel
\newdimen\topofdlabel
\def\getlabeldims{
  \topofshaft=0.5\sh@ftdiam
  \botofshaft=\topofshaft
  \advance\topofshaft by \@xisheight
  \advance\botofshaft by -\@xisheight
  \botofulabel=\topofshaft
  \advance\botofulabel by \labelskip
  \topofdlabel=\botofshaft
  \advance\topofdlabel by \labelskip}

\def\ulabel{\ifnum\row=\@ne\let\next\ulabeli
   \else\let\next\ulabellap\fi\next}
\def\ulabeli#1{\vbox{
  \clap{\kern-\@rrwd$\scriptstyle#1$}%
  \kern\botofulabel}\maybeoffset}
\def\ulabellap#1{\vbox to \z@{\vss
  \clap{\kern-\@rrwd$\scriptstyle#1$}%
  \kern\botofulabel}\maybeoffset}
\def\dlabel{\ifnum\row=\numrows\let\next\dlabeli
   \else\let\next\dlabellap\fi\next}
\def\dlabeli#1{\vtop{\kern\topofdlabel
  \clap{\kern-\@rrwd$\scriptstyle#1$}%
  }\maybeoffset}
\def\dlabellap#1{\vbox to \z@{\kern\topofdlabel
  \clap{\kern-\@rrwd$\scriptstyle#1$}%
  \vss}\maybeoffset}
\def\rlabel#1{\vbox to \z@{\vss
  \rlap{\kern\labelskip$\scriptstyle#1$}%
  \vss\kern-\@rrdp}\maybeoffset}
\def\llabel#1{\vbox to \z@{\vss
  \llap{$\scriptstyle#1$\kern\labelskip}%
  \vss\kern-\@rrdp}\maybeoffset}
\def\swlabel#1{\vtop{\kern0.5\@rrdp
  \llap{$\scriptstyle#1$\kern\labelskip\kern-0.5\@rrwd}
  }\maybeoffset}
\def\nwlabel#1{\vbox{
  \llap{$\scriptstyle#1$\kern\labelskip\kern-0.5\@rrwd}%
  \kern-0.5\@rrdp}\maybeoffset}
\def\selabel#1{\vtop{\kern0.5\@rrdp
  \rlap{\kern0.5\@rrwd\kern\labelskip$\scriptstyle#1$}%
  }\maybeoffset}
\def\nelabel#1{\vbox{
  \rlap{\kern0.5\@rrwd\kern\labelskip$\scriptstyle#1$}%
  \kern-0.5\@rrdp}\maybeoffset}
\def\cplace#1{\vbox to \z@{\vss
  \clap{$#1$\kern-\@rrwd}%
  \kern-\@rrdp\vss}\maybeoffset}
\def\hplace#1{\hbox to \@rrwd{#1}\maybeoffset}
\def\vplace#1{\clap{\vbox to \z@{#1\kern-\@rrdp}}\maybeoffset}
\newdimen\nudgeamount \nudgeamount=\z@
\newdimen\liftamount \liftamount=\z@
\let\maybeoffset\relax
\newbox\offsetbox \newdimen\lastheight
\def\dooffset{%assumes that \lastbox is a <box> set in horiz. mode
  \setbox\offsetbox=\lastbox \lastheight=\ht\offsetbox
  \setbox\offsetbox=\vbox{\kern-\liftamount\box\offsetbox}%
  \ht\offsetbox=\lastheight
  \kern\nudgeamount\box\offsetbox\kern-\nudgeamount
  \global\nudgeamount=\z@ \global\liftamount=\z@
  \glet\maybeoffset=\relax}
\def\nudge#1{\ifinmeasureCD\else
  \global\advance\nudgeamount#1\relax
  \global\let\maybeoffset\dooffset\fi}
\def\lift#1{\ifinmeasureCD\else
  \global\advance\liftamount#1\relax
  \global\let\maybeoffset\dooffset\fi}

\def\findarrdp{\@rrdp=\CDvarrlen
  \ifnum\sp@ncnt@>1
    \advance\@rrdp by \CDstrutlen
    \multiply\@rrdp by \sp@ncnt@
    \advance\@rrdp by -\CDstrutlen \fi
 }

\def\varrow#1#2#3{\ifnum\sp@ncnt>\@ne
     \sp@ncnt@=\sp@ncnt\relax\fi
  \findarrdp \@rrwd=\z@ %&0211 It might be used by \locate
  \kern\c@l
   \hbox to \z@{\options
   \@rrdp=\scalefactor\@rrdp
    \hss\vplace{#1}\llabel{#2}\rlabel{#3}\hss}%
  \global\advance\col by \@ne \setsp@n\@ne\@ne
  }

\def\novarrow{\varrow\vfill{}{}}

\def\tweenarrows#1{\findarrwd \findarrdp \setsp@n\@ne\@ne
  \rlap{\options\cplace{#1}}}

\def\usarrow #1#2#3{\dr@tn=\@ne
  \findarrwd \findarrdp \setsp@n\@ne\@ne
  \rlap{\options\cplace{#1}\nwlabel{#2}\selabel{#3}}%
  \dr@tn=\z@}
\def\dsarrow #1#2#3{\dr@tn=\tw@
  \findarrwd \findarrdp \setsp@n\@ne\@ne
  \rlap{\options\cplace{#1}\swlabel{#2}\nelabel{#3}}%
  \dr@tn=\z@}
 \def\@rrow#1{\csname #1@rrow\endcsname}
 \def\R@rrow{\harrow \rtarrfill}
 \def\L@rrow{\harrow \ltarrfill}
 \def\V@rrow{\varrow \dnarrfill}
 \def\A@rrow{\varrow \uparrfill}
 \def\SE@rrow{\dsarrow \SEarrow}
 \def\NW@rrow{\dsarrow \NWarrow}
 \def\SW@rrow{\usarrow \SWarrow}
 \def\NE@rrow{\usarrow \NEarrow}
 \def\DS@rrow{\dsarrow \dnslope}
 \def\US@rrow{\usarrow \upslope}
 \def\upslope{\find@xyargs
       \@linelen=\unitlength\@sline}
 \def\dnslope{\find@xyargs\@yarg=-\@yarg\relax
       \@linelen=\unitlength\@sline}

\newtoks\optionlist
\optionlist={}
\let\options\relax
\def\dooptions{\the\optionlist\global\optionlist={}%
  \glet\options=\relax}
\def\option#1{\ifinmeasureCD\else
  \glet\options=\dooptions
  \global\optionlist=\expandafter{\the\optionlist\relax#1}\fi}
\def\wider#1{\ifinmeasureCD\else
   \option{\advance\@rrwd by #1}\fi}
\def\deeper#1{\ifinmeasureCD\else
   \option{\advance\@rrdp by #1}\fi}
{\def\\{\global\let\sptoken= }\\ }%now \sptoken is a spacetoken

\def\CR{\futurelet\nexttok\testCR}
\def\testCR{\ifx\nexttok\sptoken
   \let\next\eatspaceCR\else\let\next\CDCR\fi\next}
\def\eatspaceCR#1 {\CR}
\def\measureCR{\ifx\nexttok\endmeasure\let\nextCR\relax
    \else\let\nextCR\drop\fi\nextCR}
\def\setCR{\ifodd\row
  \ifx\nexttok\endsetCD\else\hend\betweenCDrows\vbegin\fi
  \else\vend\betweenCDrows\hbegin\fi}

\countdef\dimenc@unt=11
\def\CD#1\endCD{%CRAL
   \begingroup\let\\=\CR
  \m@th\offinterlineskip
   \measure#1\endmeasure\null\,\vcenter{\setCD#1\endsetCD}\,
   \endgroup
    }

\ifx\@clnwd\undefined \nol@g\else\catcode`\ =14\relax\fi
 \font\@linefnt=line10
 \newcount\@tempcnta
 \newcount\@tempcntb
 \newdimen\@tempdima
 \newdimen\@tempdimb
 \newdimen\@wholewidth
 \newdimen\@halfwidth
   \@wholewidth\fontdimen8\@linefnt \@halfwidth .5\@wholewidth
 \newdimen\unitlength
 \newcount\@xarg
 \newcount\@yarg
 \newcount\@yyarg
 \newbox\@linechar
 \newdimen\@linelen
 \newdimen\@clnwd
 \newdimen\@clnht
 \newif\if@negarg

 \def\@whilenoop#1{}

 \def\@whiledim#1\do #2{\ifdim #1\relax#2\@iwhiledim{#1\relax#2}\fi}

 \def\@iwhiledim#1{\ifdim #1\let\@nextwhile=\@iwhiledim
         \else\let\@nextwhile=\@whilenoop\fi\@nextwhile{#1}}

 \def\@sline{\ifnum\@xarg< 0 \@negargtrue \@xarg -\@xarg \@yyarg -\@yarg
   \else \@negargfalse \@yyarg \@yarg \fi
 \ifnum \@yyarg >0 \@tempcnta\@yyarg \else \@tempcnta -\@yyarg \fi
 \ifnum\@tempcnta>6 \@badlinearg\@tempcnta0 \fi
 \ifnum\@xarg>6 \@badlinearg\@xarg 1 \fi
 \setbox\@linechar\hbox{\@linefnt\@getlinechar(\@xarg,\@yyarg)}%
 \ifnum \@yarg >0 \let\@upordown\raise \@clnht\z@
    \else\let\@upordown\lower \@clnht \ht\@linechar\fi
 \@clnwd=\wd\@linechar
 \if@negarg \hskip -\wd\@linechar \def\@tempa{\hskip -2\wd\@linechar}\else
      \let\@tempa\relax \fi
 \@whiledim \@clnwd <\@linelen \do
   {\@upordown\@clnht\copy\@linechar
    \@tempa
    \advance\@clnht \ht\@linechar
    \advance\@clnwd \wd\@linechar}%
 \advance\@clnht -\ht\@linechar
 \advance\@clnwd -\wd\@linechar
 \@tempdima\@linelen\advance\@tempdima -\@clnwd
 \@tempdimb\@tempdima\advance\@tempdimb -\wd\@linechar
 \if@negarg \hskip -\@tempdimb \else \hskip \@tempdimb \fi
 \multiply\@tempdima \@m
 \@tempcnta \@tempdima \@tempdima \wd\@linechar \divide\@tempcnta \@tempdima
 \@tempdima \ht\@linechar \multiply\@tempdima \@tempcnta
 \divide\@tempdima \@m
 \advance\@clnht \@tempdima
 \ifdim \@linelen <\wd\@linechar
    \hskip \wd\@linechar
   \else\@upordown\@clnht\copy\@linechar\fi}

 \def\@getlinechar(#1,#2){\@tempcnta#1\relax\multiply\@tempcnta 8
 \advance\@tempcnta -9 \ifnum #2>0 \advance\@tempcnta #2\relax\else
 \advance\@tempcnta -#2\relax\advance\@tempcnta 64 \fi
 \char\@tempcnta}

 \def\vector(#1,#2)#3{\@xarg #1\relax \@yarg #2\relax
 \@tempcnta \ifnum\@xarg<0 -\@xarg\else\@xarg\fi
 \ifnum\@tempcnta<5\relax
 \@linelen=#3\unitlength
 \ifnum\@xarg =0 \@vvector
   \else \ifnum\@yarg =0 \@hvector \else \@svector\fi
 \fi
 \else\@badlinearg\fi}

 \def\@svector{\@sline
 \@tempcnta\@yarg \ifnum\@tempcnta <0 \@tempcnta=-\@tempcnta\fi
 \ifnum\@tempcnta <5
   \hskip -\wd\@linechar
   \@upordown\@clnht \hbox{\@linefnt  \if@negarg
   \@getlarrow(\@xarg,\@yyarg) \else \@getrarrow(\@xarg,\@yyarg) \fi}%
 \else\@badlinearg\fi}

 \def\@getlarrow(#1,#2){\ifnum #2 =\z@ \@tempcnta='33\else
 \@tempcnta=#1\relax\multiply\@tempcnta \sixt@@n \advance\@tempcnta
 -9 \@tempcntb=#2\relax\multiply\@tempcntb \tw@
 \ifnum \@tempcntb >0 \advance\@tempcnta \@tempcntb\relax
 \else\advance\@tempcnta -\@tempcntb\advance\@tempcnta 64
 \fi\fi\char\@tempcnta}

 \def\@getrarrow(#1,#2){\@tempcntb=#2\relax
 \ifnum\@tempcntb < 0 \@tempcntb=-\@tempcntb\relax\fi
 \ifcase \@tempcntb\relax \@tempcnta='55 \or
 \ifnum #1<3 \@tempcnta=#1\relax\multiply\@tempcnta
 24 \advance\@tempcnta -6 \else \ifnum #1=3 \@tempcnta=49
 \else\@tempcnta=58 \fi\fi\or
 \ifnum #1<3 \@tempcnta=#1\relax\multiply\@tempcnta
 24 \advance\@tempcnta -3 \else \@tempcnta=51\fi\or
 \@tempcnta=#1\relax\multiply\@tempcnta
 \sixt@@n \advance\@tempcnta -\tw@ \else
 \@tempcnta=#1\relax\multiply\@tempcnta
 \sixt@@n \advance\@tempcnta 7 \fi\ifnum #2<0 \advance\@tempcnta 64 \fi
 \char\@tempcnta}
\catcode`\ =10

\dol@g %925
\catcode`\@=\active
\LaTeXfeatures